\documentclass[10pt]{article}
\usepackage{amsmath, amssymb, amsfonts, amsxtra, latexsym, amscd, eucal, amsthm, graphicx, epsfig, color, enumerate}      
\usepackage{hyperref}   
\theoremstyle{theorem}
\newtheorem{Def}{Definition}[section]
\newtheorem{Prop}[Def]{Proposition}
\newtheorem{Lem}[Def]{Lemma}
\newtheorem{Thm}[Def]{Theorem}

\theoremstyle{definition}

\newtheorem{Rem}[Def]{Remark}

\newcommand{\bR}{\mathbb{R}}

\newcommand{\mf}{\mathcal{F}}

\newcommand{\pr}{\mathbb{P}}

\newcommand{\bE}{\mathbb{E}}
 
\newcommand{\br}{\mathbb{R}} 

\newcommand{\sm}{\sigma}

\newcommand{\Om}{\Omega}

\newcommand{\ellf}{\textbf{l}}

\usepackage[top=2cm, bottom=2cm, left=2cm, right=2cm]{geometry}

\allowdisplaybreaks

\begin{document}

\title{On the infinite time horizon approximation for L\'evy-driven McKean-Vlasov SDEs with non-globally Lipschitz continuous and super-linearly growth drift and diffusion coefficients}
\author{
Ngoc Khue Tran\footnote{Faculty of Mathematics and Informatics, Hanoi University of Science and Technology, 1 Dai Co Viet, Hai Ba Trung, Hanoi, Vietnam. Email: khue.tranngoc@hust.edu.vn} \quad 
Trung-Thuy Kieu\footnote{Hanoi National University of Education. Email: thuykt@hnue.edu.vn} \quad  Duc-Trong Luong\footnote{Hanoi National University of Education. Email: trongld@hnue.edu.vn} \quad  Hoang-Long Ngo\footnote{Corresponding author. Hanoi National University of Education. Email: ngolong@hnue.edu.vn}}  
\maketitle

\begin{abstract} 
	This paper studies the numerical approximation for McKean-Vlasov stochastic differential equations driven by L\'evy processes. We propose a tamed-adaptive Euler-Maruyama scheme and consider its strong convergence in both finite and infinite time horizons when applying for some classes of L\'evy-driven McKean-Vlasov stochastic differential equations with non-globally Lipschitz continuous and super-linearly growth drift and diffusion coefficients.
\end{abstract} 



\section{Introduction}
On a complete probability space $(\Om, \mf, \pr)$, we consider the $d$-dimensional process $X=(X_t)_{t \geq 0}$ solution to the following McKean-Vlasov stochastic differential equation (SDE) with jumps
\begin{equation} \label{eqn1}
dX_t=b(X_t,\mathcal{L}_{X_t})dt+\sigma(X_t,\mathcal{L}_{X_t})dW_t +  c\left(X_{t-},\mathcal{L}_{X_{t-}}\right)dZ_t, 
\end{equation}
for $t \geq 0$, where $X_0=x_0\in \mathbb{R}^d$ is a fixed initial value, $\mathcal{L}_{X_t}$ denotes the marginal law of the process $X$ at time $t$, $W=(W_t)_{t\geq 0}$ is a $d$-dimensional standard Brownian motion and $Z=(Z_t)_{t\geq 0}$ is a $d$-dimensional centered pure jump L\'evy process whose L\'evy measure $\nu$ satisfies $\int_{\mathbb{R}^d}(1\wedge \vert z\vert^2)\nu(dz)<+\infty$. Two processes $W$ and $Z$ are supposed to be independent.  The natural filtration $(\mf_t)_{t\geq 0}$ is generated by $W$ and $Z$. The L\'evy-It\^o decomposition of $Z$ is given by
$$
Z_t=\int_{0}^{t}\int_{\mathbb{R}_0^d}z(N(ds,dz)-\nu(dz)ds),
$$
for any $t\geq 0$, where $\mathbb{R}_0^d:=\mathbb{R}^d\setminus\{0\}$, and $N(dt,dz)$ is a Poisson random measure on the measurable space $(\mathbb{R}_{+}\times \mathbb{R}_0^d,\mathcal{B}(\mathbb{R}_{+}\times \mathbb{R}_0^d))$ that is associated with the jumps of the L\'evy process $Z$ with intensity measure $\nu(dz)dt$. That is, 
$$
N(dt,dz):=\sum_{0\leq s\leq t}{\bf 1}_{\{\Delta Z_s\neq 0\}}\delta_{(s,\Delta Z_s)}(ds,dz).
$$
Here, the jump size of $Z$ at instant $s$ is defined as $\Delta Z_s:=Z_s-Z_{s-}:=Z_s-\lim_{u \uparrow s}Z_{u}$ for any $s>0$, $\Delta Z_0:=0$, $\delta_{(s,z)}$ denotes the Dirac measure at the point $(s,z)\in\mathbb{R}_{+}\times \mathbb{R}_0^d$, and $\mathcal{B}(\mathbb{R}_{+}\times \mathbb{R}_0^d)$ denotes the Borel $\sigma$-algebra on $\mathbb{R}_{+}\times \mathbb{R}_0^d$. The compensated Poisson random measure associated with $N(dt,dz)$ is denoted by $\widetilde{N}(dt,dz):=N(dt,dz)-\nu(dz)dt$.

We denote by $\mathcal{P}(\mathbb{R}^d)$ the space of all probability measures defined on a measurable space $(\mathbb{R}^d,\mathcal{B}(\mathbb{R}^d))$, where $\mathcal{B}(\mathbb{R}^d)$ denotes the Borel $\sigma$-field over $\mathbb{R}^d$, and by
$$
\mathcal{P}_2(\mathbb{R}^d):=\left\{\mu\in\mathcal{P}(\mathbb{R}^d): \int_{\mathbb{R}^d}\vert x\vert^2\mu(dx)<\infty \right\}
$$
the subset of probability measures with the finite second moment. As metric on the space $\mathcal{P}_2(\mathbb{R}^d)$, we use the $\mathcal{L}_2$-Wasserstein distance. That is, for $\mu,\nu\in \mathcal{P}_2(\mathbb{R}^d)$, the $\mathcal{L}_2$-Wasserstein distance between $\mu$ and $\nu$ is defined as
$$
\mathcal{W}_2(\mu,\nu):=\inf_{\pi\in\mathcal{C}(\mu,\nu)}\left(\int_{\mathbb{R}^d\times \mathbb{R}^d}\left\vert x-y\right\vert^2\pi(dx,dy)\right)^{1/2},
$$
where $\mathcal{C}(\mu,\nu)$ denotes all the couplings of $\mu$ and $\nu$. That is, $\pi\in\mathcal{C}(\mu,\nu)$ if and only if $\pi(\cdot,\mathbb{R}^d)=\mu(\cdot)$ and $\pi(\mathbb{R}^d,\cdot)=\nu(\cdot)$.

The coefficients $b=(b_i)_{1\leq i\leq d}: \mathbb{R}^d \times \mathcal{P}_2(\mathbb{R}^d)\to \mathbb{R}^d$, $\sigma=(\sigma_{ij})_{1\leq i,j\leq d}: \mathbb{R}^d \times \mathcal{P}_2(\mathbb{R}^d)\to \mathbb{R}^d\otimes \mathbb{R}^d$ and $c=(c_{ij})_{1\leq i,j\leq d}: \mathbb{R}^d \times \mathcal{P}_2(\mathbb{R}^d) \to \mathbb{R}^d\otimes \mathbb{R}^d$ are measurable functions. The integral equation \eqref{eqn1} now writes as 
\begin{equation*} 
X_t=x_0+\int_0^t b(X_s,\mathcal{L}_{X_s})ds+\int_0^t \sigma(X_s,\mathcal{L}_{X_s})dW_s + \int_0^t\int_{\mathbb{R}_0^d} c\left(X_{s-},\mathcal{L}_{X_{s-}}\right)z\widetilde{N}(d s, d z),
\end{equation*}	
for any $t\geq 0$.

The McKean-Vlasov process was first studied by McKean in \cite{MK66} as a model for the Vlasov equation describing the time evolution of the distribution function of a plasma consisting of charged particles with long-range interaction. The process can be obtained as a limit of a mean-field system of interacting particles as the number of particles tends to infinity. The very first studies on the numerical approximation for McKean-Vlasov SDEs are the works of Ogawa \cite{O95}, Kohatsu-Higa and Ogawa \cite{KO97} and Bossy and Talay \cite{BT97}, where the authors considered the weak approximation of McKean-Vlasov SDEs with regular coefficients. However, the numerical approximation for McKean-Vlasov SDEs has only become active in the last decade.

Let $X^{(n)}_T $ be an approximation of $X_T$ which depends on the values of $W$ and $Z$ at fixed equidistance times $t_k = \frac{kT}{n}, k=0,1,\ldots, n$. Then under some regularity conditions on the coefficients $b, \sigma,$ and $c$, one may prove that  
\begin{equation} \label{def.rate}
\|X^{(n)}_T - X_T\|_{L^p} :=  \Big(\bE[| X^{(n)}_T - X_T|^p]\Big)^{1/p} \leq \frac{C(T)}{n^{\zeta_0}},
\end{equation}
for some positive constants $p$ and ${\zeta_0}$. In case the estimate \eqref{def.rate} holds, we say that the $X^{(n)}_T $ converges at the rate of order ${\zeta_0}$ in $L^p$-norm.

It is now well-known that for SDEs with super-linear growth coefficients, the Euler-Maruyama scheme may not converge in the $L^p$-norm  (see \cite{HJK11}). Therefore, the numerical approximation for SDEs with super-linear growth coefficients has attracted lots of attention recently. New approximation schemes have been introduced to solve the problem, such as the tamed Euler-Maruyama scheme (\cite{HJK12, S1, S2, KS2}), the truncated Euler-Maruyama scheme (\cite{Mao2}), the implicit Euler-Maruyama scheme (\cite{MS}), the adaptive Euler-Maruyama scheme (\cite{FG}). In \cite{dRES22, KN21, KNRS22,  LWW} the tamed Euler-Maruyama scheme has been developed to approximate McKean-Vlasov SDEs with super-linear growth coefficients.  In \cite{RW22}, the authors introduced several adaptive Euler-Maruyama and Milstein schemes and studied their strong rate of convergence for McKean-Vlasov SDEs with super-linear drift. 
In \cite{HKN22}, the authors introduced a multilevel Picard approximation, which has a low computational cost,  for McKean-Vlasov SDEs with additive noise. In \cite{CG22, dRES22}, the authors introduced the implicit Euler-Maruyama scheme and studied its convergence in $L^p$-norm for McKean-Vlasov SDEs with drifts of super-linear growth. 

The McKean-Vlasov SDEs with jumps were studied by many authors with applications in many domains (see  \cite{G92, E22, ELL22, AP21,  FP22} and the reference therein). In \cite{NBKGR20}, the authors considered McKean-Vlasov SDEs driven by infinite activity L\'evy processes with super-linear growth coefficients. They proved the existence and uniqueness of the solution and proposed a tamed Euler-Maruyama approximation for the associated interacting particle system and proved that the rate of its convergence in $L^p$-norm is arbitrarily close to 0.5. 

In some applications, it is necessary to approximate $X_T$ for $T$ large, e.g., to approximate the invariant distribution of process $X$ (see Section 3 in \cite{FG}). However, since the proof of convergence of many schemes often involves Gronwall's inequality, the quantity $C(T)$ may grow exponentially fast to $T$.  Recently, there have been some attempts to introduce numerical schemes where $C(T)$ does not depend on $T$.  The paper \cite{FG} proposed an adaptive Euler-Maruyama scheme for SDEs where each stepsize is adjusted to the size of the current values of $X$, and showed its strong convergence in the interval $[0,\infty)$ when applying for SDEs whose coefficients $b$ and $\sigma$ satisfy the contractive Lipschitz condition (Assumption 9 in \cite{FG}), $b$ is polynomial growth Lipschitz continuous, and $\sigma$ is bounded and globally Lipschitz continuous. The paper \cite{KLN22} 
introduced a tamed-adaptive Euler-Maruyama scheme and considered its rate of convergence in $L^1$-norm when applying for one-dimensional SDEs whose diffusion coefficient $\sigma$ is of super-linearly growth. In \cite{KLNT22}, the result of \cite{KLN22} was generalized for one-dimensional SDEs with jumps. 

This paper aims to generalize the result of the papers \cite{FG, KLN22, KLNT22} for multi-dimensional McKean-Vlasov SDEs with jumps.  In particular, we propose a  tamed-adaptive Euler-Maruyama approximation scheme for the   L\'evy-driven SDEs \eqref{eqn1} where \textcolor{black}{$\sigma$ and $b$ are non-globally Lipschitz continuous and of super-linearly growth}, and  $c$ is Lipschitz continuous.  We will study the strong convergence of the scheme in both finite and infinite time intervals. Note that in \cite{RW22} the authors only considered the strong convergence of their adaptive scheme in a fixed time interval. In \cite{NBKGR20}, the tamed Euler-Maruyama scheme is also proven to converge in a fixed time interval when applying for McKean-Vlasov SDEs with jumps. To the best of our knowledge, our tamed-adaptive Euler-Maruyama scheme is the first approximation method for L\'evy driven McKean-Vlasov SDEs that could be shown to converge in an infinite time horizon.

	We denote the vector Euclidean norm of $v$ by $\vert v\vert:=(\vert v_1\vert^2+\vert v_2\vert^2+\ldots+\vert v_d\vert^2)^{\frac{1}{2}}$, the inner product of vectors $v$ and $w$ by $\langle v,w\rangle:=v_1w_1+v_2w_2+\ldots+v_dw_d$ for any $v, w\in \mathbb{R}^d$, the Frobenius matrix norm by $\vert A\vert:=\sqrt{\sum_{i,j=1}^{d}A^2_{ij}}$ for all $A\in \mathbb{R}^{d \times d}$, the integer part of $a$ by $[a]$, and  the transpose of a vector or matrix $A$ by $A^\mathsf{T}$. Moreover, the binomial coefficients are denoted by ${n \choose k}:=\frac{n!}{k!(n-k)!}$. In all that follows, we denote by $C$ positive constants whose value may change from one line to the next.

The rest of this paper is structured as follows. In Section \ref{Sec:2}, we introduce assumptions on the coefficients of equation \eqref{eqn1} and show some moment estimates under these assumptions. In Section \ref{Sec:3}, we prove the propagation of chaos. In Section \ref{sec:main}, we introduce the tamed-adaptive Euler-Maruyama scheme and prove that it is well-defined and converges in $L^2$-norm.  \textcolor{black}{Section \ref{sec:nume}  presents a numerical study for the tamed-adaptive Euler-Maruyama scheme.}

\section{Model assumptions and moment estimates} \label{Sec:2}

	Assume that the drift, diffusion and jump coefficients $b, \sigma, c$ and the L\'evy measure $\nu$ of equation \eqref{eqn1} satisfy the following conditions:
	
\begin{enumerate}
		\item[\bf A1.]  There exists a positive constant $L$ such that
		\begin{align*}
		2\left\langle x,b(x,\mu)\right\rangle + \left\vert\sigma(x,\mu)\right\vert^2 +|c(x,\mu)|^2 \int_{\bR_{0}^d}|z|^{2} \nu(d z) \leq L \left(1+|x|^2+\mathcal{W}_2^2(\mu,\delta_0)\right),
		\end{align*}
		for any $x\in \br^d$ and $\mu\in \mathcal{P}_2(\mathbb{R}^d)$.		
		\item[\bf A2.]
		There exist constants $\kappa_1>0, \kappa_2>0$, $L_1\in\mathbb{R}$ and $L_2\geq 0$ such that
		\begin{align*}
		&2\left\langle x-\overline{x},b(x,\mu)-b(\overline{x},\overline{\mu})\right\rangle +\kappa_1 \left\vert\sigma(x,\mu)-\sigma(\overline{x},\overline{\mu})\right\vert^2 +\kappa_2|c(x,\mu)-c(\overline{x},\overline{\mu})|^2 \int_{\bR_{0}^d}|z|^{2} \nu(d z) \\
		&\leq L_1 |x-\overline{x}|^2+L_2\mathcal{W}_2^2(\mu,\overline{\mu}),
		\end{align*}
		for any $x,\overline{x}\in \br^d$ and $\mu, \overline{\mu}\in \mathcal{P}_2(\mathbb{R}^d)$.
		\item[\bf A3.]  $b(x,\mu)$ is a continuous function of $x\in \br^d$ and $\mu\in \mathcal{P}_2(\mathbb{R}^d)$.
		\item[\bf A4.]  There exist constants $L>0$ and $\ell \geq 0 $ such that
		\begin{align*}
		\left\vert b(x,\mu)-b(\overline{x},\overline{\mu})\right\vert\leq L\left(1+\vert x\vert^{\ell}+\vert \overline{x}\vert^{\ell}\right) \left(\vert x-\overline{x}\vert+\mathcal{W}_2(\mu,\overline{\mu})\right),
		\end{align*}
		for any $x,\overline{x}\in \br^d$ and $\mu, \overline{\mu}\in \mathcal{P}_2(\mathbb{R}^d)$.
		\item[\bf A5.] There exists an even integer $p_0 \in [2, +\infty)$ such that  $\int_{\vert z\vert >1} \vert z\vert^{p_0}\nu(dz)<\infty$  and $\int_{0< \vert z\vert \le 1} \vert z\vert \nu(dz)<\infty$. 
		\item[\bf A6.] There exists a positive constant $L_3$ such that
		\begin{equation*}
		|c(x,\mu)| \leq L_3 \left(1+|x|+\mathcal{W}_2(\mu,\delta_0)\right), 
		\end{equation*}
		for any $\;x \in \br^d$ and $\mu\in \mathcal{P}_2(\mathbb{R}^d)$.	
		\item[\bf A7.]  For the even integer $p_0 \in [2, +\infty)$ given in \textbf{A5} and the positive constant $L_3$ given in \textbf{A6}, there exist constants  $\gamma_1 \in \mathbb R$, $\gamma_2\geq 0$ and $\eta\geq 0$ such that 		
		\begin{align*}
		\langle x,b(x,\mu) \rangle  + \dfrac{p_0-1}{2}  |\sm(x,\mu)|^2 +\dfrac{1}{2L_3}|c(x,\mu)|^2 \int_{\bR_{0}^d}|z|\left(\left(1+L_{3}|z|\right)^{p_0-1}-1\right) \nu(d z) \leq \gamma_1  |x|^2+\gamma_2\mathcal{W}_2^2(\mu,\delta_0)+\eta,
		\end{align*} 
		for any $x \in \br^d$ and $\mu\in \mathcal{P}_2(\mathbb{R}^d)$.
\end{enumerate}
\begin{Rem} \label{sigmac} (i) It follows from Condition \textbf{A4} that
	\begin{align*}
	\left\vert b(x,\mu)\right\vert\leq \left\vert b(0,\delta_0)\right\vert+L\left(1+\vert x\vert^{\ell}\right) \left(\vert x\vert+\mathcal{W}_2(\mu,\delta_0)\right),
	\end{align*}
	for any $x \in \bR^d$ and $\mu\in \mathcal{P}_2(\mathbb{R}^d)$.	
	
	(ii) Assume that Condition \textbf{A2} holds for $\kappa_1=\kappa_2=1+\varepsilon$, $L_1\in\mathbb{R}$, $L_2\geq 0$ with a constant $\varepsilon>0$ . This, combined with Condition \textbf{A4} and Cauchy's inequality, implies that 
	\begin{align*}
	&\left(1+\varepsilon\right) \left\vert\sigma(x,\mu)-\sigma(\overline{x},\overline{\mu})\right\vert^2 +\left(1+\varepsilon\right)|c(x,\mu)-c(\overline{x},\overline{\mu})|^2 \int_{\bR_{0}^d}|z|^{2} \nu(d z) \\
	&\leq 2\vert x-\overline{x}\vert \vert b(x,\mu)-b(\overline{x},\overline{\mu})\vert+ L_1 |x-\overline{x}|^2+L_2\mathcal{W}_2^2(\mu,\overline{\mu})\\
	&\leq L\left(1+\vert x\vert^{\ell}+\vert \overline{x}\vert^{\ell}\right) \left(2\vert x-\overline{x}\vert^2+2\vert x-\overline{x}\vert\mathcal{W}_2(\mu,\overline{\mu})\right)+ L_1 |x-\overline{x}|^2+L_2\mathcal{W}_2^2(\mu,\overline{\mu})\\
	&\leq L\left(1+\vert x\vert^{\ell}+\vert \overline{x}\vert^{\ell}\right) \left(2\vert x-\overline{x}\vert^2+\vert x-\overline{x}\vert^2+\mathcal{W}_2^2(\mu,\overline{\mu})\right)+ L_1 |x-\overline{x}|^2+L_2\mathcal{W}_2^2(\mu,\overline{\mu})\\
	&\leq L\widetilde{L}\left(1+\vert x\vert^{\ell}+\vert \overline{x}\vert^{\ell}\right)\left(\vert x-\overline{x}\vert^2+\mathcal{W}_2^2(\mu,\overline{\mu})\right),
	\end{align*}
	with $\widetilde{L}=\max\{3,L_1,L_2\}$ for any $x,\overline{x}\in \br^d$ and $\mu, \overline{\mu}\in \mathcal{P}_2(\mathbb{R}^d)$. This yields to
	\begin{align*}
	\left\vert\sigma(x,\mu)-\sigma(\overline{x},\overline{\mu})\right\vert^2+|c(x,\mu)-c(\overline{x},\overline{\mu})|^2 \int_{\bR_{0}^d}|z|^{2} \nu(d z)\leq \dfrac{L\widetilde{L}}{1+\varepsilon}\left(1+\vert x\vert^{\ell}+\vert \overline{x}\vert^{\ell}\right)\left(\vert x-\overline{x}\vert^2+\mathcal{W}_2^2(\mu,\overline{\mu})\right),
	\end{align*}
	for any $x,\overline{x}\in \br^d$ and $\mu, \overline{\mu}\in \mathcal{P}_2(\mathbb{R}^d)$.
	
	(iii) From (ii), we have that for any $x\in \br^d$ and $\mu\in \mathcal{P}_2(\mathbb{R}^d)$,
	\begin{align*}
	\left\vert\sigma(x,\mu)\right\vert^2&\leq 2\left\vert\sigma(x,\mu)-\sigma(0,\delta_0)\right\vert^2+2\left\vert\sigma(0,\delta_0)\right\vert^2\\
	&\leq 2\dfrac{L\widetilde{L}}{1+\varepsilon}\left(1+\vert x\vert^{\ell}\right)\left(\vert x\vert^2+\mathcal{W}_2^2(\mu,\delta_0)\right)+2\left\vert\sigma(0,\delta_0)\right\vert^2,
	\end{align*}
	and similarly,
	\begin{align*}
	|c(x,\mu)|^2 \int_{\bR_{0}^d}|z|^{2} \nu(d z)\leq 2\dfrac{L\widetilde{L}}{1+\varepsilon}\left(1+\vert x\vert^{\ell}\right)\left(\vert x\vert^2+\mathcal{W}_2^2(\mu,\delta_0)\right)+2\left\vert c(0,\delta_0)\right\vert^2\int_{\bR_{0}^d}|z|^{2} \nu(d z).
	\end{align*}
\end{Rem}
\begin{Rem} \label{c5}
	Observe that Condition \textbf{A7} yields to 
	\begin{align*}
	\langle x,b(x,\mu) \rangle  + \dfrac{p-1}{2}  |\sm(x,\mu)|^2 +\dfrac{1}{2L_3}|c(x,\mu)|^2 \int_{\bR_{0}^d}|z|\left(\left(1+L_{3}|z|\right)^{p-1}-1\right) \nu(d z) \leq \gamma_1  |x|^2+\gamma_2\mathcal{W}_2^2(\mu,\delta_0)+\eta,
	\end{align*} 
	for any $p\in[2,p_0]$, $x \in \br^d$ and $\mu\in \mathcal{P}_2(\mathbb{R}^d)$.
\end{Rem}

We first recall a result on the existence and uniqueness of the strong solution of McKean-Vlasov SDE with jumps \eqref{eqn1} which is taken from \cite{NBKGR20}.
\begin{Prop}\label{existenceunique} (\cite[Theorem 2.1]{NBKGR20})
Assume Conditions \textbf{A1}, \textbf{A3} and that Condition \textbf{A2} holds for $\kappa_1=\kappa_2=1, L_1= L_2>0$. Then, there exists a unique c\`adl\`ag process $X=(X_t)_{t \ge 0}$ taking values in $\mathbb{R}^d$ satisfying the McKean-Vlasov SDE with jumps \eqref{eqn1} such that
$$
\sup_{t\in[0,T]}\bE \left[\vert X_t\vert^2\right]\leq K,
$$
where $T>0$ is a fixed constant and $K:=K(\vert x_0\vert^2,d,L,L_1,T)$ is a positive constant.
\end{Prop}

We now recall the Burkholder-Davis-Gundy inequality for the compensated Poisson stochastic integral which will be useful in the sequel.
\begin{Lem} \label{BDGjump}(\cite[Theorem 4.4.23]{A09} and \cite[Proposition 2.2]{Zhu}) Let  $\mathcal{B}(\mathbb{R}_0^d)$ be the Borel $\sigma$-algebra of \ $\mathbb{R}_0^d$ and $\mathcal{P}$ be the progressive $\sigma$-algebra on $\mathbb{R}_+\times\Omega$. Let $g$ be a $\mathcal{P}\otimes \mathcal{B}(\mathbb{R}_0^d)$-measurable function satisfying that $\int_0^T\int_{\mathbb{R}_0^d}\vert g(s,z) \vert^2\nu(dz)ds<\infty$ $\pr$-a.s. for all $T\geq 0$. Then for any $p\geq 2$, there exists a positive constant $C_p$ such that
	\begin{align*}
		&\mathbb{E}\left[\sup_{t\in[0,T]}\left\vert \int_0^t\int_{\mathbb{R}_0^d} g(s,z) \widetilde{N}(d s, d z)\right\vert^p\right]\\
		&\leq C_p\left(\mathbb{E}\left[\left(\int_0^T\int_{\mathbb{R}_0^d} \vert g(s,z) \vert^2 \nu(dz)ds\right)^{\frac{p}{2}}\right]
		+\mathbb{E}\left[\int_0^T\int_{\mathbb{R}_0^d} \vert g(s,z) \vert^p \nu(dz)ds\right]\right).
	\end{align*}
	Furthermore, for any $1\leq p<2$, there exists a positive constant $C_p$ such that
	\begin{align*}
		\mathbb{E}\left[\sup_{t\in[0,T]}\left\vert \int_0^t\int_{\mathbb{R}_0^d} g(s,z) \widetilde{N}(d s, d z)\right\vert^p\right]\leq C_p\mathbb{E}\left[\left(\int_0^T\int_{\mathbb{R}_0^d} \vert g(s,z) \vert^2 \nu(dz)ds\right)^{\frac{p}{2}}\right].
	\end{align*}	
\end{Lem}
We next show the following moment estimates of the exact solution $X=(X_t)_{t \geq 0}$.   
\begin{Prop}\label{nghiem dung 1}
	Let $X=(X_t)_{t \geq 0}$ be a solution to equation \eqref{eqn1}. Assume Conditions \textbf{A6, A7}, and that $\sigma$ is bounded on $C\times\mathcal{P}_2(\mathbb{R}^d)$ for every compact subset $C$ of \; $\mathbb{R}^d$, and \textbf{A5} holds for $q=2p_0$.
	Then for any $p \in [2,p_0]$, there exists a positive constant $C_p$ such that for any $t \ge 0$
	\begin{equation} \label{proofmomentp}  \begin{split} 
		\bE \left[\vert X_t\vert^p\right] \le \begin{cases}  C_p (1+e^{\gamma pt}) & \text{ if \;\;} \gamma \not = 0,\\
			C_p (1+t)^{p/2} & \text{ if \;\;} \gamma  = 0, p=2 \text{\; or \;} \gamma  = 0, \gamma_2>0, p\in (2,p_0],\\ 
			C_p (1+t)^{p} & \text{ if \;\;} \gamma  = 0, \gamma_2=0, p\in (2,p_0],
			\end{cases} 
			\end{split} 
	\end{equation}
	where $\gamma=\gamma_1+\gamma_2$. 
\end{Prop}
Note that if $\gamma<0$, we have that $\sup_{t \geq 0} \bE \left[\vert X_t\vert^p\right] \le 2C_p$.
\begin{proof}	
\underline{\it Step 1:} We first show that for any even natural number $p\in [2,p_0]$ and $T>0$,
\begin{align}\label{momentp}
\sup_{t\in[0,T]}\bE \left[\left\vert X_t\right\vert^p\right] \leq C(T,p).
\end{align}
Note that \eqref{momentp} holds for $p=2$ due to Proposition \ref{existenceunique}. Next, we assume that
\eqref{momentp} is valid for any even natural number $q\in [2,p-2]$. That is,
\begin{align}\label{inductionq}
\sup_{t\in[0,T]}\bE \left[\left\vert X_t\right\vert^q\right] \leq C(T,q).
\end{align}
Now, for $\lambda \in \br$ and even natural number $p\in [2,p_0]$, applying It\^o's formula to $e^{-\lambda t } |X_t|^p$, we obtain that for any $t\geq 0$,
\begin{align}
&e^{ - \lambda t} |X_t|^p
=|x_0|^p+\int_0^t e^{-\lambda s}\bigg[-\lambda   \left|X_{s}\right|^{p}+p \left|X_{s}\right|^{p-2}  \left\langle X_{s}, b\left(X_{s},\mathcal{L}_{X_s}\right)\right\rangle+ \dfrac{p}{2}  \left|X_{s}\right|^{p-2}  \left|\sigma\left(X_{s},\mathcal{L}_{X_s}\right)\right|^{2}\notag\\
&\quad + \dfrac{p(p-2)}{2}  \left|X_{s}\right|^{p-4}  \left|X_{s}^\mathsf{T}  \sigma\left(X_{s},\mathcal{L}_{X_s}\right)\right|^{2} \bigg]ds +p\int_0^te^{-\lambda s} \left|X_{s}\right|^{p-2}\left\langle X_{s}, \sigma\left(X_{s},\mathcal{L}_{X_s}\right) d W_{s}\right\rangle\notag \\
&\quad  +\int_0^t\int_{\bR_{0}^{d}} e^{-\lambda s}  \left( \left|X_{s-}+c\left(X_{s-},\mathcal{L}_{X_{s-}}\right) z\right|^{p}-\left|X_{s-}\right|^{p}\right) \widetilde{N}(d s, d z) \notag \\
&\quad  + \int_0^t\int_{\bR_{0}^{d}} e^{- \lambda s}\left( \left|X_{s}+c\left(X_{s},\mathcal{L}_{X_s}\right) z \right|^{p}-\left|X_{s}\right|^{p}-p\left|X_{s}\right|^{p-2} \left\langle X_{s}, c\left(X_{s},\mathcal{L}_{X_s}\right) z\right\rangle \right) \nu(d z)ds. \label{ulxt}
\end{align}
In order to treat the last integral in \eqref{ulxt}, it suffices to use the binomial theorem to get that, for any $t\geq 0$,
\begin{align}
&\left|X_s+c\left(X_s,\mathcal{L}_{X_s}\right) z\right|^{p} 
= \left(\left|X_s+c\left(X_s,\mathcal{L}_{X_s}\right) z\right|^{2}\right)^{p/2} \notag\\
&=\left( |X_s|^2+|c\left(X_s,\mathcal{L}_{X_s}\right) z|^2 +2\left\langle X_s, c\left(X_s,\mathcal{L}_{X_s}\right) z \right\rangle \right)^{p/2} \notag\\
&=|X_s|^p + \dfrac{p}{2}|X_s|^{p-2} \left(|c\left(X_s,\mathcal{L}_{X_s}\right) z|^2+2\langle X_s, c\left(X_s,\mathcal{L}_{X_s}\right) z \rangle \right) \notag\\ 
&\quad+\sum_{i=2}^{p/2} {p/2 \choose i}|X_s|^{p-2i} \left(|c\left(X_s,\mathcal{L}_{X_s}\right) z|^2+2\left\langle X_s, c\left(X_s,\mathcal{L}_{X_s}\right) z\right\rangle \right)^i \notag\\
&= |X_s|^p +p |X_s|^{p-2} \left\langle X_s, c\left(X_s,\mathcal{L}_{X_s}\right) z \right\rangle + \dfrac{p}{2}|X_s|^{p-2} \left|c\left(X_s,\mathcal{L}_{X_s}\right) z\right|^2 \notag\\ &\quad+\sum_{i=2}^{p/2} {p/2 \choose i}|X_s|^{p-2i} \left(|c\left(X_s,\mathcal{L}_{X_s}\right) z|^2+2\left\langle X_s, c\left(X_s,\mathcal{L}_{X_s}\right) z\right\rangle \right)^i. \label{for1}
\end{align}
Next, by using the binomial theorem repeatedly, Condition \textbf{A6}, the equality $\mathcal{W}_2^2\left(\mathcal{L}_{X_s},\delta_0\right) = \bE[\vert X_s\vert^2]$, and the inequality $y\vert x\vert^{p-3}\leq \frac{1}{2}(\vert x\vert^{p-2}+y^2\vert x\vert^{p-4})$ valid for any $x\in \br^d, y>0$, $\sum_{j=0}^{i}{i \choose j} a^j=(1+a)^i$ and $\sum_{j=0}^{i}{i \choose j} j a^{j}=ia(1+a)^{i-1}$ valid for any $a\in \br$, we obtain that
\begin{align*}
& |X_s|^{p-2i} \left(|c\left(X_s,\mathcal{L}_{X_s}\right) z|^2+2\left\langle X_s, c\left(X_s,\mathcal{L}_{X_s}\right) z\right\rangle \right)^i =|X_s|^{p-2i} \sum_{j=0}^{i} {i \choose j} |c\left(X_s,\mathcal{L}_{X_s}\right) z|^{2i-2j}2^j \left( \langle X_s, c\left(X_s,\mathcal{L}_{X_s}\right) z\rangle\right)^j \\
&\le \sum_{j=0}^{i} {i \choose j} 2^j |X_s|^{p-2i+j}   |c\left(X_s,\mathcal{L}_{X_s}\right)|^{2i-j} |z|^{2i-j} \\
&= |c\left(X_s,\mathcal{L}_{X_s}\right)|^{2}\sum_{j=0}^{i} {i \choose j} 2^j |X_s|^{p-2i+j}   |c\left(X_s,\mathcal{L}_{X_s}\right)|^{2i-j-2} |z|^{2i-j}\\
&\le  |c\left(X_s,\mathcal{L}_{X_s}\right)|^{2}\sum_{j=0}^{i}  {i \choose j} 2^j |X_s|^{p-2i+j}    |z|^{2i-j} L_3^{2i-j-2} \left(1+|X_s|+\mathcal{W}_2(\mathcal{L}_{X_s},\delta_0)\right)^{2i-j-2}  \\
&=  |c\left(X_s,\mathcal{L}_{X_s}\right)|^{2}\sum_{j=0}^{i}  {i \choose j} 2^j |X_s|^{p-2i+j}    |z|^{2i-j} L_3^{2i-j-2} \left(1+|X_s|+\sqrt{\bE[\vert X_s\vert^2]}\right)^{2i-j-2}  \\
&= |c\left(X_s,\mathcal{L}_{X_s}\right)|^{2} \sum_{j=0}^{i}  {i \choose j} 2^j |X_s|^{p-2i+j}    |z|^{2i-j} L_3^{2i-j-2} \Bigg(|X_s|^{2i-j-2}+(2i-j-2)\left(1+\sqrt{\bE[\vert X_s\vert^2]}\right)|X_s|^{2i-j-3}\\
&\quad+\sum_{k=2}^{2i-j-2} {2i-j-2 \choose k}\left(1+\sqrt{\bE[\vert X_s\vert^2]}\right)^k|X_s|^{2i-j-2-k}\Bigg)   \\
&=  |c\left(X_s,\mathcal{L}_{X_s}\right)|^{2}\sum_{j=0}^{i}  {i \choose j} 2^j  |z|^{2i-j} L_3^{2i-j-2} \Bigg(|X_s|^{p-2}+(2i-j-2)\left(1+\sqrt{\bE[\vert X_s\vert^2]}\right)|X_s|^{p-3}\\
&\quad+\sum_{k=2}^{2i-j-2} {2i-j-2 \choose k}\left(1+\sqrt{\bE[\vert X_s\vert^2]}\right)^k|X_s|^{p-2-k}\Bigg)   \\
&\le |c\left(X_s,\mathcal{L}_{X_s}\right)|^{2}\sum_{j=0}^{i} {i \choose j} 2^j  |z|^{2i-j} L_3^{2i-j-2} \Bigg(|X_s|^{p-2}+\dfrac{(2i-j-2)}{2}\left(|X_s|^{p-2}+\left(1+\sqrt{\bE[\vert X_s\vert^2]}\right)^2|X_s|^{p-4}\right)\\
&\quad+\sum_{k=2}^{2i-j-2} {2i-j-2 \choose k}\left(1+\sqrt{\bE[\vert X_s\vert^2]}\right)^k|X_s|^{p-2-k}\Bigg)   \\
&= |c\left(X_s,\mathcal{L}_{X_s}\right)|^{2}|X_s|^{p-2}\sum_{j=0}^{i}  {i \choose j} 2^j  |z|^{2i-j} L_3^{2i-j-2} \left(i-\dfrac{j}{2}\right)  + |c\left(X_s,\mathcal{L}_{X_s}\right)|^{2}\sum_{j=0}^{i}  {i \choose j} 2^j  |z|^{2i-j} L_3^{2i-j-2} \\
&\quad \times \left( \left(i-\dfrac{j}{2}-1\right) \left(1+\sqrt{\bE[\vert X_s\vert^2]}\right)^2|X_s|^{p-4}+\sum_{k=2}^{2i-j-2} {2i-j-2 \choose k}\left(1+\sqrt{\bE[\vert X_s\vert^2]}\right)^k|X_s|^{p-2-k}\right)   \\
&= |c\left(X_s,\mathcal{L}_{X_s}\right)|^{2} |X_s|^{p-2} \left(\dfrac{i}{L_{3}^{2}}\left(L_{3}^{2}|z|^{2}+2 L_{3}|z|\right)^{i}-\dfrac{|z|}{L_{3}} i  \left(L_{3}^{2}|z|^{2}+2 L_{3}|z|\right)^{i-1}\right) + |c\left(X_s,\mathcal{L}_{X_s}\right)|^{2}\sum_{j=0}^{i}  {i \choose j} 2^j  |z|^{2i-j} \\
&\quad \times L_3^{2i-j-2}\left( \left(i-\dfrac{j}{2}-1\right) \left(1+\sqrt{\bE[\vert X_s\vert^2]}\right)^2|X_s|^{p-4}+\sum_{k=2}^{2i-j-2} {2i-j-2 \choose k}\left(1+\sqrt{\bE[\vert X_s\vert^2]}\right)^k|X_s|^{p-2-k}\right)   \\
&\leq |c\left(X_s,\mathcal{L}_{X_s}\right)|^{2} |X_s|^{p-2} \left(\dfrac{i}{L_{3}^{2}}\left(L_{3}^{2}|z|^{2}+2 L_{3}|z|\right)^{i}-\dfrac{|z|}{L_{3}} i  \left(L_{3}^{2}|z|^{2}+2 L_{3}|z|\right)^{i-1}\right)  \\
&\quad +  L_3^{2} \left(1+|X_s|+\sqrt{\bE[\vert X_s\vert^2]}\right)^{2}\sum_{j=0}^{i}  {i \choose j} 2^j  |z|^{2i-j} L_3^{2i-j-2} \Bigg( \big(i-\dfrac{j}{2}-1\big) \big(1+\sqrt{\bE[\vert X_s\vert^2]}\big)^2|X_s|^{p-4}\\
&\quad+\sum_{k=2}^{2i-j-2} {2i-j-2 \choose k}\big(1+\sqrt{\bE[\vert X_s\vert^2]}\big)^k|X_s|^{p-2-k}\Bigg)  \\
&=  |c\left(X_s,\mathcal{L}_{X_s}\right)|^{2} |X_s|^{p-2}  \dfrac{i|z|}{L_{3}}\left(1+L_{3}|z|\right)\left(L_{3}^{2}|z|^{2}+2 L_{3}|z|\right)^{i-1}   +\sum_{j=0}^{i} |z|^{2i-j} Q_{p}\left(p-2,2i-j,|X_s|, 1+\sqrt{\bE[\vert X_s\vert^2]}\right),
\end{align*}
where $2i-j\geq 2$ and
\begin{align*}
Q_{q}(m,n,x,y)= \sum_{\ell_1\leq m,\ell_2\leq n, \ell_1+\ell_2=q}c_{\ell_1\ell_2}x^{\ell_1}y^{\ell_2}.
\end{align*}
This, together with the fact that $\sum_{i=2}^{p/2}{p/2 \choose i}i a^{i-1}=\frac{p}{2}((1+a)^{p/2-1}-1)$ valid for any $a\in \br$, yields to
\begin{align}
&\quad \sum_{i=2}^{p/2} {p/2 \choose i}|X_s|^{p-2i} \left(|c\left(X_s,\mathcal{L}_{X_s}\right) z|^2+2\left\langle X_s, c\left(X_s,\mathcal{L}_{X_s}\right) z\right\rangle \right)^i \notag\\	
&\le  \dfrac{|z|}{L_3}\left(1+L_{3}|z|\right)|c\left(X_s,\mathcal{L}_{X_s}\right)|^{2} |X_s|^{p-2} \sum_{i=2}^{p/2}  {p/2 \choose i}   i\left(L_{3}^{2}|z|^{2}+2 L_{3}|z|\right)^{i-1} \notag \\
&\quad  +\sum_{i=2}^{p/2}  \sum_{j=0}^{i} {p/2 \choose i}|z|^{2i-j} Q_{p}\left(p-2,2i-j,|X_s|, 1+\sqrt{\bE[\vert X_s\vert^2]}\right)  \notag\\
&=  \dfrac{p}{2}|X_s|^{p-2} |c\left(X_s,\mathcal{L}_{X_s}\right)|^2  \dfrac{|z|}{L_3} \left(\left(1+L_{3}|z|\right)^{p-1}-L_3 |z|-1\right) \notag\\
&\quad   +\sum_{i=2}^{p/2}  \sum_{j=0}^{i} {p/2 \choose i}|z|^{2i-j} Q_{p}\left(p-2,2i-j,|X_s|, 1+\sqrt{\bE[\vert X_s\vert^2]}\right). \label{for2}
\end{align}
As a consequence of \eqref{for1} and \eqref{for2}, we have shown that for any $s\geq 0$,
\begin{align}
&\left|X_s+c\left(X_s,\mathcal{L}_{X_s}\right) z\right|^{p}  -|X_s|^p-p |X_s|^{p-2} \left\langle X_s, c\left(X_s,\mathcal{L}_{X_s}\right) z \right\rangle \notag \\
&\le  \dfrac{p}{2}|X_s|^{p-2} |c\left(X_s,\mathcal{L}_{X_s}\right)|^2 \left(|z|^2 + \dfrac{|z|}{L_3} \left(\left(1+L_{3}|z|\right)^{p-1}-L_3 |z|-1\right) \right) \notag\\
&\quad +\sum_{i=2}^{p/2}  \sum_{j=0}^{i} {p/2 \choose i}|z|^{2i-j} Q_{p}\left(p-2,2i-j,|X_s|, 1+\sqrt{\bE[\vert X_s\vert^2]}\right)\notag\\
&=  \dfrac{p}{2L_3}|X_s|^{p-2} |c\left(X_s,\mathcal{L}_{X_s}\right)|^2 |z| \left(\left(1+L_{3}|z|\right)^{p-1}-1\right)  \notag\\
&\quad +\sum_{i=2}^{p/2}  \sum_{j=0}^{i} {p/2 \choose i}|z|^{2i-j} Q_{p}\left(p-2,2i-j,|X_s|, 1+\sqrt{\bE[\vert X_s\vert^2]}\right). \label{ulxct1}
\end{align}	
Therefore, substituting \eqref{ulxct1} into \eqref{ulxt}, we get that for any $t\geq 0$,
\begin{align*}
&e^{ -\lambda t} |X_{t}|^{p} 
\leq |x_0|^p+p\int_{0}^{t}  e^{ -\lambda s} |X_{s}|^{p-2}\bigg[-\dfrac{\lambda}{p} |X_{s}|^{2}+ \langle X_{s},   b\left(X_{s},\mathcal{L}_{X_s}\right) \rangle +\dfrac{p-1}{2}|\sigma\left(X_{s},\mathcal{L}_{X_s}\right)|^{2} \\
&\quad+\dfrac{1}{2L_3} |c\left(X_{s},\mathcal{L}_{X_s}\right)|^2 \int_{\bR_{0}^{d}}  |z| \left(\left(1+L_{3}|z|\right)^{p-1}-1\right)  \nu(dz) \bigg]ds  \\
&\quad+ \int_0^t\int_{\bR_{0}^{d}} e^{- \lambda s}\sum_{i=2}^{p/2}  \sum_{j=0}^{i} {p/2 \choose i}|z|^{2i-j} Q_{p}\left(p-2,2i-j,|X_s|, 1+\sqrt{\bE[\vert X_s\vert^2]}\right)\nu(d z)ds  \\
&\quad +p\int_0^te^{-\lambda s} \left|X_{s}\right|^{p-2}\left\langle X_{s}, \sigma\left(X_{s},\mathcal{L}_{X_s}\right) d W_{s}\right\rangle +\int_0^t\int_{\bR_{0}^{d}} e^{-\lambda s}  \left( \left|X_{s-}+c\left(X_{s-},\mathcal{L}_{X_{s-}}\right) z\right|^{p}-\left|X_{s-}\right|^{p}\right) \widetilde{N}(d s, d z).
\end{align*}
	Now,  we define $\tau_N:=\inf\{t\geq0: |X_t|\geq N \}$ for each $N>0$. Choosing $\lambda= \gamma_1 p$ and using  Condition  \textbf{A7}, Remark \ref{c5} together with $\mathcal{W}_2^2\left(\mathcal{L}_{X_s},\delta_0\right) = \bE[\vert X_s\vert^2]$, we obtain that
	\begin{align}
		&\bE \left[e^{ -\gamma_1 p(t\wedge \tau_N)} |X_{t\wedge \tau_N}|^{p}\right]
		\leq |x_0|^p+{p}\int_{0}^{t}  e^{ -\gamma_1 ps} \bE \left[|X_{s}|^{p-2}\right]\left(\gamma_2\mathcal{W}_2^2\left(\mathcal{L}_{X_s},\delta_0\right)+\eta\right) ds  \notag  \\
		&\quad + \int_0^t\int_{\bR_{0}^{d}} e^{- \gamma_1 p s}\sum_{i=2}^{p/2}  \sum_{j=0}^{i} {p/2 \choose i}|z|^{2i-j} \bE \left[Q_{p}\left(p-2,2i-j,|X_s|, 1+\sqrt{\bE[\vert X_s\vert^2]}\right)\right]\nu(d z)ds  \notag \\		
		&\leq |x_0|^p+{p}\int_{0}^{t}  e^{ -\gamma_1 ps} \bE \left[|X_{s}|^{p-2}\right]\left(\gamma_2\bE \left[|X_{s}|^{2}\right]+\eta\right) ds  \notag  \\
		&\quad + \int_0^t\int_{\bR_{0}^{d}} e^{- \gamma_1 p s}\sum_{i=2}^{p/2}  \sum_{j=0}^{i} {p/2 \choose i}|z|^{2i-j} \bE \left[Q_{p}\left(p-2,2i-j,|X_s|, 1+\sqrt{\bE[\vert X_s\vert^2]}\right)\right]\nu(d z)ds . \label{ulxct2}
	\end{align}	
	Next, using $e^{ -\gamma_1 p(t\wedge \tau_N)}\geq e^{ -\gamma_1 pt}$ and the induction assumption \eqref{inductionq}, there exists a positive constant $C(T,p)$ which does not depend on $N$ such that
	\begin{align}\label{p}
	\sup_{t\in[0,T]}\bE \left[\left\vert X_{t\wedge \tau_N}\right\vert^p\right] \leq C(T,p).
	\end{align}
	This yields to
	$$
	\sup_{t\in[0,T]}\mathbb{P}(\tau_N<t)\leq \dfrac{C(T,p)}{N^{p}},
	$$
	which implies that $\tau_N$ tends to infinity a.s. as $N$ tends to infinity. Now, it suffices to let $N\uparrow \infty$ and use Fatou's lemma for the left-hand side of \eqref{p} to get that $\sup_{t\in[0,T]}\bE [\vert X_{t}\vert^p] \leq C(T,p)$. Thus, by the induction principle, we have shown \eqref{momentp}.
	
	\underline{\it Step 2:} We next wish to show \eqref{proofmomentp}  for any even natural number $p\in [2,p_0]$.
	
	First, applying \eqref{ulxt} to $\lambda=2\gamma$ and $p=2$, and using \textbf{A7}, we get
	\begin{align}
	&e^{-2\gamma t}  \left|X_{t}\right|^{2} =|x_0|^2 +\int_{0}^{t} e^{-2\gamma s}\left[-2\gamma   \left|X_{s}\right|^2 + 2\left\langle X_{s}, b\left(X_{s},\mathcal{L}_{X_s}\right)\right\rangle +  \left|\sigma\left(X_{s},\mathcal{L}_{X_s}\right)\right|^{2} \right] ds \notag \\
	&\quad +2\int_{0}^{t} e^{-2\gamma s} \left\langle X_{s}, \sigma\left(X_{s},\mathcal{L}_{X_s}\right) d W_{s}\right\rangle +\int_{0}^t \int_{\bR_{0}^ d} e^{-2\gamma s}  \left(\left|X_{s-}+c\left(X_{s-},\mathcal{L}_{X_{s-}}\right) z\right|^{2}-\left|X_{s-}\right|^{2}\right) \widetilde{N}(d s, d z) \notag \\
	&\quad +\int_{0}^t \int_{\bR_{0}^{d}} e^{-2\gamma  s} |c\left(X_{s},\mathcal{L}_{X_s}\right)z|^2 \nu(dz)ds \notag \\
	&\le |x_0|^2 +2\int_{0}^{t} e^{-2\gamma s}\left[-\gamma   \left|X_{s}\right|^2 + \left\langle X_{s}, b\left(X_{s},\mathcal{L}_{X_s}\right)\right\rangle +  \dfrac{1}{2}\left|\sigma\left(X_{s},\mathcal{L}_{X_s}\right)\right|^{2}+ \dfrac{1}{2}|c\left(X_{s},\mathcal{L}_{X_s}\right)|^2 \int_{\bR_{0}^{d}} |z|^2 \nu(dz) \right] ds \notag \\
	&\quad +2\int_{0}^{t} e^{-2\gamma s} \left\langle X_{s}, \sigma\left(X_{s},\mathcal{L}_{X_s}\right) d W_{s}\right\rangle +\int_{0}^t \int_{\bR_{0}^ d} e^{-2\gamma s}  \left(\left|X_{s-}+c\left(X_{s-},\mathcal{L}_{X_{s-}}\right) z\right|^{2}-\left|X_{s-}\right|^{2}\right) \widetilde{N}(d s, d z) \notag \\
	&\leq |x_{0}|^{2}+ 2 \int_{0}^{t} e^{ - 2\gamma s }\left(-\gamma_2\left|X_{s}\right|^2 +\gamma_2\mathcal{W}_2^2\left(\mathcal{L}_{X_s},\delta_0\right)+\eta\right) d s + 2\int_{0}^{t}  e^{ -2\gamma s} \langle X_{s},\sigma\left(X_{s},\mathcal{L}_{X_s}\right) d W_{s}\rangle  \notag \\
	&\quad +\int_{0}^{t} \int_{\mathbb{R}_0^d} e^{ - 2\gamma s } \left(\left|X_{s-}+c\left(X_{s-},\mathcal{L}_{X_{s-}}\right) z\right|^{2}-\left|X_{s-}\right|^{2}\right) \widetilde{N}(d s, d z).\label{ulxt2}
	\end{align}
Thanks to  the fact that $\mathcal{W}_2^2\left(\mathcal{L}_{X_s},\delta_0\right) = \bE[\vert X_s\vert^2]$, and the estimate  \eqref{momentp}, we get 
	$$
	\bE \left[e^{ - 2\gamma t}|X_{t}|^{2}\right] \le |x_{0}|^{2}+ 2\eta \int_{0}^{t} e^{ -2\gamma s} d s,
	$$	
	which yields to
	\begin{align*} 
	\bE \left[|X_{t}|^{2}\right] \le \begin{cases}  
	\left(|x_{0}|^{2}+\dfrac{\eta}{\gamma}\right) e^{2\gamma t}-\dfrac{\eta}{\gamma} & \text{ if \;\;} \gamma \not =0,\\
	|x_{0}|^{2}+2\eta t & \text{ if \;\;} \gamma=0.
	\end{cases} 
	\end{align*}	
	Thus, \eqref{proofmomentp}  holds for $p=2$. \\
	\indent Now, we suppose that \eqref{proofmomentp} is valid for all even integer $q \in [2, p-2]$. That is,
	\begin{equation}\label{induction} \begin{split}  
		\bE \left[\vert X_t\vert^q\right] \le \begin{cases}  C_q (1+e^{ \gamma q t}) & \text{ if \;\;} \gamma \not = 0,\\
			C_q (1+t)^{q/2} & \text{ if \;\;} \gamma  = 0, q=2 \text{\; or \;} \gamma  = 0, \gamma_2>0, q\in (2,p-2],\\
			C_q (1+t)^{q} & \text{ if \;\;} \gamma  = 0, \gamma_2=0, q\in (2,p-2].
			 \end{cases} 
		\end{split} 
	\end{equation}
	We are going to show that  \eqref{proofmomentp} holds for  even integer $p$. For this, it suffices to use  \eqref{ulxct2}, the inductive assumption \eqref{induction} and Condition \textbf{A5}.\\	
	\underline{Case $\gamma \ne 0$:} We have
	\begin{align*}
	\bE \left[e^{ -\gamma_1 p(t\wedge \tau_N)} |X_{t\wedge \tau_N}|^{p}\right]
	\leq |x_{0}|^{p}+  C \int_{0}^{t} e^{ -\gamma_1 ps}\left(1+e^{ \gamma p t}\right)ds=|x_{0}|^{p}+  C \int_{0}^{t} \left(e^{ -\gamma_1 ps}+e^{ \gamma_2 ps}\right)ds.
	\end{align*}
	Thanks to fact that $\tau_N\uparrow \infty$ a.s. as $N\uparrow \infty$,  applying Fatou's lemma, we get
	\begin{align*}
	\mathbb{E}\left[e^{ -\gamma_1 p t} |X_{t}|^{p} \right]
	\leq |x_{0}|^{p}+  C \int_{0}^{t} \left(e^{ -\gamma_1 ps}+e^{ \gamma_2 ps}\right)ds.
	\end{align*}
	When $\gamma_1=0$, we have
	\begin{align*}
	\bE \left[\vert X_{t}\vert^{p}\right]&\leq |x_{0}|^{p}+C\left(1+t+e^{ \gamma_2 pt}\right)\\
	&\leq C\left(1+e^{ \gamma pt}\right).
	\end{align*}
	When $\gamma_1\neq 0$, we have
	\begin{align*}
	\bE \left[\vert X_{t}\vert^{p}\right]&\leq |x_{0}|^{p}e^{\gamma_1 p t}+\dfrac{C}{p\gamma_1}\left(e^{\gamma_1 p t}-1\right)+\dfrac{C}{p\gamma_2}e^{\gamma p t}\\
	&\leq C\left(1+e^{ \gamma pt}\right).
	\end{align*}	
	\underline{Case $\gamma=0$:}	When $\gamma_2>0$, we have
	\begin{align*}
	\bE \left[e^{ -\gamma_1 p(t\wedge \tau_N)} |X_{t\wedge \tau_N}|^{p}\right]
	&\leq |x_{0}|^{p}+  C \int_{0}^{t} e^{ -\gamma_1 ps}\left(1+s\right)^{p/2}ds\\
	&\leq |x_{0}|^{p}+  C \left(1+t\right)^{p/2}\int_{0}^{t} e^{ -\gamma_1 ps}ds.
	\end{align*}
	Then, letting $N\uparrow \infty$ and using $-\gamma_1=\gamma_2$, we obtain
	\begin{align*}
	\mathbb{E}\left[e^{ \gamma_2 p t} |X_{t}|^{p} \right]
	\leq |x_{0}|^{p}+C\left(1+t\right)^{p/2}\int_{0}^{t} e^{ \gamma_2 ps}ds\leq |x_{0}|^{p}+\dfrac{C}{\gamma_2 p}\left(1+t\right)^{p/2}e^{ \gamma_2 pt}.
	\end{align*}
	Hence,
	\begin{align*}
	\mathbb{E}\left[|X_{t}|^{p} \right]\leq C\left(1+t\right)^{p/2}.
	\end{align*}
	When $\gamma_2=\gamma_1=0$, we have
	\begin{align*}
	\bE \left[|X_{t\wedge \tau_N}|^{p}\right]
	\leq |x_{0}|^{p}+  C \int_{0}^{t} \left(1+s\right)^{p-1}ds\leq C\left(1+t\right)^{p}.
	\end{align*}
	Then, letting $N\uparrow \infty$, we obtain
	\begin{align*}
	\bE \left[|X_{t}|^{p}\right]\leq C\left(1+t\right)^{p}.
	\end{align*}
	Consequently, \eqref{proofmomentp} holds for even integer $p$. Due to the induction principle, \eqref{proofmomentp} is valid for any even natural number $p\in [2,p_0]$. Finally,  \eqref{proofmomentp} is also valid for any $p \in [2,p_0]$ thanks to H\"older's inequality. This finishes the proof.
\end{proof}

\section{Propagation of chaos} \label{Sec:3}
For $N\in\mathbb{N}$, suppose that $(W^i,Z^i)$ are  independent copies of the couple $(W,Z)$ for $i\in \{1,\ldots,N\}$. Let $N^i(dt,dz)$ be the Poisson random measure associated with the jumps of the L\'evy process $Z^i$ with intensity measure $\nu(dz)dt$, and $\widetilde{N}^i(dt,dz):=N^i(dt,dz)-\nu(dz)dt$ be the compensated Poisson random measure associated with $N^i(dt,dz)$. Thus, the L\'evy-It\^o decomposition of $Z^i$ is given by $Z_t^i=\int_{0}^{t}\int_{\mathbb{R}_0^d}z\widetilde{N}^i(ds,dz)$ for $t\geq 0$. We now consider the system of non-interacting particles which is associated with the L\'evy-driven McKean-Vlasov SDE \eqref{eqn1}, where  the state $X^i=(X_t^i)_{t \geq 0}$ of the particle $i$ is defined by
\begin{equation} \begin{split}
X_t^i&=x_0+\int_0^t b(X_s^i,\mathcal{L}_{X_s^i})ds+\int_0^t \sigma(X_s^i,\mathcal{L}_{X_s^i})dW_s^i + \int_0^t c\left(X_{s-}^i,\mathcal{L}_{X_{s-}^i}\right)dZ_s^i\\
&=x_0+\int_0^t b(X_s^i,\mathcal{L}_{X_s^i})ds+\int_0^t \sigma(X_s^i,\mathcal{L}_{X_s^i})dW_s^i + \int_0^t\int_{\mathbb{R}_0^d} c\left(X_{s-}^i,\mathcal{L}_{X_{s-}^i}\right)z\widetilde{N}^i(d s, d z),
\end{split}
\end{equation}
for any $t\geq 0$ and $i\in \{1,\ldots,N\}$.

For $\boldsymbol{x}^N:=(x_1,x_2,\ldots,x_N), \boldsymbol{y}^N:=(y_1,y_2,\ldots,y_N)\in \mathbb{R}^{dN}$,  we have
\begin{align*}
\mathcal{W}_2^2(\mu^{\boldsymbol{x}^N},\delta_0)= \dfrac{1}{N}\sum_{i=1}^{N}\left\vert x_i \right\vert^2.
\end{align*}
Here, the empirical measure is defined by $\mu^{\boldsymbol{x}^N}(dx):=\frac{1}{N}\sum_{i=1}^{N}\delta_{x_i}(dx)$, where $\delta_x$ denotes the Dirac measure at $x$. Moreover,  a standard bound for the Wasserstein distance between two empirical measures $\mu^{\boldsymbol{x}^N},\mu^{\boldsymbol{y}^N}$ is given by
\begin{align*}
\mathcal{W}_2^2(\mu^{\boldsymbol{x}^N},\mu^{\boldsymbol{y}^N})\leq \dfrac{1}{N}\sum_{i=1}^{N}\left\vert x_i-y_i \right\vert^2=\dfrac{1}{N} \left\vert \boldsymbol{x}^N-\boldsymbol{y}^N \right\vert^2,
\end{align*}
(see (1.24) of  \cite{C16}).
	
Now, the true measure $\mathcal{L}_{X_t}$ at each time $t$ is approximated by the empirical measure
\begin{align}
\mu_{t}^{\boldsymbol{X}^N}(dx):=\dfrac{1}{N}\sum_{i=1}^{N}\delta_{X_{t}^{i,N}}(dx), 
\end{align}	
where $\boldsymbol{X}^N=(\boldsymbol{X}_t^N)_{t\geq 0}=(X_t^{1,N},\ldots,X_t^{N,N})_{t\geq 0}^\mathsf{T}$, which is called the system of interacting particles, is the solution to the $\mathbb{R}^{dN}$-dimensional L\'evy-driven SDE with components $X^{i,N}=(X^{i,N}_t)_{t\geq 0}$
\begin{equation} \label{equa2}\begin{split}
X_t^{i,N}&=x_0+\int_0^tb(X_s^{i,N},\mu_{s}^{\boldsymbol{X}^N})ds+\int_0^t\sigma(X_s^{i,N},\mu_{s}^{\boldsymbol{X}^N})dW_s^i + \int_0^t c\left(X_{s-}^{i,N},\mu_{s-}^{\boldsymbol{X}^N}\right)dZ_s^i\\
&=x_0+\int_0^tb(X_s^{i,N},\mu_{s}^{\boldsymbol{X}^N})ds+\int_0^t\sigma(X_s^{i,N},\mu_{s}^{\boldsymbol{X}^N})dW_s^i + \int_0^t\int_{\mathbb{R}_0^d} c\left(X_{s-}^{i,N},\mu_{s-}^{\boldsymbol{X}^N}\right)z\widetilde{N}^i(d s, d z),
\end{split} 
\end{equation}
for any $t \geq 0$ and $i\in \{1,\ldots,N\}$.

	Observe that the interacting particle system $\boldsymbol{X}^N=(X^{i,N})^\mathsf{T}_{i\in\{1,\ldots,N\}}$ can be viewed as an ordinary L\'evy-driven SDE with random coefficients taking values in $\mathbb{R}^{d  N}$. Therefore, under Conditions \textbf{A1}, \textbf{A3} and \textbf{A2} valid for $\kappa_1=\kappa_2=1, L_1= L_2>0$, there exists a unique c\`adl\`ag solution such that
	$$
	\max_{i\in\{1,\ldots,N\}}\sup_{t\in[0,T]}\bE \left[\vert X_t^{i,N}\vert^2\right]\leq K,
	$$
	for any $N\in\mathbb{N}$, where $K>0$ does not depend on $N$.

\begin{Prop}\label{moment XiN}
	Let $X^{i,N}=(X_t^{i,N})_{t \geq 0}$ be a solution to equation \eqref{equa2}. Assume Conditions \textbf{A6, A7} and that $\sigma$ is bounded on $C\times\mathcal{P}_2(\mathbb{R}^d)$ for every compact subset $C$ of \; $\mathbb{R}^d$,  and \textbf{A5} holds for $q=2p_0$.
	Then for any $p \in [2,p_0]$, there exists a positive constant $C_p$ such that for any $t \ge 0$,
	\begin{equation*} \begin{split} 
	\max_{i\in\{1,\ldots,N\}}\bE \left[\vert X_t^{i,N}\vert^p\right] \le \begin{cases}  C_p (1+e^{\gamma pt}) & \text{ if \;\;} \gamma \not = 0,\\
	C_p (1+t)^{p/2} & \text{ if \;\;} \gamma  = 0, p=2 \text{\; or \;} \gamma  = 0, \gamma_2>0, p\in (2,p_0],\\ 
	C_p (1+t)^{p} & \text{ if \;\;} \gamma  = 0, \gamma_2=0, p\in (2,p_0],
	\end{cases} 
	\end{split} 
	\end{equation*}
	where $\gamma=\gamma_1+\gamma_2$. 
\end{Prop}
Note that when $\gamma<0$, we have that $\max_{i\in\{1,\ldots,N\}}\sup_{t \geq 0} \bE \left[\vert X_t^{i,N}\vert^p\right] \le 2C_p$.
\begin{proof} The proof follows the same lines as the one of Proposition \ref{nghiem dung 1}, thus we omit it.
\end{proof}
Next, we provide a result on the propagation of chaos which is the key to the convergence as $N\uparrow \infty$.  To simplify the exposition, we define
\begin{align*}
	\varphi(N)=\begin{cases}
		N^{-1/2} \quad&\textnormal{if\;\;} d<4,\\
		N^{-1/2}\ln N\quad&\textnormal{if\;\;} d=4,\\
		N^{-2/d}\quad&\textnormal{if\;\;} d>4.
	\end{cases}
\end{align*}
\begin{Prop}\label{PoC} Assume that all conditions in Proposition \ref{moment XiN} hold and that Condition \textbf{A2} holds for $\kappa_1=\kappa_2=1$, $L_1\in\mathbb{R}$, $L_2\geq 0$. Then, we have
\begin{equation}\label{esti1} \begin{split}
\max_{i\in\{1,\ldots,N\}}\sup_{t\in[0,T]}\mathbb{E}\left[\left\vert X_t^i-X_t^{i,N}\right\vert^2\right]	\leq C_T \varphi(N),
\end{split} 
\end{equation}
for any $N\in\mathbb{N}$, where the positive constant $C_T$ does not depend on $N$.

Assume further that $L_1+L_2<0$ and $\gamma<0$. Then, we have
\begin{equation}\label{esti2} \begin{split}
\max_{i\in\{1,\ldots,N\}}\sup_{t\geq 0}\mathbb{E}\left[\left\vert X_t^i-X_t^{i,N}\right\vert^2\right]	\leq C \varphi(N),
\end{split} 
\end{equation}
for any $N\in\mathbb{N}$, where the positive constant $C$ does not depend on $N$ and $T$.
\end{Prop}
\begin{proof} Observe that  for any $t\geq 0$,
\begin{align*}
X_t^i-X_t^{i,N}&=\int_0^t \left(b(X_s^i,\mathcal{L}_{X_s^i})-b(X_s^{i,N},\mu_{s}^{\boldsymbol{X}^N})\right)ds+\int_0^t \left(\sigma(X_s^i,\mathcal{L}_{X_s^i})-\sigma(X_s^{i,N},\mu_{s}^{\boldsymbol{X}^N})\right)dW_s^i \\
&\quad+ \int_0^t\int_{\mathbb{R}_0^d} \left(c(X_{s-}^i,\mathcal{L}_{X_{s-}^i})-c(X_{s-}^{i,N},\mu_{s-}^{\boldsymbol{X}^N})\right)z\widetilde{N}^i(d s, d z).
\end{align*}
Then, for $\lambda \in \br$, applying It\^o's formula and Condition \textbf{A2} valid for $\kappa_1=\kappa_2=1$, $L_1\in\mathbb{R}$, $L_2\geq 0$, we obtain that for any $t\geq 0$,	
\begin{align*}
&e^{-\lambda t}\vert X_t^i-X_t^{i,N}\vert^2=\int_0^te^{-\lambda s}\Big(-\lambda\vert X_s^i-X_s^{i,N}\vert^2+2\langle X_s^i-X_s^{i,N},b(X_s^i,\mathcal{L}_{X_s^i})-b(X_s^{i,N},\mu_{s}^{\boldsymbol{X}^N})\rangle\\
&\quad+\vert \sigma(X_s^i,\mathcal{L}_{X_s^i})-\sigma(X_s^{i,N},\mu_{s}^{\boldsymbol{X}^N})\vert^2\Big)ds+2\int_0^te^{-\lambda s}\langle X_s^i-X_s^{i,N},(\sigma(X_s^i,\mathcal{L}_{X_s^i})-\sigma(X_s^{i,N},\mu_{s}^{\boldsymbol{X}^N}))dW_s^i\rangle\\
&\quad+\int_0^t\int_{\mathbb{R}_0^d}e^{-\lambda s}\left(\vert X_{s-}^i-X_{s-}^{i,N}+\big(c(X_{s-}^i,\mathcal{L}_{X_{s-}^i})-c(X_{s-}^{i,N},\mu_{s-}^{\boldsymbol{X}^N})\big)z\vert^2-\vert X_{s-}^i-X_{s-}^{i,N}\vert^2\right)\widetilde{N}^i(d s, d z)\\
&\quad+\int_0^t\int_{\mathbb{R}_0^d}e^{-\lambda s}\big\vert \big(c(X_{s-}^i,\mathcal{L}_{X_{s-}^i})-c(X_{s-}^{i,N},\mu_{s-}^{\boldsymbol{X}^N})\big)z\big\vert^2\nu(d z) ds\\
&\leq \int_0^te^{-\lambda s}\Big(-\lambda\vert X_s^i-X_s^{i,N}\vert^2+2\langle X_s^i-X_s^{i,N},b(X_s^i,\mathcal{L}_{X_s^i})-b(X_s^{i,N},\mu_{s}^{\boldsymbol{X}^N})\rangle\\
&\quad+\vert \sigma(X_s^i,\mathcal{L}_{X_s^i})-\sigma(X_s^{i,N},\mu_{s}^{\boldsymbol{X}^N})\vert^2+\vert c(X_{s-}^i,\mathcal{L}_{X_{s-}^i})-c(X_{s-}^{i,N},\mu_{s-}^{\boldsymbol{X}^N}\vert^2\int_{\mathbb{R}_0^d}\vert z\vert^2\nu(d z)\Big)ds\\
&\quad+2\int_0^te^{-\lambda s}\langle X_s^i-X_s^{i,N},(\sigma(X_s^i,\mathcal{L}_{X_s^i})-\sigma(X_s^{i,N},\mu_{s}^{\boldsymbol{X}^N}))dW_s^i\rangle\\
&\quad+\int_0^t\int_{\mathbb{R}_0^d}e^{-\lambda s}\left(\big\vert X_{s-}^i-X_{s-}^{i,N}+\big(c(X_{s-}^i,\mathcal{L}_{X_{s-}^i})-c(X_{s-}^{i,N},\mu_{s-}^{\boldsymbol{X}^N})\big)z\big\vert^2-\vert X_{s-}^i-X_{s-}^{i,N}\vert^2\right)\widetilde{N}^i(d s, d z)\\
&\leq \int_0^te^{-\lambda s}\Big(-\lambda\vert X_s^i-X_s^{i,N}\vert^2+L_1\vert X_s^i-X_s^{i,N}\vert^2+L_2\mathcal{W}_2^2(\mathcal{L}_{X_{s}^i},\mu_{s}^{\boldsymbol{X}^N})\Big)ds\\
&\quad+2\int_0^te^{-\lambda s}\langle X_s^i-X_s^{i,N},(\sigma(X_s^i,\mathcal{L}_{X_s^i})-\sigma(X_s^{i,N},\mu_{s}^{\boldsymbol{X}^N}))dW_s^i\rangle\\
&\quad+\int_0^t\int_{\mathbb{R}_0^d}e^{-\lambda s}\left(\big\vert X_{s-}^i-X_{s-}^{i,N}+\big(c(X_{s-}^i,\mathcal{L}_{X_{s-}^i})-c(X_{s-}^{i,N},\mu_{s-}^{\boldsymbol{X}^N})\big)z\vert^2-\vert X_{s-}^i-X_{s-}^{i,N}\big\vert^2\right)\widetilde{N}^i(d s, d z).
\end{align*}
Therefore, taking the expectation and using the estimate $\mathcal{W}_2^2(\mathcal{L}_{X_{s}^i},\mathcal{L}_{X_{s}^{i,N}})\leq \mathbb{E}[\vert X_{s}^i-X_{s}^{i,N}\vert^2]$, we obtain that for any $\varepsilon>0$,
\begin{align}
&e^{-\lambda t}\mathbb{E}\left[\left\vert X_t^i-X_t^{i,N}\right\vert^2\right]\leq \int_0^te^{-\lambda s}\left((-\lambda+L_1)\mathbb{E}\left[\vert X_s^i-X_s^{i,N}\vert^2\right] +L_2\mathbb{E}\left[\mathcal{W}_2^2(\mathcal{L}_{X_{s}^i},\mu_{s}^{\boldsymbol{X}^N})\right]\right)ds\notag\\
&\leq\int_0^te^{-\lambda s}\Bigg((-\lambda+L_1)\mathbb{E}\left[\vert X_s^i-X_s^{i,N}\vert^2\right] \notag\\
&\quad+L_2\bigg(\big(1+\frac{\varepsilon}{L_2}\big)\mathbb{E}\left[\mathcal{W}_2^2(\mathcal{L}_{X_{s}^i},\mathcal{L}_{X_{s}^{i,N}})\right]+\big(1+\frac{L_2}{\varepsilon}\big)\mathbb{E}\left[\mathcal{W}_2^2(\mathcal{L}_{X_{s}^{i,N}},\mu_{s}^{\boldsymbol{X}^N})\right]\bigg)\Bigg)ds \notag\\
&\leq \int_0^te^{-\lambda s}\left(\left(-\lambda+L_1+L_2+\varepsilon\right)\mathbb{E}\left[\vert X_s^i-X_s^{i,N}\vert^2\right]+L_2\big(1+\frac{L_2}{\varepsilon}\big)\mathbb{E}\left[\mathcal{W}_2^2(\mathcal{L}_{X_{s}^{i,N}},\mu_{s}^{\boldsymbol{X}^N})\right]\right)ds. \label{estimate1}
\end{align}
Moreover, from Proposition \ref{moment XiN}, we have that for any $p\in (4,p_0]$,
\begin{align*}
\max_{i\in\{1,\ldots,N\}}\sup_{t\in[0,T]}\mathbb{E}\left[\left\vert X_t^{i,N}\right\vert^p\right]	\leq C_T,
\end{align*}
for some constant  $C_T>0$. This, together with \cite[Theorem 5.8]{CD18}, deduces that
\begin{align}\label{estimate2}
\max_{i\in\{1,\ldots,N\}}\mathbb{E}\left[\mathcal{W}_2^2(\mathcal{L}_{X_{s}^{i,N}},\mu_{s}^{\boldsymbol{X}^N})\right] &\leq C \begin{cases}
N^{-1/2} \quad&\textnormal{if\;\;} d<4,\\
N^{-1/2}\ln N\quad&\textnormal{if\;\;} d=4,\\
N^{-2/d}\quad&\textnormal{if\;\;} d>4
\end{cases} \notag\\
&=C \varphi(N),
\end{align}
for any $s\in[0,T]$, where the positive constant $C$ does not depend on the time.

Consequently, it suffices to choose $\lambda=L_1+L_2+\varepsilon$ in \eqref{estimate1} and use the estimate \eqref{estimate2} to conclude \eqref{esti1}.

Finally, when $\gamma<0$, it follows from Proposition \ref{moment XiN} that for any $p\in (4,p_0]$,
\begin{align*}
\max_{i\in\{1,\ldots,N\}}\sup_{t \geq 0}\mathbb{E}\left[\left\vert X_t^{i,N}\right\vert^q\right]	\leq C,
\end{align*}
where the positive constant $C$ does not depend on $T$. Furthermore, when $L_1+L_2<0$, one can always choose $\varepsilon$ sufficiently small such that $\lambda<0$. This allows to conclude \eqref{esti2}. The desired proof follows.
\end{proof}

\section{Tamed-adaptive Euler-Maruyama scheme} \label{sec:main} 

\subsection{Definition}
Let $\sigma_\Delta=(\sigma_{\Delta,ij})_{1\leq i,j\leq d}: \mathbb{R}^d \times \mathcal{P}_2(\mathbb{R}^d)\to \mathbb{R}^d\otimes \mathbb{R}^d$ and $c_\Delta=(c_{\Delta,ij})_{1\leq i,j\leq d}: \mathbb{R}^d \times \mathcal{P}_2(\mathbb{R}^d)\to \mathbb{R}^d\otimes \mathbb{R}^d$ be approximations of  the coefficients $\sigma$ and $c$,  respectively, which will be specified later. For all $i\in\{1,\ldots,N\}$, $\Delta \in (0,1)$ and $k\in\mathbb{N}$, we define the tamed-adaptive Euler-Maruyama discretization of equation \eqref{equa2} by
\begin{equation}\label{EM1}
\begin{cases} 
&t_0=0, \quad \widehat{X}_0^{i,N}=x_0, \quad t_{k+1}=t_k+\textbf{h}(\widehat{\boldsymbol{X}}_{t_k}^N,\mu_{t_k}^{\widehat{\boldsymbol{X}}^N})\Delta,\\ 
&\widehat{X}_{t_{k+1}}^{i,N}= \widehat{X}_{t_k}^{i,N}+b(\widehat{X}_{t_k}^{i,N},\mu_{t_k}^{\widehat{\boldsymbol{X}}^N})(t_{k+1}-t_k)+\sm_{\Delta}(\widehat{X}_{t_k}^{i,N},\mu_{t_k}^{\widehat{\boldsymbol{X}}^N})(W_{t_{k+1}}^i-W_{t_{k}}^i) + c_{\Delta}(\widehat{X}_{t_k}^{i,N},\mu_{t_k}^{\widehat{\boldsymbol{X}}^N})(Z_{t_{k+1}}^i-Z_{t_k}^i),
\end{cases} 
\end{equation}
where
\begin{align*}
&\widehat{\boldsymbol{X}}_{t_k}^N=\left(\widehat{X}^{1,N}_{t_k},\ldots,\widehat{X}^{N,N}_{t_k}\right),\\
&\mu_{t_k}^{\widehat{\boldsymbol{X}}^N}(dx):=\dfrac{1}{N}\sum_{i=1}^{N}\delta_{\widehat{X}_{t_k}^{i,N}}(dx),\\
&\textbf{h}(\widehat{\boldsymbol{X}}^N_{t_k},\mu_{t_k}^{\widehat{\boldsymbol{X}}^N})=\min\left\{h(\widehat{X}^{1,N}_{t_k},\mu_{t_k}^{\widehat{\boldsymbol{X}}^N}),\ldots,h(\widehat{X}^{N,N}_{t_k},\mu_{t_k}^{\widehat{\boldsymbol{X}}^N})\right\},
\end{align*}
and 
\begin{equation} \label{chooseh} 
h(x,\mu)=\dfrac{h_0}{(1+|b(x,\mu)|+|\sm(x,\mu)|+|x|^l)^2+|c(x,\mu)|^{p_0}},
\end{equation}
for $x \in \br^d$, $\mu\in \mathcal{P}_2(\mathbb{R}^d)$ and some positive constant $h_0$. Here, the constants $\ell$ and $p_0$ are respectively defined in Conditions  \textbf{A4} and  \textbf{A7}.  

In all what follows, to simplify the exposition, we take $h_0=1$ in the proofs.

Now, we provide sufficient conditions to ensure $t_k \uparrow \infty$ as $k \uparrow \infty$, which shows that the tamed-adaptive Euler-Maruyama approximation scheme \eqref{EM1} is well-defined. 

\begin{Prop}\label{dinh ly 4}
	Assume that Condition \textbf{A5} holds for $p=2$ and there exist positive constants $L, \beta_1$ and $\beta_2$ such that the functions $h, b,  \sm_{\Delta}$ and $c_{\Delta}$ satisfy 
	\begin{enumerate}[\indent \bf T1.]
		\item  $\dfrac{1}{h(x,\mu)}\leq L\left(1+\vert x\vert^{\beta_1}+\mathcal{W}_2^{\beta_2}(\mu,\delta_0)\right)$ \; and \; $\vert b(x,\mu)\vert \left(1+\vert b(x,\mu)\vert\right) h(x,\mu) \leq L$;
		\item  $\left\langle x, b(x,\mu)-b(0,\delta_0) \right\rangle \le L \left(|x|^2+ \mathcal{W}_2^2(\mu,\delta_0)\right)$;
		\item  $|\sm_{\Delta}(x,\mu)|\left(1+\vert x\vert\right)\le \dfrac{L}{\sqrt{\Delta}}$;  $|c_{\Delta}(x,\mu)|\left(1+\vert x\vert\right)\le \dfrac{L}{\sqrt{\Delta}}$; 		$\vert b(x,\mu)\vert \vert c_{\Delta}(x,\mu)\vert \le \dfrac{L}{\sqrt{\Delta}}$;		
	\end{enumerate}
	for any $x \in \mathbb{R}^d$ and $\mu\in \mathcal{P}_2(\mathbb{R}^d)$. Then, we have
	\begin{align*}
	\lim\limits_{k \to +\infty} t_k =+\infty \quad \text{a.s.} 
	\end{align*}
\end{Prop}
\begin{proof}
For all $i\in\{1,\ldots,N\}$ and $H>0$, we define the tamed-adaptive Euler-Maruyama discretization of equation \eqref{equa2} as follows
\begin{equation}\label{EM00}
\begin{cases} 
&t_0^H=0, \quad \widehat{X}_0^{i,N,H}=x_0, \quad t_{k+1}^H=t_k^H+\textbf{h}^H(\widehat{\boldsymbol{X}}^{N,H}_{t_k^H},\mu_{t_k^H}^{\widehat{\boldsymbol{X}}^{N,H}})\Delta,\\ 
&  \widehat{X}_{t_{k+1}^H}^{i,N,H}= \widehat{X}_{t_k^H}^{i,N,H}+b_H(\widehat{X}_{t_k^H}^{i,N,H},\mu_{t_k^H}^{\widehat{\boldsymbol{X}}^{N,H}})(t_{k+1}^H-t_k^H)+\sm_{\Delta}(\widehat{X}_{t_k^H}^{i,N,H},\mu_{t_k^H}^{\widehat{\boldsymbol{X}}^{N,H}})(W_{t_{k+1}^H}^i-W_{t_{k}^H}^i) \\
&\qquad \quad\quad \quad+ c_{\Delta}(\widehat{X}_{t_k^H}^{i,N,H},\mu_{t_k^H}^{\widehat{\boldsymbol{X}}^{N,H}})(Z_{t_{k+1}^H}^i-Z_{t_k^H}^i),
\end{cases} 
\end{equation}
where
\begin{align*}
&\widehat{\boldsymbol{X}}_{t_k}^{N,H}=\left(\widehat{X}^{1,N,H}_{t_k},\ldots,\widehat{X}^{N,N,H}_{t_k}\right),\\
&\mu_{t_k}^{\widehat{\boldsymbol{X}}^{N,H}}(dx):=\dfrac{1}{N}\sum_{i=1}^{N}\delta_{\widehat{X}_{t_k}^{i,N,H}}(dx),\\
&\textbf{h}^H(\widehat{\boldsymbol{X}}^N_{t_k^H},\mu_{t_k^H}^{\widehat{\boldsymbol{X}}^{N,H}})=\min\left\{h^H(\widehat{X}^{1,N,H}_{t_k^H},\mu_{t_k^H}^{\widehat{\boldsymbol{X}}^{N,H}}),\ldots,h^H(\widehat{X}^{N,N,H}_{t_k^H},\mu_{t_k^H}^{\widehat{\boldsymbol{X}}^{N,H}})\right\},\\
&h^H(x,\mu)=\begin{cases}
h(x,\mu)  & \text{ if \;\;} \vert x\vert^{\beta_1}+\mathcal{W}_2^{\beta_2}(\mu,\delta_0) \leq H,\\
\dfrac{1}{1+H} & \text{ if \;\;} \vert x\vert^{\beta_1}+\mathcal{W}_2^{\beta_2}(\mu,\delta_0) >H,
 \end{cases} \\ 
 &b_H(x,\mu)=\begin{cases}
 b(x,\mu)  & \text{ if \;\;} \vert x\vert^{\beta_1}+\mathcal{W}_2^{\beta_2}(\mu,\delta_0) \leq H,\\
 \dfrac{x}{1+\vert x\vert^2} + b(0,\delta_0)& \text{ if \;\;} \vert x\vert^{\beta_1}+\mathcal{W}_2^{\beta_2}(\mu,\delta_0) >H,
 \end{cases} 
\end{align*}
for $x \in \mathbb{R}^d$ and $\mu\in \mathcal{P}_2(\mathbb{R}^d)$. Then, it can be checked that for all $x \in \mathbb{R}^d$ and $\mu\in \mathcal{P}_2(\mathbb{R}^d)$,
\begin{enumerate}
\item[(1)] $\vert b_H(x,\mu)\vert h^H (x,\mu) \leq C_0$ and  $\vert b_H(x,\mu)\vert^2 h^H (x,\mu) \leq C_0$,
\item[(2)] $\left\langle x, b_H(x,\mu)-b(0,\delta_0) \right\rangle \le C_0 \left(|x|^2+ \mathcal{W}_2^2(\mu,\delta_0)\right)$ due to Condition \textbf{T2},
\end{enumerate}
for some other positive constant $C_0$. Moreover, from  Condition \textbf{T1}, we have
\begin{color}{black}
$$
h^H(x,\mu)\Delta \geq \dfrac{\min\{1,L^{-1}\}\Delta}{1+H},
$$
which implies that
$t_{k+1}^H-t_k^H\geq \frac{\min\{1,L^{-1}\}\Delta}{1+H}.$
\end{color} Therefore, 
\begin{align*}
\lim\limits_{k \to +\infty} t_k^H =+\infty \quad \text{a.s.} 
\end{align*}
Now,  we define by $\underline{t}^H := \max \left\{t_{k}^H: t_{k}^H \leq t\right\}$ the nearest time point before $t$. The continuous interpolant process is defined by
\begin{equation} \label{xinh}\begin{split} 
\widehat{X}_{t}^{i,N,H}:&=\widehat{X}_{\underline{t}^H}^{i,N,H}+b_H\left(\widehat{X}_{\underline{t}^H}^{i,N,H},\mu_{\underline{t}^H}^{\widehat{\boldsymbol{X}}^{N,H}}\right)(t-\underline{t}^H)+\sigma_{\Delta}\left(\widehat{X}_{\underline{t}^H}^{i,N,H},\mu_{\underline{t}^H}^{\widehat{\boldsymbol{X}}^{N,H}}\right)\left(W_{t}^i-W_{\underline{t}^H}^i\right) \\
&\quad+c_{\Delta}\left(\widehat{X}_{\underline{t}^H}^{i,N,H},\mu_{\underline{t}^H}^{\widehat{\boldsymbol{X}}^{N,H}}\right)\left(Z_{t}^i-Z_{\underline{t}^H}^i\right) \\
&=x_0+\int_0^tb_H\left(\widehat{X}_{\underline{s}^H}^{i,N,H},\mu_{\underline{s}^H}^{\widehat{\boldsymbol{X}}^{N,H}}\right)ds+\int_0^t\sigma_{\Delta}\left(\widehat{X}_{\underline{s}^H}^{i,N,H},\mu_{\underline{s}^H}^{\widehat{\boldsymbol{X}}^{N,H}}\right)dW_{s}^i \\
&\quad+\int_{0}^t \int_{\bR_{0}^ d}c_{\Delta}\left(\widehat{X}_{\underline{s}^H}^{i,N,H},\mu_{\underline{s}^H}^{\widehat{\boldsymbol{X}}^{N,H}}\right)z\widetilde{N}^i(d s, d z). 
\end{split}
\end{equation}
Using It\^o's formula, we get
\begin{align}
&\vert \widehat{X}_{t}^{i,N,H}\vert^2=\vert x_0\vert^2 +\int_{0}^{t} \left(2\left\langle \widehat{X}_{s}^{i,N,H}, b_H\left(\widehat{X}_{\underline{s}^H}^{i,N,H},\mu_{\underline{s}^H}^{\widehat{\boldsymbol{X}}^{N,H}}\right)\right\rangle +  \left|\sigma_{\Delta}\left(\widehat{X}_{\underline{s}^H}^{i,N,H},\mu_{\underline{s}^H}^{\widehat{\boldsymbol{X}}^{N,H}}\right)\right|^{2} \right) ds  \notag\\
&\quad +2\int_{0}^{t}\left\langle \widehat{X}_{s}^{i,N,H}, \sigma_{\Delta}\left(\widehat{X}_{\underline{s}^H}^{i,N,H},\mu_{\underline{s}^H}^{\widehat{\boldsymbol{X}}^{N,H}}\right) d W_{s}^i\right\rangle +\int_{0}^t \int_{\bR_{0}^{d}}  \left|c_{\Delta}\left(\widehat{X}_{\underline{s}^H}^{i,N,H},\mu_{\underline{s}^H}^{\widehat{\boldsymbol{X}}^{N,H}}\right)z\right|^2 \nu(dz)ds  \notag\\
&\quad  +\int_{0}^t \int_{\bR_{0}^ d}  \left(\left|c_{\Delta}\left(\widehat{X}_{\underline{s}^H}^{i,N,H},\mu_{\underline{s}^H}^{\widehat{\boldsymbol{X}}^{N,H}}\right) z\right|^{2}+2\left\langle \widehat{X}_{s-}^{i,N,H},c_{\Delta}\left(\widehat{X}_{\underline{s}^H}^{i,N,H},\mu_{\underline{s}^H}^{\widehat{\boldsymbol{X}}^{N,H}}\right) z\right\rangle\right) \widetilde{N}^i(d s, d z) \notag\\
&\leq \vert x_0\vert^2 +\int_{0}^{t} \Bigg(2\left\langle \widehat{X}_{s}^{i,N,H}-\widehat{X}_{\underline{s}^H}^{i,N,H}, b_H\left(\widehat{X}_{\underline{s}^H}^{i,N,H},\mu_{\underline{s}^H}^{\widehat{\boldsymbol{X}}^{N,H}}\right)\right\rangle+2\left\langle \widehat{X}_{\underline{s}^H}^{i,N,H}, b_H\left(\widehat{X}_{\underline{s}^H}^{i,N,H},\mu_{\underline{s}^H}^{\widehat{\boldsymbol{X}}^{N,H}}\right)-b(0,\delta_0)\right\rangle  \notag\\
&\quad+2\left\langle \widehat{X}_{\underline{s}^H}^{i,N,H}, b(0,\delta_0)\right\rangle+  \left|\sigma_{\Delta}\left(\widehat{X}_{\underline{s}^H}^{i,N,H},\mu_{\underline{s}^H}^{\widehat{\boldsymbol{X}}^{N,H}}\right)\right|^{2} +  \left|c_{\Delta}\left(\widehat{X}_{\underline{s}^H}^{i,N,H},\mu_{\underline{s}^H}^{\widehat{\boldsymbol{X}}^{N,H}}\right)\right|^2 \int_{\bR_{0}^{d}} \vert z \vert^2\nu(dz)\Bigg) ds  \notag\\
&\quad+2\int_{0}^{t}\left\langle \widehat{X}_{s}^{i,N,H}-\widehat{X}_{\underline{s}^H}^{i,N,H}, \sigma_{\Delta}\left(\widehat{X}_{\underline{s}^H}^{i,N,H},\mu_{\underline{s}^H}^{\widehat{\boldsymbol{X}}^{N,H}}\right) d W_{s}^i\right\rangle +2\int_{0}^{t}\left\langle \widehat{X}_{\underline{s}^H}^{i,N,H}, \sigma_{\Delta}\left(\widehat{X}_{\underline{s}^H}^{i,N,H},\mu_{\underline{s}^H}^{\widehat{\boldsymbol{X}}^{N,H}}\right) d W_{s}^i\right\rangle  \notag\\
&\quad  +\int_{0}^t \int_{\bR_{0}^ d}  \left(\left|c_{\Delta}\left(\widehat{X}_{\underline{s}^H}^{i,N,H},\mu_{\underline{s}^H}^{\widehat{\boldsymbol{X}}^{N,H}}\right) z\right|^{2}+2\left\langle \widehat{X}_{s-}^{i,N,H}-\widehat{X}_{\underline{s}^H}^{i,N,H},c_{\Delta}\left(\widehat{X}_{\underline{s}^H}^{i,N,H},\mu_{\underline{s}^H}^{\widehat{\boldsymbol{X}}^{N,H}}\right) z\right\rangle\right) \widetilde{N}^i(d s, d z)  \notag\\
&\quad+2\int_{0}^t \int_{\bR_{0}^ d}  \left\langle \widehat{X}_{\underline{s}^H}^{i,N,H},c_{\Delta}\left(\widehat{X}_{\underline{s}^H}^{i,N,H},\mu_{\underline{s}^H}^{\widehat{\boldsymbol{X}}^{N,H}}\right) z\right\rangle \widetilde{N}^i(d s, d z)  \notag\\
&\leq \vert x_0\vert^2 +\int_{0}^{t} \Bigg(2\left\vert \widehat{X}_{s}^{i,N,H}-\widehat{X}_{\underline{s}^H}^{i,N,H}\right\vert \left\vert b_H\left(\widehat{X}_{\underline{s}^H}^{i,N,H},\mu_{\underline{s}^H}^{\widehat{\boldsymbol{X}}^{N,H}}\right)\right\vert+ 2C_0\left(\vert \widehat{X}_{\underline{s}^H}^{i,N,H}\vert^2+ \mathcal{W}_2^2(\mu_{\underline{s}^H}^{\widehat{\boldsymbol{X}}^{N,H}},\delta_0)\right)+\vert \widehat{X}_{\underline{s}^H}^{i,N,H}\vert^2  \notag\\
&\quad +\vert b(0,\delta_0)\vert^2+\dfrac{L^2}{\Delta} +\dfrac{L^2}{\Delta} \int_{\bR_{0}^{d}} \vert z \vert^2\nu(dz)\Bigg) ds +2\int_{0}^{t}\left\langle \widehat{X}_{s}^{i,N,H}-\widehat{X}_{\underline{s}^H}^{i,N,H}, \sigma_{\Delta}\left(\widehat{X}_{\underline{s}^H}^{i,N,H},\mu_{\underline{s}^H}^{\widehat{\boldsymbol{X}}^{N,H}}\right) d W_{s}^i\right\rangle \notag\\
&\quad+2\int_{0}^{t}\left\langle \widehat{X}_{\underline{s}^H}^{i,N,H}, \sigma_{\Delta}\left(\widehat{X}_{\underline{s}^H}^{i,N,H},\mu_{\underline{s}^H}^{\widehat{\boldsymbol{X}}^{N,H}}\right) d W_{s}^i\right\rangle  \notag\\
&\quad  +\int_{0}^t \int_{\bR_{0}^ d}  \left(\left|c_{\Delta}\left(\widehat{X}_{\underline{s}^H}^{i,N,H},\mu_{\underline{s}^H}^{\widehat{\boldsymbol{X}}^{N,H}}\right) z\right|^{2}+2\left\langle \widehat{X}_{s-}^{i,N,H}-\widehat{X}_{\underline{s}^H}^{i,N,H},c_{\Delta}\left(\widehat{X}_{\underline{s}^H}^{i,N,H},\mu_{\underline{s}^H}^{\widehat{\boldsymbol{X}}^{N,H}}\right) z\right\rangle\right) \widetilde{N}^i(d s, d z)  \notag\\
&\quad+2\int_{0}^t \int_{\bR_{0}^ d}  \left\langle \widehat{X}_{\underline{s}^H}^{i,N,H},c_{\Delta}\left(\widehat{X}_{\underline{s}^H}^{i,N,H},\mu_{\underline{s}^H}^{\widehat{\boldsymbol{X}}^{N,H}}\right) z\right\rangle \widetilde{N}^i(d s, d z). \label{Xinh2}
\end{align}
Now, we define $\tau_R:=\inf\{t>0: \max_{i\in\{1,\ldots,N\}}\vert \widehat{X}_{t}^{i,N,H}\vert>R\}$ for each $R>0$ and $\tau:=s \wedge \tau_R$. On the one hand, using equation \eqref{xinh}, Condition \textbf{T3}, the isometry property of stochastic integrals and the fact that $\vert b_H(x,\mu)\vert h^H(x,\mu) \leq C_0$, we get
\begin{align}
&\mathbb{E}\left[\left\vert \widehat{X}_{\tau}^{i,N,H}-\widehat{X}_{\underline{\tau}^H}^{i,N,H}\right\vert^2 \right] \leq 3\bigg(\mathbb{E}\bigg[\left\vert b_H\left(\widehat{X}_{\underline{\tau}^H}^{i,N,H},\mu_{\underline{\tau}^H}^{\widehat{\boldsymbol{X}}^{N,H}}\right)\right\vert^2(\tau-\underline{\tau}^H)^2\bigg] \notag\\
&\quad+\left\vert\sigma_{\Delta}\left(\widehat{X}_{\underline{\tau}^H}^{i,N,H},\mu_{\underline{\tau}^H}^{\widehat{\boldsymbol{X}}^{N,H}}\right)\right\vert^2\left\vert W_{\tau}^i-W_{\underline{\tau}^H}^i\right\vert^2+\left\vert c_{\Delta}\left(\widehat{X}_{\underline{\tau}^H}^{i,N,H},\mu_{\underline{\tau}^H}^{\widehat{\boldsymbol{X}}^{N,H}}\right)\right\vert^2\left\vert Z_{\tau}^i-Z_{\underline{\tau}^H}^i\right\vert^2\bigg) \notag\\
&\leq 3\left(C_0^2\Delta^2+\dfrac{L^2}{\Delta}\mathbb{E}\left[\tau-\underline{\tau}^H\right]+\dfrac{L^2}{\Delta}\int_{\bR_{0}^{d}} \vert z \vert^2\nu(dz)\mathbb{E}\left[\tau-\underline{\tau}^H\right]\right) \notag\\
&\leq 3\left(C_0^2\Delta^2+L^2+L^2\int_{\bR_{0}^{d}} \vert z \vert^2\nu(dz)\right). \label{squaremoment}
\end{align}
On the other hand, Condition \textbf{T3} yields to 
\begin{equation} \label{delta}\begin{split}
&\left\vert \sigma_{\Delta}\left(\widehat{X}_{\underline{s}^H}^{i,N,H},\mu_{\underline{s}^H}^{\widehat{\boldsymbol{X}}^{N,H}}\right)\right\vert\leq \dfrac{L}{\sqrt{\Delta}},\; \; \left\vert c_{\Delta}\left(\widehat{X}_{\underline{s}^H}^{i,N,H},\mu_{\underline{s}^H}^{\widehat{\boldsymbol{X}}^{N,H}}\right)\right\vert\leq \dfrac{L}{\sqrt{\Delta}},\\
&\left\vert \widehat{X}_{\underline{s}^H}^{i,N,H}\right\vert\left\vert \sigma_{\Delta}\left(\widehat{X}_{\underline{s}^H}^{i,N,H},\mu_{\underline{s}^H}^{\widehat{\boldsymbol{X}}^{N,H}}\right)\right\vert\leq \dfrac{L}{\sqrt{\Delta}},\; \; \left\vert \widehat{X}_{\underline{s}^H}^{i,N,H}\right\vert\left\vert c_{\Delta}\left(\widehat{X}_{\underline{s}^H}^{i,N,H},\mu_{\underline{s}^H}^{\widehat{\boldsymbol{X}}^{N,H}}\right)\right\vert\leq \dfrac{L}{\sqrt{\Delta}}.
\end{split}
\end{equation}
Therefore, all the stochastic integrals to the Brownian motion and the compensated Poisson random measure above are square martingales. Thus, their moments are equal to zero.

Moreover, using \textbf{T3}, equation \eqref{xinh}, moment properties of the Brownian motion, the isometry property of stochastic integrals, and the fact that $\vert b_H(x,\mu)\vert^2 h^H(x,\mu) \leq C_0$, we get
\begin{align}
&\mathbb{E}\left[\left\vert \widehat{X}_{s}^{i,N,H}-\widehat{X}_{\underline{s}^H}^{i,N,H}\right\vert \left\vert b_H\left(\widehat{X}_{\underline{s}^H}^{i,N,H},\mu_{\underline{s}^H}^{\widehat{\boldsymbol{X}}^{N,H}}\right)\right\vert \big|\mathcal{F}_{\underline{s}^H}\right] \notag\\
&\leq \mathbb{E}\bigg[\left\vert b_H\left(\widehat{X}_{\underline{s}^H}^{i,N,H},\mu_{\underline{s}^H}^{\widehat{\boldsymbol{X}}^{N,H}}\right)\right\vert^2(s-\underline{s}^H) +\left\vert b_H\left(\widehat{X}_{\underline{s}^H}^{i,N,H},\mu_{\underline{s}^H}^{\widehat{\boldsymbol{X}}^{N,H}}\right)\right\vert \left\vert\sigma_{\Delta}\left(\widehat{X}_{\underline{s}^H}^{i,N,H},\mu_{\underline{s}^H}^{\widehat{\boldsymbol{X}}^{N,H}}\right)\right\vert\left\vert W_{s}^i-W_{\underline{s}^H}^i\right\vert \notag\\
&\quad +\left\vert b_H\left(\widehat{X}_{\underline{s}^H}^{i,N,H},\mu_{\underline{s}^H}^{\widehat{\boldsymbol{X}}^{N,H}}\right)\right\vert \left\vert c_{\Delta}\left(\widehat{X}_{\underline{s}^H}^{i,N,H},\mu_{\underline{s}^H}^{\widehat{\boldsymbol{X}}^{N,H}}\right)\right\vert\left\vert Z_{s}^i-Z_{\underline{s}^H}^i\right\vert \big|\mathcal{F}_{\underline{s}^H}\bigg]  \notag\\
&\leq \left\vert b_H\left(\widehat{X}_{\underline{s}^H}^{i,N,H},\mu_{\underline{s}^H}^{\widehat{\boldsymbol{X}}^{N,H}}\right)\right\vert^2(s-\underline{s}^H)+\dfrac{L}{\sqrt{\Delta}}\left\vert b_H\left(\widehat{X}_{\underline{s}^H}^{i,N,H},\mu_{\underline{s}^H}^{\widehat{\boldsymbol{X}}^{N,H}}\right)\right\vert\sqrt{s-\underline{s}^H}  \notag\\
&\quad+\dfrac{L}{\sqrt{\Delta}}\left(\int_{\bR_{0}^{d}} \vert z \vert^2\nu(dz)\right)^{1/2}\left\vert b_H\left(\widehat{X}_{\underline{s}^H}^{i,N,H},\mu_{\underline{s}^H}^{\widehat{\boldsymbol{X}}^{N,H}}\right)\right\vert\sqrt{s-\underline{s}^H} \notag\\
&\leq C_0\Delta+L\sqrt{C_0}+L\sqrt{C_0}\left(\int_{\bR_{0}^{d}} \vert z \vert^2\nu(dz)\right)^{1/2}. \label{Xb}
\end{align}
This, combined with $\mathbb{E}\left[\mathcal{W}_2^2(\mu_{\underline{s}^H}^{\widehat{\boldsymbol{X}}^{N,H}},\delta_0) {\bf 1}_{s\leq \tau_R}\right] = \mathbb{E}\left[\vert \widehat{X}_{\underline{s}^H}^{i,N,H}\vert^2 {\bf 1}_{s\leq \tau_R}\right]$ for $i\in\{1,\ldots,N\}$, yields that for any $t\in (0,T]$, 
 \begin{align*}
 \mathbb{E}\left[\left\vert \widehat{X}_{t\wedge \tau_R}^{i,N,H}\right\vert^2\right]\leq \vert x_0\vert^2 + \int_{0}^t C\left(L,\Delta,b(0,\delta_0),\mu_2\right) \left(1+\mathbb{E}\left[\left\vert \widehat{X}_{\underline{s}^H}^{i,N,H}\right\vert^2 {\bf 1}_{s\leq \tau_R}\right] \right)ds,
\end{align*}
where $\mu_2:=\int_{\bR_{0}^{d}} \vert z \vert^2\nu(dz)$.

Next, using equation \eqref{xinh} and \eqref{squaremoment}, we get
\begin{align*}
\mathbb{E}\left[\left\vert \widehat{X}_{\underline{s}^H}^{i,N,H}\right\vert^2 {\bf 1}_{s\leq \tau_R}\right] &\leq 2\mathbb{E}\left[\left\vert \widehat{X}_{s}^{i,N,H}\right\vert^2 {\bf 1}_{s\leq \tau_R}\right]+2\mathbb{E}\left[\left\vert \widehat{X}_{s}^{i,N,H}- \widehat{X}_{\underline{s}^H}^{i,N,H}\right\vert^2\right]\\
&\leq  2\mathbb{E}\left[\left\vert \widehat{X}_{s\wedge \tau_R}^{i,N,H}\right\vert^2\right]+ 6\left(C_0^2\Delta^2+L^2+L^2\mu_2\right).
\end{align*}
This implies that for any $t\in (0,T]$, 
\begin{align*}
\mathbb{E}\left[\left\vert \widehat{X}_{t\wedge \tau_R}^{i,N,H}\right\vert^2\right]&\leq  C\left(x_0,L,\Delta,b(0,\delta_0),\mu_2,T\right)\left(1+\int_0^t\mathbb{E}\left[\left\vert \widehat{X}_{s\wedge \tau_R}^{i,N,H}\right\vert^2\right] ds\right),
\end{align*}
which, together with Gronwall's inequality, yields that for any $t\in (0,T]$, 
\begin{align*}
\max_{i\in\{1,\ldots,N\}}\sup_{t\in[0,T]}\mathbb{E}\left[\left\vert \widehat{X}_{t\wedge \tau_R}^{i,N,H}\right\vert^2\right] \leq C\left(x_0,L,\Delta,b(0,\delta_0),\mu_2,T\right).
\end{align*}
Then, using Markov's inequality, we obtain that
\begin{align*}
\pr(\tau_R<T) &\leq \sum_{i=1}^{N}\pr(\vert \widehat{X}_{T\wedge \tau_R}^{i,N,H}\vert>R) =N\pr(\vert \widehat{X}_{T\wedge \tau_R}^{1,N,H}\vert>R)\\
&\leq N\dfrac{\mathbb{E}[\vert \widehat{X}_{T\wedge \tau_R}^{1,N,H}\vert^2]}{R^2}\\
&\leq  N\dfrac{C\left(x_0,L,\Delta,b(0,\delta_0),\mu_2,T\right)}{R^2},
\end{align*}
which tends to zero as $R \uparrow \infty$. Therefore, $\tau_R \uparrow \infty$
as $R \uparrow \infty$. Then due to Fatou's lemma, we get
\begin{align}\label{supE}
\max_{i\in\{1,\ldots,N\}}\sup_{t\in[0,T]}\mathbb{E}\left[\left\vert \widehat{X}_{t}^{i,N,H}\right\vert^2\right] \leq C\left(x_0,L,\Delta,b(0,\delta_0),\mu_2,T\right).
\end{align}
Now, from \eqref{Xinh2}, we get that for any $t\in (0,T]$,
\begin{align*}
\vert \widehat{X}_{t}^{i,N,H}\vert^2&\leq \vert x_0\vert^2 +\int_{0}^{T} \Bigg(2\left\vert \widehat{X}_{s}^{i,N,H}-\widehat{X}_{\underline{s}^H}^{i,N,H}\right\vert \left\vert b_H\left(\widehat{X}_{\underline{s}^H}^{i,N,H},\mu_{\underline{s}^H}^{\widehat{\boldsymbol{X}}^{N,H}}\right)\right\vert+ 2C_0\left(\vert \widehat{X}_{\underline{s}^H}^{i,N,H}\vert^2+ \mathcal{W}_2^2(\mu_{\underline{s}^H}^{\widehat{\boldsymbol{X}}^{N,H}},\delta_0)\right)\\
&\quad +\vert \widehat{X}_{\underline{s}^H}^{i,N,H}\vert^2+\vert b(0,\delta_0)\vert^2+\dfrac{L^2}{\Delta} +\dfrac{L^2}{\Delta} \mu_2\Bigg) ds +2\int_{0}^{t}\left\langle \widehat{X}_{s}^{i,N,H}-\widehat{X}_{\underline{s}^H}^{i,N,H}, \sigma_{\Delta}\left(\widehat{X}_{\underline{s}^H}^{i,N,H},\mu_{\underline{s}^H}^{\widehat{\boldsymbol{X}}^{N,H}}\right) d W_{s}^i\right\rangle \\
&\quad+2\int_{0}^{t}\left\langle \widehat{X}_{\underline{s}^H}^{i,N,H}, \sigma_{\Delta}\left(\widehat{X}_{\underline{s}^H}^{i,N,H},\mu_{\underline{s}^H}^{\widehat{\boldsymbol{X}}^{N,H}}\right) d W_{s}^i\right\rangle\\
&\quad  +\int_{0}^t \int_{\bR_{0}^ d}  \left(\left|c_{\Delta}\left(\widehat{X}_{\underline{s}^H}^{i,N,H},\mu_{\underline{s}^H}^{\widehat{\boldsymbol{X}}^{N,H}}\right) z\right|^{2}+2\left\langle \widehat{X}_{s-}^{i,N,H}-\widehat{X}_{\underline{s}^H}^{i,N,H},c_{\Delta}\left(\widehat{X}_{\underline{s}^H}^{i,N,H},\mu_{\underline{s}^H}^{\widehat{\boldsymbol{X}}^{N,H}}\right) z\right\rangle\right) \widetilde{N}^i(d s, d z)\\
&\quad+2\int_{0}^t \int_{\bR_{0}^ d}  \left\langle \widehat{X}_{\underline{s}^H}^{i,N,H},c_{\Delta}\left(\widehat{X}_{\underline{s}^H}^{i,N,H},\mu_{\underline{s}^H}^{\widehat{\boldsymbol{X}}^{N,H}}\right) z\right\rangle \widetilde{N}^i(d s, d z).
\end{align*}
This, combined with \eqref{Xb}, the fact that  $\mathbb{E}\left[\mathcal{W}_2^2(\mu_{\underline{s}^H}^{\widehat{\boldsymbol{X}}^{N,H}},\delta_0)\right] = \mathbb{E}\left[|\widehat{X}_{\underline{s}^H}^{i,N,H}|^2\right]$ and \eqref{supE}, deduces that
\begin{align}\label{C0}
\max_{i\in\{1,\ldots,N\}}\mathbb{E}\left[\sup_{t\in[0,T]}\left\vert \widehat{X}_{t}^{i,N,H}\right\vert^2\right] \leq C\left(x_0,L,\Delta,b(0,\delta_0),\mu_2,T\right)=:\overline{C_0}.
\end{align}
Observe that
\begin{align*}
\left\{t_k\leq T\right\}&=\left\{t_k\leq T, \;\; \max_{i\in\{1,\ldots,N\}}\sup_{t\in[0,T]}\left(\left\vert \widehat{X}_{t}^{i,N,H} \right\vert^{\beta_1}+\mathcal{W}_2^{\beta_2}\left(\mu_{t}^{\widehat{\boldsymbol{X}}^{N,H}},\delta_0\right)\right) \leq \dfrac{H}{2}\right\}\\
&\qquad\cup \left\{\max_{i\in\{1,\ldots,N\}}\sup_{t\in[0,T]}\left(\left\vert \widehat{X}_{t}^{i,N,H} \right\vert^{\beta_1}+\mathcal{W}_2^{\beta_2}\left(\mu_{t}^{\widehat{\boldsymbol{X}}^{N,H}},\delta_0\right)\right) > \dfrac{H}{2}\right\}\\
&\subset \left\{t_k^H\leq T\right\} \cup\left\{\max_{i\in\{1,\ldots,N\}}\sup_{t\in[0,T]}\left(\left\vert \widehat{X}_{t}^{i,N,H} \right\vert^{\beta_1}+\mathcal{W}_2^{\beta_2}\left(\mu_{t}^{\widehat{\boldsymbol{X}}^{N,H}},\delta_0\right)\right) > \dfrac{H}{2}\right\}.
\end{align*}
Then, using $\mathbb{E}\left[\sup_{t\in[0,T]}\mathcal{W}_2^2(\mu_{t}^{\widehat{\boldsymbol{X}}^{N,H}},\delta_0)\right] \leq \mathbb{E}\left[\sup_{t\in[0,T]}|\widehat{X}_{t}^{i,N,H}|^2\right]$ for  $i\in\{1,\ldots,N\}$ and Markov's inequality, we get that for any $H>0$,
\begin{align*}
\pr(t_k\leq T)&\leq \pr\left(t_k^H\leq T\right)+\pr\left(\max_{i\in\{1,\ldots,N\}}\sup_{t\in[0,T]}\left(\left\vert \widehat{X}_{t}^{i,N,H} \right\vert^{\beta_1}+\mathcal{W}_2^{\beta_2}\left(\mu_{t}^{\widehat{\boldsymbol{X}}^{N,H}},\delta_0\right)\right) > \dfrac{H}{2}\right)\\
&\leq \pr\left(t_k^H\leq T\right)+\pr\left(\max_{i\in\{1,\ldots,N\}}\sup_{t\in[0,T]}\left\vert \widehat{X}_{t}^{i,N,H} \right\vert^{\beta_1}> \dfrac{H}{4}\right)+\pr\left(\max_{i\in\{1,\ldots,N\}}\sup_{t\in[0,T]}\mathcal{W}_2^{\beta_2}\left(\mu_{t}^{\widehat{\boldsymbol{X}}^{N,H}},\delta_0\right)> \dfrac{H}{4}\right)\\
&= \pr\left(t_k^H\leq T\right)+\pr\left(\max_{i\in\{1,\ldots,N\}}\sup_{t\in[0,T]}\left\vert \widehat{X}_{t}^{i,N,H} \right\vert^{2}> \Big(\dfrac{H}{4}\Big)^{2/\beta_1}\right)\\
&\quad+\pr\left(\max_{i\in\{1,\ldots,N\}}\sup_{t\in[0,T]}\mathcal{W}_2^{2}\left(\mu_{t}^{\widehat{\boldsymbol{X}}^{N,H}},\delta_0\right)> \Big(\dfrac{H}{4}\Big)^{2/\beta_2}\right)\\
&\leq \pr\left(t_k^H\leq T\right)+\left(\dfrac{4}{H}\right)^{2/\beta_1}\sum_{i=1}^{N}\mathbb{E}\left[\sup_{t\in[0,T]}\left\vert \widehat{X}_{t}^{i,N,H}\right\vert^2\right]+\left(\dfrac{4}{H}\right)^{2/\beta_2}\sum_{i=1}^{N}\mathbb{E}\left[\sup_{t\in[0,T]}\mathcal{W}_2^2(\mu_{t}^{\widehat{\boldsymbol{X}}^{N,H}},\delta_0)\right]\\
&\leq \pr\left(t_k^H\leq T\right)+\left(\dfrac{4}{H}\right)^{2/\beta_1}\sum_{i=1}^{N}\mathbb{E}\left[\sup_{t\in[0,T]}\left\vert \widehat{X}_{t}^{i,N,H}\right\vert^2\right]+\left(\dfrac{4}{H}\right)^{2/\beta_2}\sum_{i=1}^{N}\mathbb{E}\left[\sup_{t\in[0,T]}\left\vert \widehat{X}_{t}^{i,N,H}\right\vert^2\right]\\
&\leq \pr\left(t_k^H\leq T\right)+\left(\left(\dfrac{4}{H}\right)^{2/\beta_1}+\left(\dfrac{4}{H}\right)^{2/\beta_2}\right)N\overline{C_0}.
\end{align*}
Then, let $k\uparrow\infty$ and recall that $\lim\limits_{k \to +\infty} t_k^H =+\infty \quad \text{a.s.}$, we have that for any $H>0$,
\begin{align*}
\limsup_{k\to\infty}\pr(t_k\leq T) \leq \left(\left(\dfrac{4}{H}\right)^{2/\beta_1}+\left(\dfrac{4}{H}\right)^{2/\beta_2}\right)N\overline{C_0}.
\end{align*}
Then, letting $H\uparrow\infty$, we get
\begin{align*}
\lim_{k\to\infty}\pr(t_k\leq T) =0.
\end{align*}
Therefore, $t_k \to \infty$ in probability as $k\uparrow\infty$. Since $(t_k)_{k\geq 0}$ is an increasing sequence, we have 
\begin{align*}
\lim\limits_{k \to +\infty} t_k =+\infty \quad \text{a.s.} 
\end{align*}
Thus, the result follows.
\end{proof} 

\begin{Rem}
Our approach to showing Proposition \ref{dinh ly 4} is quite different from that of \cite{FG}. Indeed, in \cite{FG}, the auxiliary process $\widehat{X}$ is constructed by the projection method, which makes its analysis very hard in the case of McKean-Vlasov SDEs. In our proof, $\widehat{X}$ is constructed as an It\^o process, which allows us to apply It\^o's formula to $\vert\widehat{X}\vert^2$. This fact helps to greatly simplify our proof.
\end{Rem}

Let all assumptions of Proposition \ref{dinh ly 4} be satisfied, we define by $\underline{t} := \max \left\{t_{n}: t_{n} \leq t\right\}$ the nearest time point before $t$, and by $N_{t} := \max \left\{n: t_{n} \leq t\right\}$ the number of timesteps approximation up to time $t$. Observe that $\underline{t}$ is a stopping time. Thus, we define the standard continuous interpolation as
\begin{equation}\label{EM3}
	\widehat{X}_{t}^{i,N}=\widehat{X}_{\underline{t}}^{i,N}+b\left(\widehat{X}_{\underline{t}}^{i,N},\mu_{\underline{t}}^{\widehat{\boldsymbol{X}}^{N}}\right)(t-\underline{t})+\sigma_{\Delta}\left(\widehat{X}_{\underline{t}}^{i,N},\mu_{\underline{t}}^{\widehat{\boldsymbol{X}}^{N}}\right)\left(W_{t}^i-W_{\underline{t}}^i\right)+c_{\Delta}\left(\widehat{X}_{\underline{t}}^{i,N},\mu_{\underline{t}}^{\widehat{\boldsymbol{X}}^{N}}\right)\left(Z_{t}^i-Z_{\underline{t}}^i\right).
\end{equation}
Hence, $\widehat{X}^{i,N}=(\widehat{X}_t^{i,N})_{t \geq 0}$ is the solution to the following SDE with jumps
\begin{equation}\label{EM2}
d \widehat{X}_{t}^{i,N}= b \left(\widehat{X}_{\underline{t}}^{i,N},\mu_{\underline{t}}^{\widehat{\boldsymbol{X}}^{N}}\right) dt+\sm_{\Delta} \left(\widehat{X}_{\underline{t}}^{i,N},\mu_{\underline{t}}^{\widehat{\boldsymbol{X}}^{N}}\right)  dW_{t}^i +c_{\Delta} \left( \widehat{X}_{\underline{t}}^{i,N},\mu_{\underline{t}}^{\widehat{\boldsymbol{X}}^{N}} \right)  dZ_{t}^i, \qquad \widehat{X}_{0}^{i,N} = x_0,
\end{equation}
whose integral equation has the following form 
\begin{equation*} 
\widehat{X}_{t}^{i,N}=x_0+\int_0^t b\left(\widehat{X}_{\underline{s}}^{i,N},\mu_{\underline{s}}^{\widehat{\boldsymbol{X}}^{N}}\right)ds+\int_0^t \sigma_{\Delta}\left(\widehat{X}_{\underline{s}}^{i,N},\mu_{\underline{s}}^{\widehat{\boldsymbol{X}}^{N}}\right)dW_s^i + \int_0^t\int_{\mathbb{R}_0^d} c_{\Delta}\left(\widehat{X}_{\underline{s}}^{i,N},\mu_{\underline{s}}^{\widehat{\boldsymbol{X}}^{N}}\right)z\widetilde{N}^i(d s, d z).
\end{equation*}

In all what follows, the following classical moment estimate for the Brownian motion $W$ will be useful
\begin{equation} \label{markov} 
\mathbb{E}\left[\left|W_t-W_{\underline{t}^H}\right|^r\big| \mathcal{F}_{\underline{t}^H}\right] \le   C_r(t-\underline{t}^H)^{r/2},
\end{equation} 
for any $t>0$, $r>0$ and some positive constant $C_r$. 

\subsection{Moments of the tamed-adaptive Euler-Maruyama scheme}


	We now state the first estimate on the moments of $\widehat{X}^{i,N}=(\widehat{X}_t^{i,N})_{t \geq 0}$. 
	\begin{Lem}\label{menh de 3} 
		Assume Conditions \textbf{T1}-\textbf{T3} and that Condition \textbf{A5} holds for $q=2p_0$. Then for any $p\in [1,2p_0]$ and  $T>0$, there exists a positive constant 
		$C\left(p,x_0,L,\Delta,b(0,\delta_0),\mu_2,\mu_{p/2},T\right)$ with $\mu_{p/2}:=\int_{\bR_{0}^{d}} \vert z \vert^{p/2}\nu(dz)$ such that  
		\begin{align*}
		\max_{i\in\{1,\ldots,N\}}\bE\left[\sup_{t\in[0,T]}|\widehat{X}_t^{i,N}|^p \right] \leq C\left(p,x_0,L,\Delta,b(0,\delta_0),\mu_2,\mu_{p/2},T\right).
		\end{align*}
		\end{Lem}
	
	\begin{proof}  
	 Recall that the process $\widehat{X}^{i,N,H}=(\widehat{X}_{t}^{i,N,H})_{t\geq 0}$ is defined in \eqref{EM00} and \eqref{xinh}. Using Markov's inequality, the estimate $\mathbb{E}\left[\sup_{t\in[0,T]}\mathcal{W}_2^2(\mu_{t}^{\widehat{\boldsymbol{X}}^{N,H}},\delta_0)\right] \leq \mathbb{E}\left[\sup_{t\in[0,T]}|\widehat{X}_{t}^{i,N,H}|^2\right]$ and \eqref{C0}, we obtain that for any $T>0$, $i\in\{1,\ldots,N\}$ and $H>0$,
		\begin{align*}
		&\pr\left(\sup_{t\in[0,T]}|\widehat{X}_t^{i,N}|\neq \sup_{t\in[0,T]}|\widehat{X}_t^{i,N,H}|\right) \\
		&\leq \pr\left(\sup_{t\in[0,T]}\left(\left\vert \widehat{X}_{t}^{i,N,H} \right\vert^{\beta_1}+\mathcal{W}_2^{\beta_2}\left(\mu_{t}^{\widehat{\boldsymbol{X}}^{N,H}},\delta_0\right)\right)>H\right)\\
		&\leq \pr\left(\sup_{t\in[0,T]}\left\vert \widehat{X}_{t}^{i,N,H} \right\vert^{\beta_1}> \dfrac{H}{2}\right)+\pr\left(\sup_{t\in[0,T]}\mathcal{W}_2^{\beta_2}\left(\mu_{t}^{\widehat{\boldsymbol{X}}^{N,H}},\delta_0\right)> \dfrac{H}{2}\right)\\
		&=\pr\left(\sup_{t\in[0,T]}\left\vert \widehat{X}_{t}^{i,N,H} \right\vert^{2}> \Big(\dfrac{H}{2}\Big)^{2/\beta_1}\right)+\pr\left(\sup_{t\in[0,T]}\mathcal{W}_2^{2}\left(\mu_{t}^{\widehat{\boldsymbol{X}}^{N,H}},\delta_0\right)> \Big(\dfrac{H}{2}\Big)^{2/\beta_2}\right)\\
		&\leq\left(\dfrac{2}{H}\right)^{2/\beta_1}\mathbb{E}\left[\sup_{t\in[0,T]}\left\vert \widehat{X}_{t}^{i,N,H}\right\vert^2\right]+\left(\dfrac{2}{H}\right)^{2/\beta_2}\mathbb{E}\left[\sup_{t\in[0,T]}\mathcal{W}_2^2(\mu_{t}^{\widehat{\boldsymbol{X}}^{N,H}},\delta_0)\right]\\
		&\leq \left(\dfrac{2}{H}\right)^{2/\beta_1}\mathbb{E}\left[\sup_{t\in[0,T]}\left\vert \widehat{X}_{t}^{i,N,H}\right\vert^2\right]+\left(\dfrac{2}{H}\right)^{2/\beta_2}\mathbb{E}\left[\sup_{t\in[0,T]}\left\vert \widehat{X}_{t}^{i,N,H}\right\vert^2\right]\\
		&\leq \left(\left(\dfrac{2}{H}\right)^{2/\beta_1}+\left(\dfrac{2}{H}\right)^{2/\beta_2}\right)\overline{C_0},
		\end{align*}
		which tends to zero as $H \uparrow \infty$. This implies that $\sup_{t\in[0,T]}|\widehat{X}_t^{i,N,H}|\to\sup_{t\in[0,T]}|\widehat{X}_t^{i,N}|$ in probability as $H \uparrow \infty$. Thus, for any $T>0$ and $i\in\{1,\ldots,N\}$, there exists a sequence $\{H_n\}_{n\geq 1}$ that tends to infinity such that $\sup_{t\in[0,T]}|\widehat{X}_t^{i,N,H_n}|\to \sup_{t\in[0,T]}|\widehat{X}_t^{i,N}|$ a.s. as $n \uparrow \infty$.
		
		Now, from \eqref{Xinh2}, we have that for any $t>0$, $i\in\{1,\ldots,N\}$ and $H>0$,
		\begin{align}
		&\vert \widehat{X}_{t}^{i,N,H}\vert^2\leq \vert x_0\vert^2 +\int_{0}^{T} \Bigg(2\left\vert \widehat{X}_{s}^{i,N,H}-\widehat{X}_{\underline{s}^H}^{i,N,H}\right\vert \left\vert b_H\left(\widehat{X}_{\underline{s}^H}^{i,N,H},\mu_{\underline{s}^H}^{\widehat{\boldsymbol{X}}^{N,H}}\right)\right\vert+ 2C_0\mathcal{W}_2^2(\mu_{\underline{s}^H}^{\widehat{\boldsymbol{X}}^{N,H}},\delta_0)\notag\\
		&\quad +\vert b(0,\delta_0)\vert^2+\dfrac{L^2}{\Delta} +\dfrac{L^2}{\Delta} \int_{\bR_{0}^{d}} \vert z \vert^2\nu(dz)\Bigg) ds +2\int_{0}^{t}\left\langle \widehat{X}_{s}^{i,N,H}-\widehat{X}_{\underline{s}^H}^{i,N,H}, \sigma_{\Delta}\left(\widehat{X}_{\underline{s}^H}^{i,N,H},\mu_{\underline{s}^H}^{\widehat{\boldsymbol{X}}^{N,H}}\right) d W_{s}^i\right\rangle \notag\\
		&\quad+2\int_{0}^{t}\left\langle \widehat{X}_{\underline{s}^H}^{i,N,H}, \sigma_{\Delta}\left(\widehat{X}_{\underline{s}^H}^{i,N,H},\mu_{\underline{s}^H}^{\widehat{\boldsymbol{X}}^{N,H}}\right) d W_{s}^i\right\rangle  \notag\\
		&\quad  +\int_{0}^t \int_{\bR_{0}^ d}  \left(\left|c_{\Delta}\left(\widehat{X}_{\underline{s}^H}^{i,N,H},\mu_{\underline{s}^H}^{\widehat{\boldsymbol{X}}^{N,H}}\right) z\right|^{2}+2\left\langle \widehat{X}_{s-}^{i,N,H}-\widehat{X}_{\underline{s}^H}^{i,N,H},c_{\Delta}\left(\widehat{X}_{\underline{s}^H}^{i,N,H},\mu_{\underline{s}^H}^{\widehat{\boldsymbol{X}}^{N,H}}\right) z\right\rangle\right) \widetilde{N}^i(d s, d z)  \notag\\
		&\quad+2\int_{0}^t \int_{\bR_{0}^ d}  \left\langle \widehat{X}_{\underline{s}^H}^{i,N,H},c_{\Delta}\left(\widehat{X}_{\underline{s}^H}^{i,N,H},\mu_{\underline{s}^H}^{\widehat{\boldsymbol{X}}^{N,H}}\right) z\right\rangle \widetilde{N}^i(d s, d z) +(2C_0+1) \int_{0}^{t}\vert \widehat{X}_{\underline{s}^H}^{i,N,H}\vert^2  ds.  \label{l1}
		\end{align}
		First, we have
		\begin{align}
	&\mathbb{E}\left[\left\vert \widehat{X}_{s}^{i,N,H}-\widehat{X}_{\underline{s}^H}^{i,N,H}\right\vert^{p/2} \left\vert b_H\left(\widehat{X}_{\underline{s}^H}^{i,N,H},\mu_{\underline{s}^H}^{\widehat{\boldsymbol{X}}^{N,H}}\right)\right\vert^{p/2} \big|\mathcal{F}_{\underline{s}^H}\right] \leq C(p)\Bigg(\mathbb{E}\bigg[\left\vert b_H\left(\widehat{X}_{\underline{s}^H}^{i,N,H},\mu_{\underline{s}^H}^{\widehat{\boldsymbol{X}}^{N,H}}\right)\right\vert^p(s-\underline{s}^H)^{p/2} \notag\\
		&\quad  +\left\vert b_H\left(\widehat{X}_{\underline{s}^H}^{i,N,H},\mu_{\underline{s}^H}^{\widehat{\boldsymbol{X}}^{N,H}}\right)\right\vert^{p/2} \left\vert\sigma_{\Delta}\left(\widehat{X}_{\underline{s}^H}^{i,N,H},\mu_{\underline{s}^H}^{\widehat{\boldsymbol{X}}^{N,H}}\right)\right\vert^{p/2}\left\vert W_{s}^i-W_{\underline{s}^H}^i\right\vert^{p/2} \notag\\
		&\quad +\left\vert b_H\left(\widehat{X}_{\underline{s}^H}^{i,N,H},\mu_{\underline{s}^H}^{\widehat{\boldsymbol{X}}^{N,H}}\right)\right\vert^{p/2} \left\vert c_{\Delta}\left(\widehat{X}_{\underline{s}^H}^{i,N,H},\mu_{\underline{s}^H}^{\widehat{\boldsymbol{X}}^{N,H}}\right)\right\vert^{p/2}\left\vert Z_{s}^i-Z_{\underline{s}^H}^i\right\vert^{p/2} \big|\mathcal{F}_{\underline{s}^H}\bigg]  \Bigg) \notag\\
		&\leq C(p)\Bigg(\left\vert b_H\left(\widehat{X}_{\underline{s}^H}^{i,N,H},\mu_{\underline{s}^H}^{\widehat{\boldsymbol{X}}^{N,H}}\right)\right\vert^p(s-\underline{s}^H)^{p/2}+\Big(\dfrac{L}{\sqrt{\Delta}}\Big)^{p/2}\left(\left\vert b_H\left(\widehat{X}_{\underline{s}^H}^{i,N,H},\mu_{\underline{s}^H}^{\widehat{\boldsymbol{X}}^{N,H}}\right)\right\vert^2 (s-\underline{s}^H) \right)^{p/4} \notag\\
		&\quad+\Big(\dfrac{L}{\sqrt{\Delta}}\Big)^{p/2}\int_{\bR_{0}^{d}} \vert z \vert^{p/2}\nu(dz) (s-\underline{s}^H) \Bigg) \notag\\
		&\leq C(p)\left( C_0^{p/2}\Delta^{p/2}+L^{p/2}C_0^{p/4}+\Big(\dfrac{L}{\sqrt{\Delta}}\Big)^{p/2}\int_{\bR_{0}^{d}} \vert z \vert^{p/2}\nu(dz) \right). \label{l2}
		\end{align}
		Second, we have
		\begin{align}
		&\mathbb{E}\left[\left\vert \widehat{X}_{\tau}^{i,N,H}-\widehat{X}_{\underline{\tau}^H}^{i,N,H}\right\vert^{p/2} \right] \leq 3^{p/2-1}\Bigg(\mathbb{E}\bigg[\left\vert b_H\left(\widehat{X}_{\underline{\tau}^H}^{i,N,H},\mu_{\underline{\tau}^H}^{\widehat{\boldsymbol{X}}^{N,H}}\right)\right\vert^{p/2}(\tau-\underline{\tau}^H)^{p/2}\bigg] \notag \\
		&\quad+\left\vert\sigma_{\Delta}\left(\widehat{X}_{\underline{\tau}^H}^{i,N,H},\mu_{\underline{\tau}^H}^{\widehat{\boldsymbol{X}}^{N,H}}\right)\right\vert^{p/2}\left\vert W_{\tau}^i-W_{\underline{\tau}^H}^i\right\vert^{p/2}+\left\vert c_{\Delta}\left(\widehat{X}_{\underline{\tau}^H}^{i,N,H},\mu_{\underline{\tau}^H}^{\widehat{\boldsymbol{X}}^{N,H}}\right)\right\vert^{p/2}\left\vert Z_{\tau}^i-Z_{\underline{\tau}^H}^i\right\vert^{p/2}\Bigg) \notag\\
		&\leq C_p\left(C_0^{p/2}\Delta^{p/2}+\Big(\dfrac{L}{\sqrt{\Delta}}\Big)^{p/2}\mathbb{E}\left[(\tau-\underline{\tau}^H)^{p/4}\right]+\Big(\dfrac{L}{\sqrt{\Delta}}\Big)^{p/2}\int_{\bR_{0}^{d}} \vert z \vert^{p/2}\nu(dz) \mathbb{E}\left[\tau-\underline{\tau}^H\right]\right) \notag\\
		&\leq C_p\left(C_0^{p/2}\Delta^{p/2}+L^{p/2}+\Big(\dfrac{L}{\sqrt{\Delta}}\Big)^{p/2}\int_{\bR_{0}^{d}} \vert z \vert^{p/2}\nu(dz)\right).  \label{l3}
		\end{align}
		Therefore, from \eqref{l1},  \eqref{l2}, \eqref{l3}, \eqref{delta}, the estimate $\mathbb{E}\left[\mathcal{W}_2^2(\mu_{\underline{s}^H}^{\widehat{\boldsymbol{X}}^{N,H}},\delta_0)\right] = \mathbb{E}\left[|\widehat{X}_{\underline{s}^H}^{i,N,H}|^2\right] \leq \overline{C_0}$ and the Burkholder-Davis-Gundy inequality, we get that for any $t\in (0,T]$, 
		\begin{align*}
		\mathbb{E}\left[\sup_{u\in[0,t]}|\widehat{X}_u^{i,N,H}|^p\right] &\leq  C\left(p,x_0,L,\Delta,b(0,\delta_0),\mu_2,\mu_{p/2},T\right)+ (2C_0+1)^{p/2} \mathbb{E}\left[\left(\int_{0}^{t}\vert \widehat{X}_{\underline{s}^H}^{i,N,H}\vert^2  ds\right)^{p/2}\right]\\
		&\leq  \widetilde{C}_0+ (2C_0+1)^{p/2} t^{p/2-1}\int_{0}^{t}\mathbb{E}\left[\vert \widehat{X}_{\underline{s}^H}^{i,N,H}\vert^p\right]  ds\\
		&\leq  \widetilde{C}_0+ (2C_0+1)^{p/2} t^{p/2-1}\int_{0}^{t}\mathbb{E}\left[\sup_{u\in[0,s]}|\widehat{X}_u^{i,N,H}|^p\right] ds,
		\end{align*}
		where $\widetilde{C}_0:=C\left(p,x_0,L,\Delta,b(0,\delta_0),\mu_2,\mu_{p/2},T\right)$ with $\mu_{p/2}:=\int_{\bR_{0}^{d}} \vert z \vert^{p/2}\nu(dz)$.
		
		 This, combined with Gronwall's inequality, deduces that
		 \begin{align}\label{hn}
		 \max_{i\in\{1,\ldots,N\}}\mathbb{E}\left[\sup_{t\in[0,T]}|\widehat{X}_t^{i,N,H}|^p\right]  \leq \widetilde{C}_1,
		 \end{align}
	 where  the constant $\widetilde{C}_1$ does not depend on $H$.  
	 
	Therefore, choosing $H=H_n$ in \eqref{hn} and letting $n\uparrow\infty$, combined with Fatou's lemma and the fact that $\sup_{t\in[0,T]}|\widehat{X}_t^{i,N,H_n}|\to \sup_{t\in[0,T]}|\widehat{X}_t^{i,N}|$ a.s. as $n \uparrow \infty$, we obtain that
		\begin{align*}
		\max_{i\in\{1,\ldots,N\}}\bE\left[\sup_{t\in[0,T]}|\widehat{X}_t^{i,N}|^p \right] \leq \widetilde{C}_1,
		\end{align*}
		which finishes the desired proof.
	\end{proof}

	Next, we are going to show that the moments of $\widehat{X}_t^{i,N}$ depend on $t$. For this, we need to introduce the following conditions.
	\begin{itemize}
		\item[\textbf{T4.}] There exists a positive constant $L$ such that $|\sm_{\Delta}(x,\mu)| \le  |\sm(x,\mu)|$ \; and \;  $|c_\Delta(x,\mu)| \leq |c(x,\mu)|$ for all $x \in \mathbb{R}^d$ and $\mu\in \mathcal{P}_2(\mathbb{R}^d)$;
		\item[\textbf{T5.}] For some integer $p_0 \in [2, +\infty)$, there exist constants $\widetilde{L_3}>0$, $\widetilde{\gamma}_1 \in \mathbb{R}$, $\widetilde{\gamma}_2> 0$, $\widetilde{\eta}\geq 0, \upsilon > 0$ such that for all $x \in \mathbb{R}^d$ and $\mu\in \mathcal{P}_2(\mathbb{R}^d)$,
		\begin{equation}\label{cdelta}
			|c_\Delta(x,\mu)|\leq \widetilde{L_3}\left(1+|x|+ \mathcal{W}_2(\mu,\delta_0)\right),
		\end{equation}
		and
		\begin{equation} \label{sigmacdelta}\begin{split}
			&\langle x,b(x,\mu) \rangle  + \dfrac{p_0-1}{2}  |\sm_{\Delta}(x,\mu)|^2 \\
			&+|c_{\Delta}(x,\mu)|^2 \int_{\bR_{0}^d}\left(\dfrac{\vert z\vert^2}{2}+\dfrac{1}{\widetilde{L_3}^2}\left(\left(1+\vert z \vert (\widetilde{L_3}+\upsilon)\right)^{p_0-1}-1-\vert z\vert(\widetilde{L_3}+ \upsilon)\right)\bigg(\vert z\vert\Big(\dfrac{\widetilde{L_3}}{2}+\upsilon\Big)+\upsilon\bigg)\right)\nu(d z) \\
			&\leq \widetilde{\gamma}_1  |x|^2+\widetilde{\gamma}_2\mathcal{W}_2^2(\mu,\delta_0)+\widetilde{\eta}.
		\end{split}
		\end{equation}
	\end{itemize}

\begin{Rem} \label{RMbsmdelta}
 Observe that \eqref{sigmacdelta} of Condition \textbf{T5} yields to 
	\begin{equation*} \begin{split}
	&\langle x,b(x,\mu) \rangle  + \dfrac{p-1}{2}  |\sm_{\Delta}(x,\mu)|^2 \\
	&+|c_{\Delta}(x,\mu)|^2 \int_{\bR_{0}^d}\left(\dfrac{\vert z\vert^2}{2}+\dfrac{1}{\widetilde{L_3}^2}\left(\left(1+\vert z \vert (\widetilde{L_3}+\upsilon)\right)^{p-1}-1-\vert z\vert(\widetilde{L_3}+\upsilon)\right)\bigg(\vert z\vert\Big(\dfrac{\widetilde{L_3}}{2}+\upsilon\Big)+\upsilon\bigg)\right)\nu(d z) \\
	&\leq \widetilde{\gamma}_1  |x|^2+\widetilde{\gamma}_2\mathcal{W}_2^2(\mu,\delta_0)+\widetilde{\eta}.
	\end{split}
	\end{equation*}
	for any $p\in[2,p_0]$, $x \in \br^d$ and $\mu\in \mathcal{P}_2(\mathbb{R}^d)$.
\end{Rem}

In the following, we state an estimate for $L^2$-norm of the approximate solution. 
\begin{Lem}\label{indu2} Assume  \textbf{T1}--\textbf{T5} and \textbf{A5} with $p_0=2$. 
	Then, there exists a positive constant $C=C(x_0,\widetilde{\gamma}_1,\widetilde{\gamma}_2,\widetilde{\eta},L,\widetilde{L_3})$ which depends neither on $\Delta$  nor on $t$ such that for any $t\geq 0$, 
	\begin{align*}
	\max_{i\in\{1,\ldots,N\}}\left(\bE \left[|\widehat{X}_t^{i,N}|^{2} \right] \vee \bE \left[|\widehat{X}_{\underline{t}}^{i,N}|^{2}\right]\right) \le \left\{ \begin{array}{l l}
	Ce^{{2\widetilde{\gamma} t}}\quad &\text{ if\;\; } \widetilde{\gamma} > 0,\\
	C(1+t)\quad &\text{ if \;\;} \widetilde{\gamma}  = 0, \\
	C \quad &\text{ if \;\;} \widetilde{\gamma}  < 0,
	\end{array} \right. 
	\end{align*}
	where $\widetilde{\gamma} =\widetilde{\gamma}_1 +\widetilde{\gamma}_2$.
\end{Lem}
\begin{proof}
	Proceeding as in \eqref{ulxt2} by applying It\^o's formula to $e^{-2\widetilde{\gamma} t}\vert\widehat{X}_t^{i,N}\vert^2$, we get that for any $i\in\{1,\ldots,N\}$,
	\begin{align}
	&e^{-2\widetilde{\gamma} t}|\widehat{X}_t^{i,N}|^2 
	\leq |x_0|^2+ 2\int_{0}^{t} e^{-2\widetilde{\gamma} s}\Bigg(- {\widetilde{\gamma}} |\widehat{X}_s^{i,N}|^2 + \langle \widehat{X}_s^{i,N}, b(\widehat{X}_{\underline{s}}^{i,N},\mu_{\underline{s}}^{\widehat{\boldsymbol{X}}^N}) \rangle +\dfrac{1}{2} \left|\sm_{\Delta}(\widehat{X}_{\underline{s}}^{i,N},\mu_{\underline{s}}^{\widehat{\boldsymbol{X}}^N})\right|^2 \notag\\
	&\qquad+\dfrac{1}{2}\left|c_{\Delta} (\widehat{X}_{\underline{s}}^{i,N},\mu_{\underline{s}}^{\widehat{\boldsymbol{X}}^N})\right|^2  \int_{\mathbb{R}_0^d} |z|^{2} \nu\left(dz\right) \Bigg)ds + 2\int_{0}^{t} e^{-2\widetilde{\gamma} s }\left\langle \widehat{X}_s^{i,N},  \sm_{\Delta}(\widehat{X}_{\underline{s}}^{i,N},\mu_{\underline{s}}^{\widehat{\boldsymbol{X}}^N}) dW_s^i \right\rangle \notag \\
	&\qquad  + \int_{0}^{t} \int_{\mathbb{R}_0^d} e^{-2 \widetilde{\gamma} s}\left(2\left\langle \widehat{X}_{\underline{s}-}^{i,N}, c_{\Delta}\left(\widehat{X}_{\underline{s}-}^{i,N},\mu_{\underline{s}-}^{\widehat{\boldsymbol{X}}^N}\right) z \right\rangle +  \left|c_{\Delta}\left(\widehat{X}_{\underline{s}-}^{i,N},\mu_{\underline{s}-}^{\widehat{\boldsymbol{X}}^N}\right) z\right|^2\right) \widetilde{N}^i(d s, d z).
	\label{e^-2gmX}
	\end{align}
	First, using \eqref{EM3}, we have 
	\begin{align*}
	-\widetilde{\gamma}|\widehat{X}_{s}^{i,N}|^{2}&=-\widetilde{\gamma}|\widehat{X}_{\underline{s}}^{i,N}|^{2}\\
	&\quad-2\widetilde{\gamma}\left\langle \widehat{X}_{\underline{s}}^{i,N}, b\left(\widehat{X}_{\underline{s}}^{i,N},\mu_{\underline{s}}^{\widehat{\boldsymbol{X}}^N}\right)(s-\underline{s})+\sigma_{\Delta}\left(\widehat{X}_{\underline{s}}^{i,N},\mu_{\underline{s}}^{\widehat{\boldsymbol{X}}^N}\right)\left(W_{s}^i-W_{\underline{s}}^i\right)+c_{\Delta}\left(\widehat{X}_{\underline{s}}^{i,N},\mu_{\underline{s}}^{\widehat{\boldsymbol{X}}^N}\right)\left(Z_{s}^i-Z_{\underline{s}}^i\right)\right\rangle  \\
	&\quad -\widetilde{\gamma}\left|b \left(\widehat{X}_{\underline{s}}^{i,N},\mu_{\underline{s}}^{\widehat{\boldsymbol{X}}^N}\right)(s-\underline{s})+\sigma_{\Delta}\left(\widehat{X}_{\underline{s}}^{i,N},\mu_{\underline{s}}^{\widehat{\boldsymbol{X}}^N}\right)\left(W_{s}^i-W_{\underline{s}}^i\right)+c_{\Delta}\left(\widehat{X}_{\underline{s}}^{i,N},\mu_{\underline{s}}^{\widehat{\boldsymbol{X}}^N}\right)\left(Z_{s}^i-Z_{\underline{s}}^i\right)\right|^{2}\\
	&\leq -\widetilde{\gamma}\left|\widehat{X}_{\underline{s}}^{i,N}\right|^{2}+2\vert \widetilde{\gamma}\vert \left\vert\widehat{X}_{\underline{s}}^{i,N}\right\vert \left\vert b\left(\widehat{X}_{\underline{s}}^{i,N},\mu_{\underline{s}}^{\widehat{\boldsymbol{X}}^N}\right)\right\vert(s-\underline{s})\\
	&\quad-2\widetilde{\gamma}\left\langle \widehat{X}_{\underline{s}}^{i,N}, \sigma_{\Delta}\left(\widehat{X}_{\underline{s}}^{i,N},\mu_{\underline{s}}^{\widehat{\boldsymbol{X}}^N}\right)\left(W_{s}^i-W_{\underline{s}}^i\right)+c_{\Delta}\left(\widehat{X}_{\underline{s}}^{i,N},\mu_{\underline{s}}^{\widehat{\boldsymbol{X}}^N}\right)\left(Z_{s}^i-Z_{\underline{s}}^i\right)\right\rangle  \\
	&\quad +3\vert \widetilde{\gamma} \vert\left(\left|b(\widehat{X}_{\underline{s}}^{i,N},\mu_{\underline{s}}^{\widehat{\boldsymbol{X}}^N})\right|^2 \left(s-\underline{s}\right)^2 + \left|\sigma_{\Delta}(\widehat{X}_{\underline{s}}^{i,N},\mu_{\underline{s}}^{\widehat{\boldsymbol{X}}^N})\right|^2  \left| W_{s}^i-W_{\underline{s}}^i\right|^2 + \left|c_{\Delta}(\widehat{X}_{\underline{s}}^{i,N},\mu_{\underline{s}}^{\widehat{\boldsymbol{X}}^N}) \right|^2 |Z_{s}^i-Z_{\underline{s}}^i|^2\right).
	\end{align*}
	Then, using  \textbf{T4}, \textbf{A5}, \eqref{markov}, and \eqref{chooseh}, we get
	\begin{align*} 
	&\max \Bigg\{ \left\vert\widehat{X}_{\underline{s}}^{i,N}\right\vert \left\vert b(\widehat{X}_{\underline{s}}^{i,N},\mu_{\underline{s}}^{\widehat{\boldsymbol{X}}^N})\right\vert(s-\underline{s}) ; \ \left|b(\widehat{X}_{\underline{s}}^{i,N},\mu_{\underline{s}}^{\widehat{\boldsymbol{X}}^N})\right|^2 (s-\underline{s})^2; \  \mathbb{E}\left[ \left|\sigma_\Delta\left(\widehat{X}_{\underline{s}}^{i,N},\mu_{\underline{s}}^{\widehat{\boldsymbol{X}}^N}\right)\right|^2 \left|W_s^i-W_{\underline{s}}^i\right|^2\big|\mathcal{F}_{\underline{s}}\right]; \\
	&\qquad \qquad \mathbb{E}\left[ \left|c_{\Delta}\left(\widehat{X}_{\underline{s}}^{i,N},\mu_{\underline{s}}^{\widehat{\boldsymbol{X}}^N}\right)\right|^2 \left|Z_{s}^i-Z_{\underline{s}}^i\right|^2 \big|\mathcal{F}_{\underline{s}}\right] \Bigg\}  \leq C\Delta,
	\end{align*} 
	which yields that
	\begin{equation}
	\mathbb{E}\left[-\widetilde{\gamma} |\widehat{X}_s^{i,N}|^2\right]=\mathbb{E}\left[\mathbb{E}\left[-\widetilde{\gamma} |\widehat{X}_s^{i,N}|^2\big\vert\mathcal{F}_{\underline{s}}\right] \right]\leq \mathbb{E}\left[-\widetilde{\gamma} |\widehat{X}_{\underline{s}}^{i,N}|^2\right]+C|\widetilde{\gamma}|\Delta. 
	\label{-2gmX}
	\end{equation}	
	Moreover, using again \eqref{EM3}, we have
	\begin{align*}
	\left\langle \widehat{X}_s^{i,N}, b(\widehat{X}_{\underline{s}}^{i,N},\mu_{\underline{s}}^{\widehat{\boldsymbol{X}}^N}) \right\rangle&=\left\langle \widehat{X}_{\underline{s}}^{i,N}, b(\widehat{X}_{\underline{s}}^{i,N},\mu_{\underline{s}}^{\widehat{\boldsymbol{X}}^N}) \right\rangle+\left\langle \widehat{X}_s^{i,N}-\widehat{X}_{\underline{s}}^{i,N}, b(\widehat{X}_{\underline{s}}^{i,N},\mu_{\underline{s}}^{\widehat{\boldsymbol{X}}^N}) \right\rangle\\
	&\leq \left\langle \widehat{X}_{\underline{s}}^{i,N}, b(\widehat{X}_{\underline{s}}^{i,N},\mu_{\underline{s}}^{\widehat{\boldsymbol{X}}^N}) \right\rangle+ \left|b(\widehat{X}_{\underline{s}}^{i,N},\mu_{\underline{s}}^{\widehat{\boldsymbol{X}}^N})\right|^2 (s-\underline{s})\\
	&\quad	
	+\left\langle \sigma_{\Delta}\left(\widehat{X}_{\underline{s}}^{i,N},\mu_{\underline{s}}^{\widehat{\boldsymbol{X}}^N}\right)\left(W_{s}^i-W_{\underline{s}}^i\right)+c_{\Delta}\left(\widehat{X}_{\underline{s}}^{i,N},\mu_{\underline{s}}^{\widehat{\boldsymbol{X}}^N}\right)\left(Z_{s}^i-Z_{\underline{s}}^i\right), b(\widehat{X}_{\underline{s}}^{i,N},\mu_{\underline{s}}^{\widehat{\boldsymbol{X}}^N}) \right\rangle,
	\end{align*}
	which, combined with the fact that $|b(\widehat{X}_{\underline{s}}^{i,N},\mu_{\underline{s}}^{\widehat{\boldsymbol{X}}^N})|^2 (s-\underline{s})\leq \Delta$ due to \eqref{chooseh}, deduces that 
	\begin{equation}
	\mathbb{E}\left[\left\langle \widehat{X}_s^{i,N}, b(\widehat{X}_{\underline{s}}^{i,N},\mu_{\underline{s}}^{\widehat{\boldsymbol{X}}^N}) \right\rangle \right]\leq \mathbb{E}\left[\left\langle \widehat{X}_{\underline{s}}^{i,N}, b(\widehat{X}_{\underline{s}}^{i,N},\mu_{\underline{s}}^{\widehat{\boldsymbol{X}}^N}) \right\rangle \right]+\Delta. \label{2Xb}
	\end{equation}
	Thanks to Lemma \ref{menh de 3}, the stochastic integrals in \eqref{e^-2gmX}  have zero expectation. Thus, using  
	\eqref{e^-2gmX}, \eqref{-2gmX}, \eqref{2Xb} and \eqref{sigmacdelta} of \textbf{T5} with $p_0=2$ and recall that $\widetilde{\gamma} =\widetilde{\gamma}_1 +\widetilde{\gamma}_2$, we obtain that 
	\begin{align*}
	\mathbb{E}\left[e^{-2\widetilde{\gamma} t}|\widehat{X}_t^{i,N}|^2\right]&\leq |x_0|^2+2\int_0^t e^{-2\widetilde{\gamma} s}\bigg(\mathbb{E}\bigg[-\widetilde{\gamma} |\widehat{X}_{\underline{s}}^{i,N}|^2+\left\langle \widehat{X}_{\underline{s}}^{i,N}, b(\widehat{X}_{\underline{s}}^{i,N},\mu_{\underline{s}}^{\widehat{\boldsymbol{X}}^N}) \right\rangle+\dfrac{1}{2} \left|\sm_{\Delta}(\widehat{X}_{\underline{s}}^{i,N},\mu_{\underline{s}}^{\widehat{\boldsymbol{X}}^N})\right|^2 \\
	&\qquad+\dfrac{1}{2}\left|c_{\Delta} (\widehat{X}_{\underline{s}}^{i,N},\mu_{\underline{s}}^{\widehat{\boldsymbol{X}}^N})\right|^2  \int_{\mathbb{R}_0^d} |z|^{2} \nu\left(dz\right)\bigg]+C\bigg) ds\\
	&\leq |x_0|^2+2\int_0^t e^{-2\widetilde{\gamma} s}\bigg(-\widetilde{\gamma}_2\mathbb{E}\left[ |\widehat{X}_{\underline{s}}^{i,N}|^2\right]+\widetilde{\gamma}_2\mathbb{E}\left[\mathcal{W}_2^2(\mu_{\underline{s}}^{\widehat{\boldsymbol{X}}^N},\delta_0)\right]+\widetilde{\eta}+C\bigg) ds\\
	&= |x_0|^2+2\int_0^t e^{-2\widetilde{\gamma} s}\left(-\widetilde{\gamma}_2\mathbb{E}\left[ |\widehat{X}_{\underline{s}}^{i,N}|^2\right]+\widetilde{\gamma}_2\mathbb{E}\left[ |\widehat{X}_{\underline{s}}^{i,N}|^2\right]+\widetilde{\eta}+C\right) ds\notag\\
	&= |x_0|^2+2\left(\widetilde{\eta}+C\right)\int_0^t e^{-2\widetilde{\gamma} s} ds,
	\end{align*} 
	for some positive constant $C$, where we have used the equality $\bE \left[\mathcal{W}_2^2(\mu_{\underline{s}}^{\widehat{\boldsymbol{X}}^N},\delta_0)\right]= \mathbb{E}\left[|\widehat{X}_{\underline{s}}^{i,N}|^2\right]$. This yields to
	\begin{align} 
	\bE \left[|\widehat{X}_t^{i,N}|^2\right] \le \begin{cases}  \left(|x_{0}|^{2}+\dfrac{\widetilde{\eta}+C}{\widetilde{\gamma}}\right) e^{2\widetilde{\gamma} t}-\dfrac{\widetilde{\eta}+C}{\widetilde{\gamma}} & \text{ if \;\;} \widetilde{\gamma} \not =0,\\ 
	|x_{0}|^{2}+2\left(\widetilde{\eta}+C\right) t & \text{ if \;\;} \widetilde{\gamma}=0.\label{EXtmu2}
	\end{cases} 
	\end{align}		
	Next, from \eqref{EM3}, we have
	$$
	\widehat{X}_{\underline{t}}^{i,N}=\widehat{X}_{t}^{i,N}-b\left(\widehat{X}_{\underline{t}}^{i,N},\mu_{\underline{t}}^{\widehat{\boldsymbol{X}}^N}\right)(t-\underline{t})-\sigma_{\Delta}\left(\widehat{X}_{\underline{t}}^{i,N},\mu_{\underline{t}}^{\widehat{\boldsymbol{X}}^N}\right)\left(W_{t}^i-W_{\underline{t}}^i\right)-c_{\Delta}\left(\widehat{X}_{\underline{t}}^{i,N},\mu_{\underline{t}}^{\widehat{\boldsymbol{X}}^N}\right)\left(Z_{t}^i-Z_{\underline{t}}^i\right).
	$$
	This, combined with \textbf{T4}, \textbf{A5}, \eqref{markov} and \eqref{chooseh}, we get that for any $p>1$,
	\begin{align}
	\bE \left[\left|\widehat{X}_{\underline{t}}^{i,N}\right|^p\right] 
	&\le 4^{p-1} \bigg(\bE\left[\left|\widehat{X}_t^{i,N}\right|^p\right] +\bE\left[\left|b(\widehat{X}_{\underline{t}}^{i,N},\mu_{\underline{t}}^{\widehat{\boldsymbol{X}}^N})(t-\underline{t})\right|^p\right] +\bE \left[\left|\sm_{\Delta}(\widehat{X}_{\underline{t}}^{i,N},\mu_{\underline{t}}^{\widehat{\boldsymbol{X}}^N})(W_t^i-W_{\underline{t}}^i)\right|^p\right] \notag\\
	& \qquad+\bE \left[\left|c_{\Delta}(\widehat{X}_{\underline{t}}^{i,N},\mu_{\underline{t}}^{\widehat{\boldsymbol{X}}^N})(Z_t^i-Z_{\underline{t}}^i)\right|^p\right] \bigg) \notag \\
	&\le 4^{p-1} \left(\bE\left[\left|\widehat{X}_t^{i,N}\right|^p\right] +C\Delta^p +C\Delta^{p/2} +C\Delta^{1 \wedge p/2} \right). \label{qh}
	\end{align}
	Consequently, from \eqref{EXtmu2} and \eqref{qh} with $p=2$, the result follows. 
\end{proof}

To estimate $L^p$-norm of the approximate solution for $p>2$, we need  a series of preliminary lemmas.

\begin{Lem} \label{abcLemma}
Let $p$ be a positive even integer.  For any $a, b, c \in \mathbb{R}^d$, it holds that 
\begin{align} 
S= &|a+c|^p - |b+c|^p - |a|^p + |b|^p \notag \\
\leq &\sum_{j=1}^{p/2} \sum_{k=0}^j {p/2 \choose j} {j \choose k} 2^{j-k}|c|^{j+k} \Big[ 2^{p-2j-1} (|a-b|^{p-2j}+|b|^{p-2j})\sum_{\ell=1}^{j-k}  {j-k \choose \ell} |b|^{j-k-\ell} |a-b|^{\ell}  \notag \\
& \qquad \qquad \qquad \qquad + \sum_{\ell=1}^{p-2j} {p-2j \choose \ell} |b|^{p-j-k-\ell} |a-b|^{\ell} \Big].
\label{abc_expan}
\end{align}
\end{Lem}

\begin{proof}
We first note that 
$S  = \Big(|a|^2 + 2 \langle a, c \rangle + |c|^2\Big)^{p/2} - |a|^p - \Big(|b|^2 + 2 \langle b, c \rangle + |c|^2 \Big)^{p/2} + |b|^p$.   
Using the binomial theorem, we have 
\begin{align*}
S  &= \sum_{j=1}^{p/2} \sum_{k=0}^j {p/2 \choose j} {j \choose k} 2^{j-k}|c|^{2k} \Big( |a|^{p-2j}  \langle a, c \rangle^{j-k} - |b|^{p-2j}  \langle b, c \rangle^{j-k} \Big).
\end{align*}
Next, we write 
\begin{align*}
 |a|^{p-2j}  \langle a, c \rangle^{j-k} - |b|^{p-2j}  \langle b, c \rangle^{j-k} &=  |a|^{p-2j} \Big( (\langle b, c \rangle + \langle a-b,c\rangle) ^{j-k}  - \langle b, c \rangle^{j-k} \Big) \\
 &\quad + \Big( (|b| + (|a|-|b|))^{p-2j} - |b|^{p-2j}\Big)   \langle b, c \rangle^{j-k}.
\end{align*} 
Using the binomial theorem, the estimates $|a|^{p-2j} \leq 2^{p-2j-1}(|a-b|^{p-2j} + |b|^{p-2j})$,  $|\langle a-b, c\rangle| \leq |a-b||c|$, and $|\langle b, c\rangle | \leq |b||c|$, we obtain the desired result. 
\end{proof} 

\begin{Lem} \label{Lem:condmom}
Assume  Conditions \textbf{T1}--\textbf{T5} and \textbf{A5}. 
Then, for any even integer   $p \in (0, p_0]$,  there exists a positive constant $C_p$ such that for any $s>0$, $i\in\{1,\ldots,N\}$ and $\lambda\in\mathbb{R}$, 
\begin{enumerate} 
\item[\textnormal{a)}] 
$\mathbb{E}\left[-\lambda |\widehat{X}_s^{i,N}|^p \big| \mathcal{F}_{\underline{s}}\right] \leq   -\lambda |\widehat{X}_{\underline{s}}^{i,N}|^p +C_p|\lambda|\sum\limits_{j=0}^{p-2}  |\widehat{X}_{\underline{s}}^{i,N}|^j$. 
\item[\textnormal{b)}] 
$ \mathbb{E}\left[|\widehat{X}_s^{i,N}|^{p-2} \left\langle \widehat{X}_s^{i,N}, b(\widehat{X}_{\underline{s}}^{i,N},\mu_{\underline{s}}^{\widehat{\boldsymbol{X}}^N}) \right\rangle \big| \mathcal{F}_{\underline{s}} \right]\leq  |\widehat{X}_{\underline{s}}^{i,N}|^{p-2} \left\langle \widehat{X}_{\underline{s}}^{i,N}, b(\widehat{X}_{\underline{s}}^{i,N},\mu_{\underline{s}}^{\widehat{\boldsymbol{X}}^N}) \right\rangle  +C_p\sum\limits_{j=0}^{p-2}  |\widehat{X}_{\underline{s}}^{i,N}|^j. $
\item[\textnormal{c)}] 
	$\mathbb{E}\left[|\widehat{X}_s^{i,N}|^{p-4}|(\widehat{X}_s^{i,N})^\mathsf{T} \sigma_\Delta(\widehat{X}_{\underline{s}}^{i,N},\mu_{\underline{s}}^{\widehat{\boldsymbol{X}}^N})|^2 \big| \mathcal{F}_{\underline{s}} \right] \leq   |\widehat{X}_{\underline{s}}^{i,N}|^{p-2}|\sigma_\Delta(\widehat{X}_{\underline{s}}^{i,N},\mu_{\underline{s}}^{\widehat{\boldsymbol{X}}^N})|^2 +C_p\sum\limits_{j=0}^{p-3} |\widehat{X}_{\underline{s}}^{i,N}|^j.$
\end{enumerate} 
\end{Lem}

\begin{proof}
		First, by using \textbf{T4}, \eqref{chooseh}, Burkholder-Davis-Gundy's  inequality, \eqref{markov} and \textbf{A5}, we get that for all $2\leq j\leq p$,	
		\begin{align}
			&\max \bigg \{ |\widehat{X}_{\underline{s}}^{i,N}|\left|b\left(\widehat{X}_{\underline{s}}^{i,N},\mu_{\underline{s}}^{\widehat{\boldsymbol{X}}^N}\right)\right| \left|s-\underline{s}\right|; \  \mathbb{E}\left[ |\widehat{X}_{\underline{s}}^{i,N}|\left|\sigma_{\Delta}\left(\widehat{X}_{\underline{s}}^{i,N},\mu_{\underline{s}}^{\widehat{\boldsymbol{X}}^N}\right)\right|  \left| W_{s}^i-W_{\underline{s}}^i\right|\big|\mathcal{F}_{\underline{s}}\right]; \notag\\
			&\quad  \mathbb{E}\left[ |\widehat{X}_{\underline{s}}^{i,N}|\left|c_{\Delta}\left(\widehat{X}_{\underline{s}}^{i,N},\mu_{\underline{s}}^{\widehat{\boldsymbol{X}}^N}\right) \right| \left|Z_{s}^i-Z_{\underline{s}}^i\right| \big|\mathcal{F}_{\underline{s}}\right] \bigg\}  \leq C\sqrt{\Delta}, \notag \\
			&\max \bigg \{ \left|b\left(\widehat{X}_{\underline{s}}^{i,N},\mu_{\underline{s}}^{\widehat{\boldsymbol{X}}^N}\right)\right|^j \left|s-\underline{s}\right|^j; \  \mathbb{E}\left[ \left|\sigma_{\Delta}\left(\widehat{X}_{\underline{s}}^{i,N},\mu_{\underline{s}}^{\widehat{\boldsymbol{X}}^N}\right)\right|^j  \left| W_{s}^i-W_{\underline{s}}^i\right|^j\big|\mathcal{F}_{\underline{s}}\right]; \notag\\
			&\quad  \mathbb{E}\left[ \left|c_{\Delta}\left(\widehat{X}_{\underline{s}}^{i,N},\mu_{\underline{s}}^{\widehat{\boldsymbol{X}}^N}\right) \right|^j \left|Z_{s}^i-Z_{\underline{s}}^i\right|^j  \big|\mathcal{F}_{\underline{s}}\right] \bigg\}  \leq C\Delta.  \label{con-moment}
		\end{align} 	

\noindent 
a) Using the binomial theorem and \eqref{EM3}, we get that   
		\begin{align}
			-\lambda |\widehat{X}_s^{i,N}|^p
			&=  -\lambda |\widehat{X}_{\underline{s}}^{i,N}|^p -\lambda\sum_{j=1}^p {p \choose j} |\widehat{X}_{\underline{s}}^{i,N}|^{p-j} \left(|\widehat{X}_s^{i,N}|-|\widehat{X}_{\underline{s}}^{i,N}|\right)^j \notag \\
			&\le -\lambda |\widehat{X}_{\underline{s}}^{i,N}|^p + |\lambda| p |\widehat{X}_{\underline{s}}^{i,N}|^{p-1} \left|\widehat{X}_s^{i,N}-\widehat{X}_{\underline{s}}^{i,N} \right| \quad+|\lambda|\sum_{j=2}^p  {p \choose j} |\widehat{X}_{\underline{s}}^{i,N}|^{p-j} \left|\widehat{X}_s^{i,N}-\widehat{X}_{\underline{s}}^{i,N} \right|^j  \notag \\
			&\le -\lambda |\widehat{X}_{\underline{s}}^{i,N}|^p + |\lambda|p |\widehat{X}_{\underline{s}}^{i,N}|^{p-2} \bigg(|\widehat{X}_{\underline{s}}^{i,N}|\left|b\left(\widehat{X}_{\underline{s}}^{i,N},\mu_{\underline{s}}^{\widehat{\boldsymbol{X}}^N}\right)\right| \left|s-\underline{s}\right| + |\widehat{X}_{\underline{s}}^{i,N}|\left|\sigma_{\Delta}\left(\widehat{X}_{\underline{s}}^{i,N},\mu_{\underline{s}}^{\widehat{\boldsymbol{X}}^N}\right)\right|  \left| W_{s}^i-W_{\underline{s}}^i\right| \notag\\
			&\quad+ |\widehat{X}_{\underline{s}}^{i,N}|\left|c_{\Delta}\left(\widehat{X}_{\underline{s}}^{i,N},\mu_{\underline{s}}^{\widehat{\boldsymbol{X}}^N}\right) \right| \left|Z_{s}^i-Z_{\underline{s}}^i\right|  \bigg) + |\lambda|\sum_{j=2}^p  3^{j-1} {p \choose j} |\widehat{X}_{\underline{s}}^{i,N}|^{p-j} \bigg(\left|b\left(\widehat{X}_{\underline{s}}^{i,N},\mu_{\underline{s}}^{\widehat{\boldsymbol{X}}^N}\right)\right|^j \left|s-\underline{s}\right|^j\notag \\
			&\quad  + \left|\sigma_{\Delta}\left(\widehat{X}_{\underline{s}}^{i,N},\mu_{\underline{s}}^{\widehat{\boldsymbol{X}}^N}\right)\right|^j  \left| W_{s}^i-W_{\underline{s}}^i\right|^j + \left|c_{\Delta}\left(\widehat{X}_{\underline{s}}^{i,N},\mu_{\underline{s}}^{\widehat{\boldsymbol{X}}^N}\right) \right|^j \left|Z_{s}^i-Z_{\underline{s}}^i\right|^j  \bigg). \notag
		\end{align}
	Using \eqref{con-moment}, we conclude part a) of Lemma \ref{Lem:condmom}.

b) Using a similar computation as in part a), we obtain 
		\begin{align*}
			&|\widehat{X}_s^{i,N}|^{p-2} \left\langle \widehat{X}_s^{i,N}, b(\widehat{X}_{\underline{s}}^{i,N},\mu_{\underline{s}}^{\widehat{\boldsymbol{X}}^N}) \right\rangle\\ 
			&\leq |\widehat{X}_{\underline{s}}^{i,N}|^{p-2}\left\langle \widehat{X}_{\underline{s}}^{i,N}, b(\widehat{X}_{\underline{s}}^{i,N},\mu_{\underline{s}}^{\widehat{\boldsymbol{X}}^N}) \right\rangle+|\widehat{X}_{\underline{s}}^{i,N}|^{p-2} |b(\widehat{X}_{\underline{s}}^{i,N},\mu_{\underline{s}}^{\widehat{\boldsymbol{X}}^N})|^2(s-\underline{s}) +|\widehat{X}_{\underline{s}}^{i,N}|^{p-2}\bigg\langle  \sigma_{\Delta}\left(\widehat{X}_{\underline{s}}^{i,N},\mu_{\underline{s}}^{\widehat{\boldsymbol{X}}^N}\right)\left(W_{s}^i-W_{\underline{s}}^i\right) \\
			&\quad+ c_{\Delta}\left(\widehat{X}_{\underline{s}}^{i,N},\mu_{\underline{s}}^{\widehat{\boldsymbol{X}}^N}\right)\left(Z_{s}^i-Z_{\underline{s}}^i\right), b(\widehat{X}_{\underline{s}}^{i,N},\mu_{\underline{s}}^{\widehat{\boldsymbol{X}}^N}) \bigg\rangle +\sum_{j=1}^{p-2}{p-2 \choose j}|\widehat{X}_{\underline{s}}^{i,N}|^{p-1-j}3^{j-1}|b(\widehat{X}_{\underline{s}}^{i,N},\mu_{\underline{s}}^{\widehat{\boldsymbol{X}}^N})|\\
			&\quad \times\left(\left|b\left(\widehat{X}_{\underline{s}}^{i,N},\mu_{\underline{s}}^{\widehat{\boldsymbol{X}}^N}\right)\right|^j \left|s-\underline{s}\right|^j + \left|\sigma_{\Delta}\left(\widehat{X}_{\underline{s}}^{i,N},\mu_{\underline{s}}^{\widehat{\boldsymbol{X}}^N}\right)\right|^j  \left| W_{s}^i-W_{\underline{s}}^i\right|^j + \left|c_{\Delta}\left(\widehat{X}_{\underline{s}}^{i,N},\mu_{\underline{s}}^{\widehat{\boldsymbol{X}}^N}\right) \right|^j \left|Z_{s}^i-Z_{\underline{s}}^i\right|^j\right) \\
			&\quad + \sum_{j=1}^{p-2}{p-2 \choose j}|\widehat{X}_{\underline{s}}^{i,N}|^{p-2-j}3^{j}|b(\widehat{X}_{\underline{s}}^{i,N},\mu_{\underline{s}}^{\widehat{\boldsymbol{X}}^N})|\bigg(\left|b\left(\widehat{X}_{\underline{s}}^{i,N},\mu_{\underline{s}}^{\widehat{\boldsymbol{X}}^N}\right)\right|^{j+1} \left|s-\underline{s}\right|^{j+1} \\
			&\quad +\left|\sigma_{\Delta}\left(\widehat{X}_{\underline{s}}^{i,N},\mu_{\underline{s}}^{\widehat{\boldsymbol{X}}^N}\right)\right|^{j+1}  \left| W_{s}^i-W_{\underline{s}}^i\right|^{j+1}+ \left|c_{\Delta}\left(\widehat{X}_{\underline{s}}^{i,N},\mu_{\underline{s}}^{\widehat{\boldsymbol{X}}^N}\right) \right|^{j+1} \left|Z_{s}^i-Z_{\underline{s}}^i\right|^{j+1}\bigg).
		\end{align*}
			Using again \eqref{con-moment}, we conclude part b) of Lemma \ref{Lem:condmom}.

c) Using the binomial theorem and \eqref{EM3}, we have
		\begin{align*}
			&|\widehat{X}_s^{i,N}|^{p-2} |\sm_{\Delta} (\widehat{X}_{\underline{s}}^{i,N},\mu_{\underline{s}}^{\widehat{\boldsymbol{X}}^N})|^2 \\ 
			&\leq |\widehat{X}_{\underline{s}}^{i,N}|^{p-2}|\sm_{\Delta} (\widehat{X}_{\underline{s}}^{i,N},\mu_{\underline{s}}^{\widehat{\boldsymbol{X}}^N})|^2+\sum_{j=1}^{p-2}{p-2 \choose j}|\widehat{X}_{\underline{s}}^{i,N}|^{p-2-j}|\sm_{\Delta} (\widehat{X}_{\underline{s}}^{i,N},\mu_{\underline{s}}^{\widehat{\boldsymbol{X}}^N})|^23^{j-1}\\
			&\quad\times\left(\left|b\left(\widehat{X}_{\underline{s}}^{i,N},\mu_{\underline{s}}^{\widehat{\boldsymbol{X}}^N}\right)\right|^j \left|s-\underline{s}\right|^j  + \left|\sigma_{\Delta}\left(\widehat{X}_{\underline{s}}^{i,N},\mu_{\underline{s}}^{\widehat{\boldsymbol{X}}^N}\right)\right|^j  \left| W_{s}^i-W_{\underline{s}}^i\right|^j+ \left|c_{\Delta}\left(\widehat{X}_{\underline{s}}^{i,N},\mu_{\underline{s}}^{\widehat{\boldsymbol{X}}^N}\right) \right|^j \left|Z_{s}^i-Z_{\underline{s}}^i\right|^j\right).
		\end{align*}
		Then, using the estimates \eqref{con-moment}, we get that
		\begin{align}
			\mathbb{E}\left[|\widehat{X}_s^{i,N}|^{p-4}|(\widehat{X}_s^{i,N})^\mathsf{T} \sigma_\Delta(\widehat{X}_{\underline{s}}^{i,N},\mu_{\underline{s}}^{\widehat{\boldsymbol{X}}^N})|^2 \big| \mathcal{F}_{\underline{s}} \right]&\leq\mathbb{E}\left[|\widehat{X}_s^{i,N}|^{p-2}|\sigma_\Delta(\widehat{X}_{\underline{s}}^{i,N},\mu_{\underline{s}}^{\widehat{\boldsymbol{X}}^N})|^2 \big| \mathcal{F}_{\underline{s}} \right] \notag\\
			&\leq   |\widehat{X}_{\underline{s}}^{i,N}|^{p-2}|\sigma_\Delta(\widehat{X}_{\underline{s}}^{i,N},\mu_{\underline{s}}^{\widehat{\boldsymbol{X}}^N})|^2 +C_p\sum\limits_{j=0}^{p-3} |\widehat{X}_{\underline{s}}^{i,N}|^j. \label{X^p-2sm}
		\end{align}	
		We conclude part c) of Lemma \ref{Lem:condmom}.
\end{proof}

\begin{Lem} Assume  Conditions \textbf{T1}--\textbf{T5} and \textbf{A5}. 
Then, for any even integer   $p \in (0, p_0]$,   $s>0$, and $z \in \mathbb{R}^d$, it holds that 
\begin{align}
			&\mathbb{E}\left[\left(\left|\widehat{X}_{s}^{i,N}+c_{\Delta}(\widehat{X}_{\underline{s}}^{i,N},\mu_{\underline{s}}^{\widehat{\boldsymbol{X}}^N})  z\right|^{p}-\left|\widehat{X}_{\underline{s}}^{i,N}+c_{\Delta}(\widehat{X}_{\underline{s}}^{i,N},\mu_{\underline{s}}^{\widehat{\boldsymbol{X}}^N})  z\right|^{p}\right)-\left(|\widehat{X}_{s}^{i,N}|^{p}-|\widehat{X}_{\underline{s}}^{i,N}|^{p}\right)\big| \mathcal{F}_{\underline{s}} \right] \notag\\
			&\leq \widehat{Q}_{p-2}\left(\left|\widehat{X}_{\underline{s}}^{i,N}\right|,|z|\right)+  \sum_{j=1}^{p/2}\sum_{k=0}^{j}C_{j,k} \vert z\vert^{j+k}\mathcal{W}_2^{j+k}(\mu_{\underline{s}}^{\widehat{\boldsymbol{X}}^N},\delta_0) \notag\\
			&\quad\times\bigg(\left\vert \widehat{X}_{\underline{s}}^{i,N}\right\vert^{j-k-2}+\left|\widehat{X}_{\underline{s}}^{i,N}\right\vert^{p-j-k-2}+\sum_{\ell=2}^{j-k}\left(\left\vert \widehat{X}_{\underline{s}}^{i,N}\right\vert^{j-k-\ell}+\left|\widehat{X}_{\underline{s}}^{i,N}\right\vert^{p-j-k-\ell}\right)+\sum_{\ell=2}^{p-2j}\left|\widehat{X}_{\underline{s}}^{i,N}\right\vert^{p-j-k-\ell}\bigg)			
			, \label{C52}
			\end{align}
			where $(C_{j,k})$ are some positive constants , $\widehat{Q}_{p-2}(|\widehat{X}_{\underline{s}}^{i,N}|,|z|)$ is a polynomial in $|\widehat{X}_{\underline{s}}^{i,N}|$ of degree $p-2$.
\end{Lem}

\begin{proof}
		By using Lemma \ref{abcLemma}, we have 
			\begin{align*}
			&\left(\left|\widehat{X}_{s}^{i,N}+c_{\Delta}(\widehat{X}_{\underline{s}}^{i,N},\mu_{\underline{s}}^{\widehat{\boldsymbol{X}}^N})  z\right|^{p}-\left|\widehat{X}_{\underline{s}}^{i,N}+c_{\Delta}(\widehat{X}_{\underline{s}}^{i,N},\mu_{\underline{s}}^{\widehat{\boldsymbol{X}}^N})  z\right|^{p}\right)-\left(|\widehat{X}_{s}^{i,N}|^{p}-|\widehat{X}_{\underline{s}}^{i,N}|^{p}\right)\\
			& \leq \sum_{j=1}^{p/2}\sum_{k=0}^{j}{p/2 \choose j} {j \choose k} 2^{j-k} \left\vert c_{\Delta}(\widehat{X}_{\underline{s}}^{i,N},\mu_{\underline{s}}^{\widehat{\boldsymbol{X}}^N})z\right\vert^{j+k} \Bigg(2^{p-2j-1} \sum_{\ell=1}^{j-k}{j-k \choose \ell} \left\vert \widehat{X}_{\underline{s}}^{i,N}\right\vert^{j-k-\ell} \left\vert \widehat{X}_{s}^{i,N}-\widehat{X}_{\underline{s}}^{i,N}\right\vert^{p-2j+\ell}\\
			&\quad+2^{p-2j-1}\sum_{\ell=1}^{j-k}{j-k \choose \ell} \left\vert \widehat{X}_{\underline{s}}^{i,N}\right\vert^{p-j-k-\ell}  \left\vert \widehat{X}_{s}^{i,N}-\widehat{X}_{\underline{s}}^{i,N}\right\vert^{\ell}+ \sum_{\ell=1}^{p-2j}{p-2j \choose \ell}\left|\widehat{X}_{\underline{s}}^{i,N}\right\vert^{p-j-k-\ell}\left|\widehat{X}_{s}^{i,N}-\widehat{X}_{\underline{s}}^{i,N}\right\vert^{\ell}\Bigg).
			\end{align*}	
			Then, using \eqref{cdelta} of \textbf{T5} and \eqref{EM3}, we get
			\begin{align*}
			&\left(\left|\widehat{X}_{s}^{i,N}+c_{\Delta}(\widehat{X}_{\underline{s}}^{i,N},\mu_{\underline{s}}^{\widehat{\boldsymbol{X}}^N})  z\right|^{p}-\left|\widehat{X}_{\underline{s}}^{i,N}+c_{\Delta}(\widehat{X}_{\underline{s}}^{i,N},\mu_{\underline{s}}^{\widehat{\boldsymbol{X}}^N})  z\right|^{p}\right)-\left(|\widehat{X}_{s}^{i,N}|^{p}-|\widehat{X}_{\underline{s}}^{i,N}|^{p}\right)\\
			&\leq \sum_{j=1}^{p/2}\sum_{k=0}^{j}{p/2 \choose j} {j \choose k} 2^{j-k} \left\vert z\right\vert^{j+k} (\widetilde{L_3})^{j+k}3^{j+k-1}\left(1+\left\vert \widehat{X}_{\underline{s}}^{i,N}\right\vert^{j+k}+\mathcal{W}_2^{j+k}(\mu_{\underline{s}}^{\widehat{\boldsymbol{X}}^N},\delta_0)\right)\Bigg(2^{p-2j-1} (j-k) \left\vert \widehat{X}_{\underline{s}}^{i,N}\right\vert^{j-k-2} \\
			&\times 3^{p-2j}\bigg(\left\vert \widehat{X}_{\underline{s}}^{i,N}\right\vert\left|b\left(\widehat{X}_{\underline{s}}^{i,N},\mu_{\underline{s}}^{\widehat{\boldsymbol{X}}^N}\right)\right|^{p-2j+1} \left|s-\underline{s}\right|^{p-2j+1}+ \left\vert \widehat{X}_{\underline{s}}^{i,N}\right\vert\left|\sigma_{\Delta}\left(\widehat{X}_{\underline{s}}^{i,N},\mu_{\underline{s}}^{\widehat{\boldsymbol{X}}^N}\right)\right|^{p-2j+1} \left| W_{s}^i-W_{\underline{s}}^i\right|^{p-2j+1} \\
			&   +\left\vert \widehat{X}_{\underline{s}}^{i,N}\right\vert \left|c_{\Delta}\left(\widehat{X}_{\underline{s}}^{i,N},\mu_{\underline{s}}^{\widehat{\boldsymbol{X}}^N}\right) \right|^{p-2j+1} \left|Z_{s}^i-Z_{\underline{s}}^i\right|^{p-2j+1}\bigg)+\left(2^{p-2j-1}{j-k\choose 1} + {p-2j\choose 1}\right)\left|\widehat{X}_{\underline{s}}^{i,N}\right\vert^{p-j-k-2}\\
			&\times\bigg(|\widehat{X}_{\underline{s}}^{i,N}||b(\widehat{X}_{\underline{s}}^{i,N},\mu_{\underline{s}}^{\widehat{\boldsymbol{X}}^N})|(s-\underline{s})+ |\widehat{X}_{\underline{s}}^{i,N}|\left|\sigma_{\Delta}\left(\widehat{X}_{\underline{s}}^{i,N},\mu_{\underline{s}}^{\widehat{\boldsymbol{X}}^N}\right)\right|\left|W_{s}^i-W_{\underline{s}}^i\right|+ |\widehat{X}_{\underline{s}}^{i,N}|\left|c_{\Delta}\left(\widehat{X}_{\underline{s}}^{i,N},\mu_{\underline{s}}^{\widehat{\boldsymbol{X}}^N}\right)\right|\left|Z_{s}^i-Z_{\underline{s}}^i\right| \bigg) \Bigg)\\
			& +\sum_{j=1}^{p/2}\sum_{k=0}^{j}{p/2 \choose j} {j \choose k} 2^{j-k} \left\vert z\right\vert^{j+k} (\widetilde{L_3})^{j+k}3^{j+k-1}\left(1+\vert \widehat{X}_{\underline{s}}^{i,N}\vert^{j+k}+\mathcal{W}_2^{j+k}(\mu_{\underline{s}}^{\widehat{\boldsymbol{X}}^N},\delta_0)\right)\\
			&\times \Bigg(2^{p-2j-1} \sum_{\ell=2}^{j-k}{j-k \choose \ell} \left\vert \widehat{X}_{\underline{s}}^{i,N}\right\vert^{j-k-\ell} 3^{p-2j+\ell-1} \bigg(\left|b\left(\widehat{X}_{\underline{s}}^{i,N},\mu_{\underline{s}}^{\widehat{\boldsymbol{X}}^N}\right)\right|^{p-2j+\ell} \left|s-\underline{s}\right|^{p-2j+\ell} + \left|\sigma_{\Delta}\left(\widehat{X}_{\underline{s}}^{i,N},\mu_{\underline{s}}^{\widehat{\boldsymbol{X}}^N}\right)\right|^{p-2j+\ell} \\
			&\times \left| W_{s}^i-W_{\underline{s}}^i\right|^{p-2j+\ell}   + \left|c_{\Delta}\left(\widehat{X}_{\underline{s}}^{i,N},\mu_{\underline{s}}^{\widehat{\boldsymbol{X}}^N}\right) \right|^{p-2j+\ell} \left|Z_{s}^i-Z_{\underline{s}}^i\right|^{p-2j+\ell}\bigg) +2^{p-2j-1}\sum_{\ell=2}^{j-k}{j-k \choose \ell} \left\vert \widehat{X}_{\underline{s}}^{i,N}\right\vert^{p-j-k-\ell}   3^{\ell-1}\\
			&\times \bigg(\left|b\left(\widehat{X}_{\underline{s}}^{i,N},\mu_{\underline{s}}^{\widehat{\boldsymbol{X}}^N}\right)\right|^{\ell} \left|s-\underline{s}\right|^{\ell} + \left|\sigma_{\Delta}\left(\widehat{X}_{\underline{s}}^{i,N},\mu_{\underline{s}}^{\widehat{\boldsymbol{X}}^N}\right)\right|^{\ell}  \left| W_{s}^i-W_{\underline{s}}^i\right|^{\ell}  + \left|c_{\Delta}\left(\widehat{X}_{\underline{s}}^{i,N},\mu_{\underline{s}}^{\widehat{\boldsymbol{X}}^N}\right) \right|^{\ell} \left|Z_{s}^i-Z_{\underline{s}}^i\right|^{\ell}\bigg) \\
			&+ \sum_{\ell=2}^{p-2j}{p-2j \choose \ell}3^{\ell-1} \left|\widehat{X}_{\underline{s}}^{i,N}\right\vert^{p-j-k-\ell}\bigg(\left|b\left(\widehat{X}_{\underline{s}}^{i,N},\mu_{\underline{s}}^{\widehat{\boldsymbol{X}}^N}\right)\right|^{\ell} \left|s-\underline{s}\right|^{\ell} + \left|\sigma_{\Delta}\left(\widehat{X}_{\underline{s}}^{i,N},\mu_{\underline{s}}^{\widehat{\boldsymbol{X}}^N}\right)\right|^{\ell}  \left| W_{s}^i-W_{\underline{s}}^i\right|^{\ell}  \\
			&+ \left|c_{\Delta}\left(\widehat{X}_{\underline{s}}^{i,N},\mu_{\underline{s}}^{\widehat{\boldsymbol{X}}^N}\right) \right|^{\ell} \left|Z_{s}^i-Z_{\underline{s}}^i\right|^{\ell}\bigg)\Bigg).
			\end{align*}
		By applying the estimates \eqref{con-moment} and condition \textbf{T4}, we obtain the desired result. 	
\end{proof}

\begin{Lem} \label{LemulXs_}
Assume that  Conditions \textbf{T1}--\textbf{T5} and \textbf{A5} hold, and  $N \geq \left( \frac{\max\{3\widetilde{L_3},1\}}{2\upsilon}\right)^2$. 
Then, for any even integer   $p \in (0, p_0]$,   $s>0$, and $z \in \mathbb{R}^d$, it holds
\begin{align}
&\left|\widehat{X}_{\underline{s}}^{i,N}+c_{\Delta}(\widehat{X}_{\underline{s}}^{i,N},\mu_{\underline{s}}^{\widehat{\boldsymbol{X}}^N})  z\right|^{p} - |\widehat{X}_{\underline{s}}^{i,N}|^{p} -p |\widehat{X}_{\underline{s}}^{i,N}|^{p-2}\left\langle \widehat{X}_{\underline{s}}^{i,N}, c_{\Delta} (\widehat{X}_{\underline{s}}^{i,N},\mu_{\underline{s}}^{\widehat{\boldsymbol{X}}^N} )z\right\rangle  \notag\\
&\leq  \left|c_{\Delta} (\widehat{X}_{\underline{s}}^{i,N},\mu_{\underline{s}}^{\widehat{\boldsymbol{X}}^N} )\right|^{2}|\widehat{X}_{\underline{s}}^{i,N}|^{p-2}p\Bigg(\dfrac{\vert z\vert^2}{2}+\widetilde{L_3}^{-2}  \left(\left(1+\vert z\vert (\widetilde{L_3}+\upsilon)\right)^{p-1}-\vert z\vert (\widetilde{L_3}+\upsilon)-1\right) \bigg(\vert z\vert \Big(\frac{\widetilde{L_3}}{2}+\upsilon\Big)+\upsilon\bigg)\Bigg)   \notag\\
&\quad+ \sum_{k=2}^{p/2} \sum_{\ell=0}^{k}{p/2 \choose k}|z|^{2k-\ell}Q_{p}\left(p-2,2k-\ell,|\widehat{X}_{\underline{s}}^{i,N}|, 1+\widehat{U}_{\underline{s}}^{i,N}\right), \label{Cond1}
\end{align} 
 where  
   $\widehat{U}_{\underline{s}}^{i,N}:=\dfrac{1}{\sqrt{N}}\sum_{j=1; j\neq i}^{N}\left\vert \widehat{X}_{\underline{s}}^{j,N}\right\vert$, and 
\begin{align*}
&Q_{p}\left(p-2,2k-\ell,|\widehat{X}_{\underline{s}}^{i,N}|, 1+\widehat{U}_{\underline{s}}^{i,N}\right):=2\widetilde{L_3}^{2} \left(\Big(1+\dfrac{1}{\sqrt{N}}\Big)^2|\widehat{X}_{\underline{s}}^{i,N}|^2+\left(1+\widehat{U}_{\underline{s}}^{i,N}\right)^2\right) {k \choose \ell} 2^{\ell}   \widetilde{L_3}^{2k-\ell-2}  \\
&\qquad\times\Bigg( \Big(k-\dfrac{\ell}{2}-1\Big) \big(1+\dfrac{1}{\sqrt{N}}\big)^{2k-\ell-3}\left(1+\widehat{U}_{\underline{s}}^{i,N}\right)^2|\widehat{X}_{\underline{s}}^{i,N}|^{p-4}\\
&\qquad+\sum_{m=2}^{2k-\ell-2} {2k-\ell-2 \choose m}\big(1+\dfrac{1}{\sqrt{N}}\big)^{2k-\ell-2-m}\left(1+\widehat{U}_{\underline{s}}^{i,N}\right)^m|\widehat{X}_{\underline{s}}^{i,N}|^{p-2-m}\Bigg).
\end{align*}
\end{Lem} 

\begin{proof}
Proceeding  in the same way as in \eqref{for1}, we get 
		\begin{align*}
		&\left|\widehat{X}_{\underline{s}}^{i,N}+c_{\Delta}(\widehat{X}_{\underline{s}}^{i,N},\mu_{\underline{s}}^{\widehat{\boldsymbol{X}}^N})  z\right|^{p} - |\widehat{X}_{\underline{s}}^{i,N}|^{p} -p |\widehat{X}_{\underline{s}}^{i,N}|^{p-2}\left\langle \widehat{X}_{\underline{s}}^{i,N}, c_{\Delta} (\widehat{X}_{\underline{s}}^{i,N},\mu_{\underline{s}}^{\widehat{\boldsymbol{X}}^N} )z\right\rangle\\
		&=  \dfrac{p}{2} \left|\widehat{X}_{\underline{s}}^{i,N}\right|^{p-2} \left|c_{\Delta} (\widehat{X}_{\underline{s}}^{i,N},\mu_{\underline{s}}^{\widehat{\boldsymbol{X}}^N} )z\right|^2 \\
		&\quad +\sum_{k=2}^{p/2} {p/2 \choose k}|\widehat{X}_{\underline{s}}^{i,N}|^{p-2k}\sum_{\ell=0}^{k} {k \choose \ell} \left|c_{\Delta} (\widehat{X}_{\underline{s}}^{i,N},\mu_{\underline{s}}^{\widehat{\boldsymbol{X}}^N} )z\right|^{2k-2\ell}2^{\ell} \left( \left\langle \widehat{X}_{\underline{s}}^{i,N}, c_{\Delta} (\widehat{X}_{\underline{s}}^{i,N},\mu_{\underline{s}}^{\widehat{\boldsymbol{X}}^N} )z\right\rangle\right)^{\ell}\\
		& \leq \dfrac{p}{2} \left|\widehat{X}_{\underline{s}}^{i,N}\right|^{p-2} \left|c_{\Delta} (\widehat{X}_{\underline{s}}^{i,N},\mu_{\underline{s}}^{\widehat{\boldsymbol{X}}^N} )z\right|^2 + \sum_{k=2}^{p/2} \sum_{\ell=0}^{k} {p/2 \choose k}   {k \choose \ell} 2^{\ell} |\widehat{X}_{\underline{s}}^{i,N}|^{p-2k+\ell}   \left|c_{\Delta} (\widehat{X}_{\underline{s}}^{i,N},\mu_{\underline{s}}^{\widehat{\boldsymbol{X}}^N} )\right|^{2k-\ell} |z|^{2k-\ell}. 
\end{align*} 
It follows from estimate \eqref{cdelta} of \textbf{T5} and the estimate $\mathcal{W}_2(\mu_{\underline{s}}^{\widehat{\boldsymbol{X}}^N},\delta_0) \leq \frac{1}{\sqrt{N}}\sum_{j=1}^{N}\vert \widehat{X}_{\underline{s}}^{j,N}\vert$ that 
\begin{align}
& \left|c_{\Delta} (\widehat{X}_{\underline{s}}^{i,N},\mu_{\underline{s}}^{\widehat{\boldsymbol{X}}^N} )\right|^{2k-\ell-2}  \leq \widetilde{L_3}^{2k-\ell-2} \left(1+|\widehat{X}_{\underline{s}}^{i,N}|+\mathcal{W}_2(\mu_{\underline{s}}^{\widehat{\boldsymbol{X}}^N},\delta_0)\right)^{2k-\ell-2} \notag \\
 & \leq  \widetilde{L_3}^{2k-\ell-2} \left(1+|\widehat{X}_{\underline{s}}^{i,N}|+\dfrac{1}{\sqrt{N}}\sum_{j=1}^{N}\left\vert \widehat{X}_{\underline{s}}^{j,N}\right\vert\right)^{2k-\ell-2} \notag\\
 &= \widetilde{L_3}^{2k-\ell-2} \left((1+\dfrac{1}{\sqrt{N}})|\widehat{X}_{\underline{s}}^{i,N}|+1+\dfrac{1}{\sqrt{N}}\sum_{j=1; j\neq i}^{N}\left\vert \widehat{X}_{\underline{s}}^{j,N}\right\vert\right)^{2k-\ell-2} \notag \\
 & = \widetilde{L_3}^{2k-\ell-2} \Bigg((1+\dfrac{1}{\sqrt{N}})^{2k-\ell-2}|\widehat{X}_{\underline{s}}^{i,N}|^{2k-\ell-2} +(2k-\ell-2)\left(1+\widehat{U}_{\underline{s}}^{i,N}\right)(1+\dfrac{1}{\sqrt{N}})^{2k-\ell-3}|\widehat{X}_{\underline{s}}^{i,N}|^{2k-\ell-3} \notag \\
&\quad+\sum_{m=2}^{2k-\ell-2} {2k-\ell-2 \choose m} \left(1+\widehat{U}_{\underline{s}}^{i,N}\right)^m  (1+\dfrac{1}{\sqrt{N}})^{2k-\ell-2-m}|\widehat{X}_{\underline{s}}^{i,N}|^{2k-\ell-2-m}\Bigg). \label{ct010ULcDt}
\end{align}	

Using the estimate $\left(1+\widehat{U}_{\underline{s}}^{i,N}\right) |\widehat{X}_{\underline{s}}^{i,N}|^{2k-\ell-3} \leq \frac 12 \Big( \left(1+\widehat{U}_{\underline{s}}^{i,N}\right)^2 |\widehat{X}_{\underline{s}}^{i,N}|^{2k-\ell-4} + |\widehat{X}_{\underline{s}}^{i,N}|^{2k-\ell-2}\Big)$, we get 
\begin{align*}
& \sum_{\ell=0}^{k}   {k \choose \ell} 2^{\ell} |\widehat{X}_{\underline{s}}^{i,N}|^{p-2k+\ell}   \left|c_{\Delta} (\widehat{X}_{\underline{s}}^{i,N},\mu_{\underline{s}}^{\widehat{\boldsymbol{X}}^N} )\right|^{2k-\ell} |z|^{2k-\ell}\\
\leq & \left|c_{\Delta} (\widehat{X}_{\underline{s}}^{i,N},\mu_{\underline{s}}^{\widehat{\boldsymbol{X}}^N} )\right|^{2}|\widehat{X}_{\underline{s}}^{i,N}|^{p-2}\sum_{\ell=0}^{k} {k \choose \ell} 2^{\ell}  |z|^{2k-\ell} \widetilde{L_3}^{2k-\ell-2}\left(1+\dfrac{1}{\sqrt{N}}\right)^{2k-\ell-3}\left( \dfrac{1}{\sqrt{N}}+k-\dfrac{\ell}{2}\right) \\
& + \left|c_{\Delta} (\widehat{X}_{\underline{s}}^{i,N},\mu_{\underline{s}}^{\widehat{\boldsymbol{X}}^N} )\right|^{2}\sum_{\ell=0}^{k} {k \choose \ell} 2^{\ell}  |z|^{2k-\ell} \widetilde{L_3}^{2k-\ell-2}  \Bigg( \left(k-\dfrac{\ell}{2}-1\right) \big(1+\dfrac{1}{\sqrt{N}}\big)^{2k-\ell-3}\left(1+\widehat{U}_{\underline{s}}^{i,N}\right)^2|\widehat{X}_{\underline{s}}^{i,N}|^{p-4}\\
& +\sum_{m=2}^{2k-\ell-2} {2k-\ell-2 \choose m}\big(1+\dfrac{1}{\sqrt{N}}\big)^{2k-\ell-2-m}\left(1+\widehat{U}_{\underline{s}}^{i,N}\right)^m|\widehat{X}_{\underline{s}}^{i,N}|^{p-2-m}\Bigg).
\end{align*}

Set $a=\vert z\vert \widetilde{L_3}\left(1+\dfrac{1}{\sqrt{N}}\right).$ Note that 
\begin{align*}
&\sum_{\ell=0}^{k} {k \choose \ell} 2^{\ell}  |z|^{2k-\ell} \widetilde{L_3}^{2k-\ell-2}\left(1+\dfrac{1}{\sqrt{N}}\right)^{2k-\ell-3}=\widetilde{L_3}^{-2} \left(1+\dfrac{1}{\sqrt{N}}\right)^{-3}\left(a^2+2a\right)^k,\\
&-\dfrac{1}{2}\sum_{\ell=0}^{k} {k \choose \ell} 2^{\ell}  |z|^{2k-\ell} \widetilde{L_3}^{2k-\ell-2}\left(1+\dfrac{1}{\sqrt{N}}\right)^{2k-\ell-3} \ell=- \widetilde{L_3}^{-2} \left(1+\dfrac{1}{\sqrt{N}}\right)^{-3} ka\left(a^2+2a\right)^{k-1}.
\end{align*}
These facts imply that 
$$\sum_{\ell=0}^{k} {k \choose \ell} 2^{\ell}  |z|^{2k-\ell} \widetilde{L_3}^{2k-\ell-2}\left(1+\dfrac{1}{\sqrt{N}}\right)^{2k-\ell-3} (\frac{1}{\sqrt{N}} + k - \frac{\ell }{2}) =\widetilde{L_3}^{-2} \left(1+\dfrac{1}{\sqrt{N}}\right)^{-3}\left(a^2+2a\right)^{k-1} \Big( \frac{a^2+2a}{\sqrt{N}} + k(a^2+a)\Big).$$
Moreover, similar to the estimate \eqref{ct010ULcDt}, we get 
\begin{align*}
\left|c_{\Delta} (\widehat{X}_{\underline{s}}^{i,N},\mu_{\underline{s}}^{\widehat{\boldsymbol{X}}^N} )\right|^{2}&\leq  2\widetilde{L_3}^{2} \left(\Big(1+\dfrac{1}{\sqrt{N}}\Big)^2|\widehat{X}_{\underline{s}}^{i,N}|^2+\left(1+\widehat{U}_{\underline{s}}^{i,N}\right)^2\right).
\end{align*}
Therefore, we have 
\begin{align*}
& \sum_{\ell=0}^{k}   {k \choose \ell} 2^{\ell} |\widehat{X}_{\underline{s}}^{i,N}|^{p-2k+\ell}   \left|c_{\Delta} (\widehat{X}_{\underline{s}}^{i,N},\mu_{\underline{s}}^{\widehat{\boldsymbol{X}}^N} )\right|^{2k-\ell} |z|^{2k-\ell}\\
\leq & \left|c_{\Delta} (\widehat{X}_{\underline{s}}^{i,N},\mu_{\underline{s}}^{\widehat{\boldsymbol{X}}^N} )\right|^{2}|\widehat{X}_{\underline{s}}^{i,N}|^{p-2} \widetilde{L_3}^{-2} \left(1+\dfrac{1}{\sqrt{N}}\right)^{-3}\left(a^2+2a\right)^{k-1}   \left(\dfrac{a^2+2a}{\sqrt{N}}+k(a^2+a)\right) \\
&\quad+ \sum_{\ell=0}^{k}|z|^{2k-\ell}Q_{p}\left(p-2,2k-\ell,|\widehat{X}_{\underline{s}}^{i,N}|, 1+\widehat{U}_{\underline{s}}^{i,N}\right).
\end{align*}
Next, using $\sum_{k=2}^{p/2}{p/2 \choose k}k z^{k-1}=\frac{p}{2}((1+z)^{p/2-1}-1)$ and $\sum_{k=2}^{p/2}{p/2 \choose k} z^{k-1}=z^{-1}((1+z)^{p/2}-1-\frac{p}{2}z)$ with $z>0$, we get that
\begin{align*}
&\sum_{k=2}^{p/2} {p/2 \choose k}|\widehat{X}_{\underline{s}}^{i,N}|^{p-2k} \left(\left|c_{\Delta} (\widehat{X}_{\underline{s}}^{i,N},\mu_{\underline{s}}^{\widehat{\boldsymbol{X}}^N} )z \right|^2+2\left\langle \widehat{X}_{\underline{s}}^{i,N}, c_{\Delta} (\widehat{X}_{\underline{s}}^{i,N},\mu_{\underline{s}}^{\widehat{\boldsymbol{X}}^N} )z\right\rangle \right)^k\\
&\leq \left|c_{\Delta} (\widehat{X}_{\underline{s}}^{i,N},\mu_{\underline{s}}^{\widehat{\boldsymbol{X}}^N} )\right|^{2}|\widehat{X}_{\underline{s}}^{i,N}|^{p-2} \widetilde{L_3}^{-2} \left(1+\dfrac{1}{\sqrt{N}}\right)^{-3}\\
&\quad\times\Bigg(\dfrac{1}{\sqrt{N}}\bigg((1+a)^p-1-\frac{p}{2}(a^2+2a)\bigg)+(a^2+a)\frac{p}{2}\Big((1+a)^{p-2}-1\Big)\Bigg)\\
&\quad+ \sum_{k=2}^{p/2} \sum_{\ell=0}^{k}{p/2 \choose k}|z|^{2k-\ell}Q_{p}\left(p-2,2k-\ell,|\widehat{X}_{\underline{s}}^{i,N}|, 1+\widehat{U}_{\underline{s}}^{i,N}\right).
\end{align*}
Applying the inequality 
$(1+a)^p-1-\frac{p}{2}(a^2+2a) \leq \left(a+1\right)^2\left((1+a)^{p-2}-1\right),$
with $p\geq 2$ and $a>0$, we get that for $\upsilon \geq \frac{1}{2\sqrt{N}}\max\{3\widetilde{L_3},1\}$,
\begin{align}
&\left|\widehat{X}_{\underline{s}}^{i,N}+c_{\Delta}(\widehat{X}_{\underline{s}}^{i,N},\mu_{\underline{s}}^{\widehat{\boldsymbol{X}}^N})  z\right|^{p} - |\widehat{X}_{\underline{s}}^{i,N}|^{p} -p |\widehat{X}_{\underline{s}}^{i,N}|^{p-2}\left\langle \widehat{X}_{\underline{s}}^{i,N}, c_{\Delta} (\widehat{X}_{\underline{s}}^{i,N},\mu_{\underline{s}}^{\widehat{\boldsymbol{X}}^N} )z\right\rangle  \notag\\
&\leq \left|c_{\Delta} (\widehat{X}_{\underline{s}}^{i,N},\mu_{\underline{s}}^{\widehat{\boldsymbol{X}}^N} )\right|^{2}|\widehat{X}_{\underline{s}}^{i,N}|^{p-2}\Bigg(\dfrac{p}{2}\vert z\vert^2+\widetilde{L_3}^{-2} \left(1+\dfrac{1}{\sqrt{N}}\right)^{-3}  \notag\\
&\quad\times\Bigg(\dfrac{1}{\sqrt{N}}\bigg((1+a)^p-1-\frac{p}{2}(a^2+2a)\bigg)+(a^2+a)\frac{p}{2}\Big((1+a)^{p-2}-1\Big)\Bigg)\Bigg) \notag\\
&\quad  + \sum_{k=2}^{p/2} \sum_{\ell=0}^{k}{p/2 \choose k}|z|^{2k-\ell}Q_{p}\left(p-2,2k-\ell,|\widehat{X}_{\underline{s}}^{i,N}|, 1+\widehat{U}_{\underline{s}}^{i,N}\right) \notag\\
&\leq  \left|c_{\Delta} (\widehat{X}_{\underline{s}}^{i,N},\mu_{\underline{s}}^{\widehat{\boldsymbol{X}}^N} )\right|^{2}|\widehat{X}_{\underline{s}}^{i,N}|^{p-2}\Bigg(\dfrac{p}{2}\vert z\vert^2+\widetilde{L_3}^{-2} \left(1+\dfrac{1}{\sqrt{N}}\right)^{-3}  \Big((1+a)^{p-2}-1\Big) \left(a+1\right)\Big(\dfrac{a+1}{\sqrt{N}}+\dfrac{ap}{2}\Big)\Bigg)   \notag\\
&\quad+ \sum_{k=2}^{p/2} \sum_{\ell=0}^{k}{p/2 \choose k}|z|^{2k-\ell}Q_{p}\left(p-2,2k-\ell,|\widehat{X}_{\underline{s}}^{i,N}|, 1+\widehat{U}_{\underline{s}}^{i,N}\right) \notag\\
&\leq  \left|c_{\Delta} (\widehat{X}_{\underline{s}}^{i,N},\mu_{\underline{s}}^{\widehat{\boldsymbol{X}}^N} )\right|^{2}|\widehat{X}_{\underline{s}}^{i,N}|^{p-2}p\Bigg(\dfrac{\vert z\vert^2}{2}+\widetilde{L_3}^{-2}  \Big((1+a)^{p-1}-a-1\Big) \Big(\dfrac{a+1}{p\sqrt{N}}+\dfrac{a}{2}\Big)\Bigg) \notag\\
&\quad  + \sum_{k=2}^{p/2} \sum_{\ell=0}^{k}{p/2 \choose k}|z|^{2k-\ell}Q_{p}\left(p-2,2k-\ell,|\widehat{X}_{\underline{s}}^{i,N}|, 1+\widehat{U}_{\underline{s}}^{i,N}\right)\notag\\
&\leq  \left|c_{\Delta} (\widehat{X}_{\underline{s}}^{i,N},\mu_{\underline{s}}^{\widehat{\boldsymbol{X}}^N} )\right|^{2}|\widehat{X}_{\underline{s}}^{i,N}|^{p-2}p\Bigg(\dfrac{\vert z\vert^2}{2}+\widetilde{L_3}^{-2}  \left(\left(1+\vert z\vert (\widetilde{L_3}+\upsilon)\right)^{p-1}-\vert z\vert (\widetilde{L_3}+\upsilon)-1\right) \bigg(\vert z\vert \Big(\frac{\widetilde{L_3}}{2}+\upsilon\Big)+\upsilon\bigg)\Bigg)   \notag\\
&\quad+ \sum_{k=2}^{p/2} \sum_{\ell=0}^{k}{p/2 \choose k}|z|^{2k-\ell}Q_{p}\left(p-2,2k-\ell,|\widehat{X}_{\underline{s}}^{i,N}|, 1+\widehat{U}_{\underline{s}}^{i,N}\right), \notag
\end{align} 
where we have used the fact that $
a\leq \vert z\vert (\widetilde{L_3}+\upsilon),\quad\dfrac{a+1}{p\sqrt{N}}+\dfrac{a}{2} \leq \vert z\vert \Big(\frac{\widetilde{L_3}}{2}+\upsilon\Big)+\upsilon.
$
\end{proof} 

\begin{Lem}
Assume  Conditions \textbf{T1}--\textbf{T5} and \textbf{A5}. 
Then, for any even integer   $p \in (0, p_0]$,   there exists a positive constant $C_p$ such that for any $s>0$ and $z \in \mathbb{R}^d$, 

			\begin{align}
			&\mathbb{E}\left[|\widehat{X}_{s}^{i,N}|^{p-2}\langle \widehat{X}_{s}^{i,N}, c_{\Delta} (\widehat{X}_{\underline{s}}^{i,N},\mu_{\underline{s}}^{\widehat{\boldsymbol{X}}^N} )z\rangle- |\widehat{X}_{\underline{s}}^{i,N}|^{p-2}\langle \widehat{X}_{\underline{s}}^{i,N}, c_{\Delta} (\widehat{X}_{\underline{s}}^{i,N},\mu_{\underline{s}}^{\widehat{\boldsymbol{X}}^N} )z\rangle\big| \mathcal{F}_{\underline{s}} \right] \notag\\
			&\leq C_p\vert z\vert\left(\sum\limits_{j=0}^{p-2} |\widehat{X}_{\underline{s}}^{i,N}|^j + \mathcal{W}_2(\mu_{\underline{s}}^{\widehat{\boldsymbol{X}}^N},\delta_0)\sum\limits_{j=0}^{p-3} |\widehat{X}_{\underline{s}}^{i,N}|^j\right). \label{Cond4}
			\end{align}	

\end{Lem}

\begin{proof}

			By using the binomial theorem, \eqref{EM3} and \eqref{cdelta} of \textbf{T5}, we have
			\begin{align*}
			& |\widehat{X}_{s}^{i,N}|^{p-2}\langle \widehat{X}_{s}^{i,N}, c_{\Delta} (\widehat{X}_{\underline{s}}^{i,N},\mu_{\underline{s}}^{\widehat{\boldsymbol{X}}^N} )z\rangle- |\widehat{X}_{\underline{s}}^{i,N}|^{p-2}\langle \widehat{X}_{\underline{s}}^{i,N}, c_{\Delta} (\widehat{X}_{\underline{s}}^{i,N},\mu_{\underline{s}}^{\widehat{\boldsymbol{X}}^N} )z\rangle \\
			&=\langle \widehat{X}_{\underline{s}}^{i,N}, c_{\Delta} (\widehat{X}_{\underline{s}}^{i,N},\mu_{\underline{s}}^{\widehat{\boldsymbol{X}}^N} )z\rangle\left(|\widehat{X}_{s}^{i,N}|^{p-2}-|\widehat{X}_{\underline{s}}^{i,N}|^{p-2}\right)+\langle \widehat{X}_{s}^{i,N}-\widehat{X}_{\underline{s}}^{i,N}, c_{\Delta} (\widehat{X}_{\underline{s}}^{i,N},\mu_{\underline{s}}^{\widehat{\boldsymbol{X}}^N} )z\rangle|\widehat{X}_{s}^{i,N}|^{p-2}  \\
			&=\langle \widehat{X}_{\underline{s}}^{i,N}, c_{\Delta} (\widehat{X}_{\underline{s}}^{i,N},\mu_{\underline{s}}^{\widehat{\boldsymbol{X}}^N} )z\rangle\sum_{j=1}^{p-2} {p-2\choose j} |\widehat{X}_{\underline{s}}^{i,N}|^{p-2-j} \left(|\widehat{X}_{s}^{i,N}|-|\widehat{X}_{\underline{s}}^{i,N}|\right)^j\\
			&\quad+\langle \widehat{X}_{s}^{i,N}-\widehat{X}_{\underline{s}}^{i,N}, c_{\Delta} (\widehat{X}_{\underline{s}}^{i,N},\mu_{\underline{s}}^{\widehat{\boldsymbol{X}}^N} )z\rangle|\widehat{X}_{\underline{s}}^{i,N}|^{p-2} \\
			&\quad +\langle \widehat{X}_{s}^{i,N}-\widehat{X}_{\underline{s}}^{i,N}, c_{\Delta} (\widehat{X}_{\underline{s}}^{i,N},\mu_{\underline{s}}^{\widehat{\boldsymbol{X}}^N} )z\rangle\sum_{j=1}^{p-2} {p-2\choose j} |\widehat{X}_{\underline{s}}^{i,N}|^{p-2-j} \left(|\widehat{X}_{s}^{i,N}|-|\widehat{X}_{\underline{s}}^{i,N}|\right)^j  \\
			&\leq (p-2) \vert c_{\Delta} (\widehat{X}_{\underline{s}}^{i,N},\mu_{\underline{s}}^{\widehat{\boldsymbol{X}}^N} )\vert \vert z\vert |\widehat{X}_{\underline{s}}^{i,N}|^{p-2} \left|\widehat{X}_{s}^{i,N}-\widehat{X}_{\underline{s}}^{i,N}\right|\\
			&\quad+\vert c_{\Delta} (\widehat{X}_{\underline{s}}^{i,N},\mu_{\underline{s}}^{\widehat{\boldsymbol{X}}^N} )\vert \vert z\vert\sum_{j=2}^{p-2} {p-2\choose j} |\widehat{X}_{\underline{s}}^{i,N}|^{p-1-j} \left|\widehat{X}_{s}^{i,N}-\widehat{X}_{\underline{s}}^{i,N}\right|^j  \\
			&\quad +\left\langle b\left(\widehat{X}_{\underline{s}}^{i,N},\mu_{\underline{s}}^{\widehat{\boldsymbol{X}}^N}\right)(s-\underline{s}), c_{\Delta} (\widehat{X}_{\underline{s}}^{i,N},\mu_{\underline{s}}^{\widehat{\boldsymbol{X}}^N} )z\right\rangle|\widehat{X}_{\underline{s}}^{i,N}|^{p-2} \\
			&\quad+ \left\langle \sigma_{\Delta}\left(\widehat{X}_{\underline{s}}^{i,N},\mu_{\underline{s}}^{\widehat{\boldsymbol{X}}^N}\right)\left(W_{s}^i-W_{\underline{s}}^i\right)+c_{\Delta}\left(\widehat{X}_{\underline{s}}^{i,N},\mu_{\underline{s}}^{\widehat{\boldsymbol{X}}^N}\right)\left(Z_{s}^i-Z_{\underline{s}}^i\right), c_{\Delta} (\widehat{X}_{\underline{s}}^{i,N},\mu_{\underline{s}}^{\widehat{\boldsymbol{X}}^N} )z\right\rangle|\widehat{X}_{\underline{s}}^{i,N}|^{p-2}  \\
			&\quad +\vert c_{\Delta} (\widehat{X}_{\underline{s}}^{i,N},\mu_{\underline{s}}^{\widehat{\boldsymbol{X}}^N} )\vert \vert z\vert \sum_{j=1}^{p-2} {p-2\choose j} |\widehat{X}_{\underline{s}}^{i,N}|^{p-2-j} \left|\widehat{X}_{s}^{i,N}-\widehat{X}_{\underline{s}}^{i,N}\right|^{j+1}\\
			&\leq (p-2) \vert z\vert |\widehat{X}_{\underline{s}}^{i,N}|^{p-2} \bigg(|c_{\Delta}(\widehat{X}_{\underline{s}}^{i,N},\mu_{\underline{s}}^{\widehat{\boldsymbol{X}}^N})|  |b(\widehat{X}_{\underline{s}}^{i,N},\mu_{\underline{s}}^{\widehat{\boldsymbol{X}}^N})|(s-\underline{s})  + |c_{\Delta}(\widehat{X}_{\underline{s}}^{i,N},\mu_{\underline{s}}^{\widehat{\boldsymbol{X}}^N})|  |\sigma_{\Delta}\left(\widehat{X}_{\underline{s}}^{i,N},\mu_{\underline{s}}^{\widehat{\boldsymbol{X}}^N}\right)|\left|W_{s}^i-W_{\underline{s}}^i\right| \\
			&\quad+ |c_{\Delta}(\widehat{X}_{\underline{s}}^{i,N},\mu_{\underline{s}}^{\widehat{\boldsymbol{X}}^N})|^2\left|Z_{s}^i-Z_{\underline{s}}^i\right| \bigg) +\widetilde{L_3}\left(1+\vert \widehat{X}_{\underline{s}}^{i,N}\vert+\mathcal{W}_2(\mu_{\underline{s}}^{\widehat{\boldsymbol{X}}^N},\delta_0)\right) \vert z\vert\sum_{j=2}^{p-2} {p-2\choose j} |\widehat{X}_{\underline{s}}^{i,N}|^{p-1-j} 3^{j-1}\\
			&\qquad\times\bigg(\left|b\left(\widehat{X}_{\underline{s}}^{i,N},\mu_{\underline{s}}^{\widehat{\boldsymbol{X}}^N}\right)\right|^j \left|s-\underline{s}\right|^j + \left|\sigma_{\Delta}\left(\widehat{X}_{\underline{s}}^{i,N},\mu_{\underline{s}}^{\widehat{\boldsymbol{X}}^N}\right)\right|^j  \left| W_{s}^i-W_{\underline{s}}^i\right|^j  + \left|c_{\Delta}\left(\widehat{X}_{\underline{s}}^{i,N},\mu_{\underline{s}}^{\widehat{\boldsymbol{X}}^N}\right) \right|^j \left|Z_{s}^i-Z_{\underline{s}}^i\right|^j\bigg)  \\
			&\quad +\left\langle b\left(\widehat{X}_{\underline{s}}^{i,N},\mu_{\underline{s}}^{\widehat{\boldsymbol{X}}^N}\right)(s-\underline{s}), c_{\Delta} (\widehat{X}_{\underline{s}}^{i,N},\mu_{\underline{s}}^{\widehat{\boldsymbol{X}}^N} )z\right\rangle|\widehat{X}_{\underline{s}}^{i,N}|^{p-2} + \Big\langle \sigma_{\Delta}\left(\widehat{X}_{\underline{s}}^{i,N},\mu_{\underline{s}}^{\widehat{\boldsymbol{X}}^N}\right)\left(W_{s}^i-W_{\underline{s}}^i\right)\\
			&\quad+c_{\Delta}\left(\widehat{X}_{\underline{s}}^{i,N},\mu_{\underline{s}}^{\widehat{\boldsymbol{X}}^N}\right)\left(Z_{s}^i-Z_{\underline{s}}^i\right), c_{\Delta} (\widehat{X}_{\underline{s}}^{i,N},\mu_{\underline{s}}^{\widehat{\boldsymbol{X}}^N} )z\Big\rangle|\widehat{X}_{\underline{s}}^{i,N}|^{p-2}  +\widetilde{L_3}\left(1+\vert \widehat{X}_{\underline{s}}^{i,N}\vert+\mathcal{W}_2(\mu_{\underline{s}}^{\widehat{\boldsymbol{X}}^N},\delta_0)\right) \vert z\vert \\
			&\qquad \times\sum_{j=1}^{p-2} {p-2\choose j} |\widehat{X}_{\underline{s}}^{i,N}|^{p-2-j} 3^{j}\bigg(\left|b\left(\widehat{X}_{\underline{s}}^{i,N},\mu_{\underline{s}}^{\widehat{\boldsymbol{X}}^N}\right)\right|^{j+1} \left|s-\underline{s}\right|^{j+1} + \left|\sigma_{\Delta}\left(\widehat{X}_{\underline{s}}^{i,N},\mu_{\underline{s}}^{\widehat{\boldsymbol{X}}^N}\right)\right|^{j+1}  \left| W_{s}^i-W_{\underline{s}}^i\right|^{j+1} \\
			&\quad + \left|c_{\Delta}\left(\widehat{X}_{\underline{s}}^{i,N},\mu_{\underline{s}}^{\widehat{\boldsymbol{X}}^N}\right) \right|^{j+1} \left|Z_{s}^i-Z_{\underline{s}}^i\right|^{j+1}\bigg).
			\end{align*}
			By applying the estimates \eqref{con-moment} and condition \textbf{T4}, we obtain the desired result. 			
\end{proof}

	\begin{Prop}\label{dinh ly 2}
		Assume that  Conditions \textbf{T1}--\textbf{T5} and \textbf{A5} hold, and  $N \geq \left( \frac{\max\{3\widetilde{L_3},1\}}{2\upsilon}\right)^2$. 
		Then, for any positive  $k \leq  p_0/2$,  there exists a positive constant $C=C(x_0,k,\widetilde{\gamma}_1,\widetilde{\gamma}_2,\widetilde{\eta},L,\widetilde{L_3},p_0)$ which depends neither on $\Delta$  nor on $t$ such that for any $t\geq 0$, 
		\begin{align}
			\max_{i\in\{1,\ldots,N\}}\left(\bE \left[|\widehat{X}_t^{i,N}|^{2k} \right] \vee \bE \left[|\widehat{X}_{\underline{t}}^{i,N}|^{2k}\right]\right) \le \left\{ \begin{array}{l l}
				Ce^{{2k\widetilde{\gamma} t}}\quad &\text{ if\;\; } \widetilde{\gamma} > 0,\\
				C(1+t)^{k} \quad &\text{ if \;\;} \widetilde{\gamma}  = 0, \\
				C \quad &\text{ if \;\;} \widetilde{\gamma}  < 0,
			\end{array} \right. \label{EXtmu}
		\end{align}
		where $\widetilde{\gamma} =\widetilde{\gamma}_1 +\widetilde{\gamma}_2$.
	\end{Prop}
	\begin{proof}
		Using H\"older's inequality, it suffices to show \eqref{EXtmu} for a positive integer $k \leq p_0/2$. We are going to use the induction method. First, note that \eqref{EXtmu} is valid for $k=1$ due to Lemma \ref{indu2}.
		
		Next, assume that \eqref{EXtmu} holds for any $k \leq k_0 \leq [p_0/2]-1$. 
		We wish to show that \eqref{EXtmu} still holds for $k=k_0+1$. For this, using It\^o's formula for $e^{-p\lambda t}|\widehat{X}_t^{i,N}|^p$ with even integer $p:=2(k_0+1)$, we have 
\begin{align}\label{e^pX^p}
e^{-p\lambda t}\left|\widehat{X}_t^{i,N}\right|^p 
			=|x_0|^p+ \int_0^t e^{-p\lambda s} \mathcal{R}_s ds  + \mathcal{M}_t,
\end{align}			
where 
		\begin{align*}
		\mathcal{R}_s &=	-p\lambda |\widehat{X}_s^{i,N}|^p +  p|\widehat{X}_s^{i,N}|^{p-2} \left\langle \widehat{X}_s^{i,N}, b(\widehat{X}_{\underline{s}}^{i,N},\mu_{\underline{s}}^{\widehat{\boldsymbol{X}}^N}) \right\rangle \notag\\
			&\quad + \dfrac{p}{2} |\widehat{X}_s^{i,N}|^{p-2} \left|\sm_{\Delta} (\widehat{X}_{\underline{s}}^{i,N},\mu_{\underline{s}}^{\widehat{\boldsymbol{X}}^N})\right|^2 + \dfrac{p(p-2)}{2} |\widehat{X}_s^{i,N}|^{p-4} \left|(\widehat{X}_s^{i,N})^\mathsf{T} \sm_{\Delta} (\widehat{X}_{\underline{s}}^{i,N},\mu_{\underline{s}}^{\widehat{\boldsymbol{X}}^N})\right|^2  \notag \\
			&\quad +  \int_{\mathbb{R}_0^d}  \left( \left|\widehat{X}_{s}^{i,N}+c_{\Delta}(\widehat{X}_{\underline{s}}^{i,N},\mu_{\underline{s}}^{\widehat{\boldsymbol{X}}^N})  z\right|^{p} - |\widehat{X}_{s}^{i,N}|^{p} -p |\widehat{X}_{s}^{i,N}|^{p-2}\left\langle \widehat{X}_s^{i,N}, c_{\Delta} (\widehat{X}_{\underline{s}}^{i,N},\mu_{\underline{s}}^{\widehat{\boldsymbol{X}}^N} )z\right\rangle \right) \nu\left(dz \right),  \notag \\
		\mathcal{M}_t 	&=  p\int_{0}^{t} e^{-p\lambda s} |\widehat{X}_s^{i,N}|^{p-2} \left\langle \widehat{X}_s^{i,N}, \sm_{\Delta}(\widehat{X}_{\underline{s}}^{i,N},\mu_{\underline{s}}^{\widehat{\boldsymbol{X}}^N}) dW_s^i \right\rangle \notag\\ &\quad+\int_{0}^{t} \int_{\mathbb{R}_0^d} e^{-p \lambda s}\left( \left|\widehat{X}_{s-}^{i,N}+c_{\Delta}(\widehat{X}_{\underline{s}-}^{i,N},\mu_{\underline{s}-}^{\widehat{\boldsymbol{X}}^N})  z\right|^{p} - |\widehat{X}_{s-}^{i,N}|^{p} \right) \widetilde{N}^i(d s, d z). 
		\end{align*}
If follows from Lemma  \ref{Lem:condmom} that 
		\begin{align}
		& \bE \bigg[-{p\lambda} |\widehat{X}_s^{i,N}|^p +  p|\widehat{X}_s^{i,N}|^{p-2} \left\langle \widehat{X}_s^{i,N}, b(\widehat{X}_{\underline{s}}^{i,N},\mu_{\underline{s}}^{\widehat{\boldsymbol{X}}^N}) \right\rangle + \dfrac{p}{2} |\widehat{X}_s^{i,N}|^{p-2} \left|\sm_{\Delta} (\widehat{X}_{\underline{s}}^{i,N},\mu_{\underline{s}}^{\widehat{\boldsymbol{X}}^N})\right|^2\notag\\
		&\quad  + \dfrac{p(p-2)}{2} |\widehat{X}_s^{i,N}|^{p-4} \left|(\widehat{X}_s^{i,N})^\mathsf{T} \sm_{\Delta} (\widehat{X}_{\underline{s}}^{i,N},\mu_{\underline{s}}^{\widehat{\boldsymbol{X}}^N})\right|^2 \Big|\mathcal{F}_{\underline{s}}    \bigg]	  \notag\\
		&\le p|\widehat{X}_{\underline{s}}^{i,N}|^{p-2}\left(-\lambda |\widehat{X}_{\underline{s}}^{i,N}|^2 +  \left\langle \widehat{X}_{\underline{s}}^{i,N}, b(\widehat{X}_{\underline{s}}^{i,N},\mu_{\underline{s}}^{\widehat{\boldsymbol{X}}^N}) \right\rangle + \dfrac{p-1}{2}|\sigma_\Delta(\widehat{X}_{\underline{s}}^{i,N},\mu_{\underline{s}}^{\widehat{\boldsymbol{X}}^N})|^2 \right)+ \overline{Q}_{p-2}\left(\left|\widehat{X}_{\underline{s}}^{i,N}\right|\right), \label{continuouspart}
		\end{align}
		where $\overline{Q}_{p-2}(\vert\widehat{X}_{\underline{s}}^{i,N}\vert)$ is a polynomial in $|\widehat{X}_{\underline{s}}^{i,N}|$ of degree $p-2$.
		
We write
		\begin{align}
			&\left|\widehat{X}_{s}^{i,N}+c_{\Delta}(\widehat{X}_{\underline{s}}^{i,N},\mu_{\underline{s}}^{\widehat{\boldsymbol{X}}^N})  z\right|^{p} - |\widehat{X}_{s}^{i,N}|^{p} -p |\widehat{X}_{s}^{i,N}|^{p-2}\left\langle \widehat{X}_s^{i,N}, c_{\Delta} (\widehat{X}_{\underline{s}}^{i,N},\mu_{\underline{s}}^{\widehat{\boldsymbol{X}}^N} )z\right\rangle \notag \\
			&= \left|\widehat{X}_{\underline{s}}^{i,N}+c_{\Delta}(\widehat{X}_{\underline{s}}^{i,N},\mu_{\underline{s}}^{\widehat{\boldsymbol{X}}^N})  z\right|^{p} - |\widehat{X}_{\underline{s}}^{i,N}|^{p} -p |\widehat{X}_{\underline{s}}^{i,N}|^{p-2}\left\langle \widehat{X}_{\underline{s}}^{i,N}, c_{\Delta} (\widehat{X}_{\underline{s}}^{i,N},\mu_{\underline{s}}^{\widehat{\boldsymbol{X}}^N} )z\right\rangle \notag \\	  
			&\quad  + \left(\left|\widehat{X}_{s}^{i,N}+c_{\Delta}(\widehat{X}_{\underline{s}}^{i,N},\mu_{\underline{s}}^{\widehat{\boldsymbol{X}}^N})  z\right|^{p}-\left|\widehat{X}_{\underline{s}}^{i,N}+c_{\Delta}(\widehat{X}_{\underline{s}}^{i,N},\mu_{\underline{s}}^{\widehat{\boldsymbol{X}}^N})  z\right|^{p}\right)-\left(|\widehat{X}_{s}^{i,N}|^{p}-|\widehat{X}_{\underline{s}}^{i,N}|^{p}\right)\notag\\
			&\quad -p\left( |\widehat{X}_{s}^{i,N}|^{p-2}\left\langle \widehat{X}_s^{i,N}, c_{\Delta} (\widehat{X}_{\underline{s}}^{i,N},\mu_{\underline{s}}^{\widehat{\boldsymbol{X}}^N} )z\right\rangle- |\widehat{X}_{\underline{s}}^{i,N}|^{p-2}\left\langle \widehat{X}_{\underline{s}}^{i,N}, c_{\Delta} (\widehat{X}_{\underline{s}}^{i,N},\mu_{\underline{s}}^{\widehat{\boldsymbol{X}}^N} )z\right\rangle\right).\label{c}
		\end{align}

			Therefore, taking the conditional expectation on both sides of \eqref{c} and inserting   \eqref{C52}, \eqref{Cond1}, and \eqref{Cond4} into the right hand side, we obtain that for $\upsilon\geq \frac{1}{2\sqrt{N}}\max\{3\widetilde{L_3},1\}$,
			\begin{align}
			& \mathbb{E}\left[\left|\widehat{X}_{s}^{i,N}+c_{\Delta}\left(\widehat{X}_{\underline{s}}^{i,N},\mu_{\underline{s}}^{\widehat{\boldsymbol{X}}^N}\right) z\right|^{p}-|\widehat{X}_{s}^{i,N}|^{p}-p |\widehat{X}_{s}^{i,N}|^{p-2} \left\langle \widehat{X}_{s}^{i,N}, c_{\Delta}\left(\widehat{X}_{\underline{s}}^{i,N},\mu_{\underline{s}}^{\widehat{\boldsymbol{X}}^N}\right) z \right\rangle \big| \mathcal{F}_{\underline{s}}\right] \notag \\
			&\le  p\left|c_{\Delta} (\widehat{X}_{\underline{s}}^{i,N},\mu_{\underline{s}}^{\widehat{\boldsymbol{X}}^N} )\right|^{2}|\widehat{X}_{\underline{s}}^{i,N}|^{p-2}\Bigg(\dfrac{\vert z\vert^2}{2}+\widetilde{L_3}^{-2}  \left(\left(1+\vert z\vert (\widetilde{L_3}+\upsilon)\right)^{p-1}-\vert z\vert (\widetilde{L_3}+\upsilon)-1\right) \bigg(\vert z\vert \Big(\frac{\widetilde{L_3}}{2}+\upsilon\Big)+\upsilon\bigg)\Bigg)    \notag\\
			&\quad+ \sum_{k=2}^{p/2} \sum_{\ell=0}^{k}{p/2 \choose k}|z|^{2k-\ell}Q_{p}\left(p-2,2k-\ell,|\widehat{X}_{\underline{s}}^{i,N}|, 1+\widehat{U}_{\underline{s}}^{i,N}\right)+\widehat{Q}_{p-2}\left(\left|\widehat{X}_{\underline{s}}^{i,N}\right|,|z|\right)\notag\\
			&\quad+  \sum_{j=1}^{p/2}\sum_{k=0}^{j}C_{j,k} \vert z\vert^{j+k}\mathcal{W}_2^{j+k}(\mu_{\underline{s}}^{\widehat{\boldsymbol{X}}^N},\delta_0) \bigg(\left\vert \widehat{X}_{\underline{s}}^{i,N}\right\vert^{j-k-2}+\left|\widehat{X}_{\underline{s}}^{i,N}\right\vert^{p-j-k-2} +\sum_{\ell=2}^{j-k}\left(\left\vert \widehat{X}_{\underline{s}}^{i,N}\right\vert^{j-k-\ell}+\left|\widehat{X}_{\underline{s}}^{i,N}\right\vert^{p-j-k-\ell}\right)\notag\\
			&\quad+\sum_{\ell=2}^{p-2j}\left|\widehat{X}_{\underline{s}}^{i,N}\right\vert^{p-j-k-\ell}\bigg)+C_p\vert z\vert\left(\sum\limits_{j=0}^{p-2} |\widehat{X}_{\underline{s}}^{i,N}|^j + \mathcal{W}_2(\mu_{\underline{s}}^{\widehat{\boldsymbol{X}}^N},\delta_0)\sum\limits_{j=0}^{p-3} |\widehat{X}_{\underline{s}}^{i,N}|^j\right).
			  \label{Cond5}
			\end{align}		
			Hence, combining \eqref{continuouspart} and  \eqref{Cond5}, we get 
			\begin{align*}
			 \bE [ \mathcal{R}_s |\mathcal{F}_{\underline{s}}   ]	  			&\le p|\widehat{X}_{\underline{s}}^{i,N}|^{p-2}\Bigg(-\lambda |\widehat{X}_{\underline{s}}^{i,N}|^2 +  \langle \widehat{X}_{\underline{s}}^{i,N}, b(\widehat{X}_{\underline{s}}^{i,N},\mu_{\underline{s}}^{\widehat{\boldsymbol{X}}^N}) \rangle + \dfrac{p-1}{2}|\sigma_\Delta(\widehat{X}_{\underline{s}}^{i,N},\mu_{\underline{s}}^{\widehat{\boldsymbol{X}}^N})|^2 \notag   \\
			&\quad + \left|c_{\Delta}\left(\widehat{X}_{\underline{s}}^{i,N},\mu_{\underline{s}}^{\widehat{\boldsymbol{X}}^N}\right)\right|^{2} \int_{\bR_{0}^{d}} \Bigg(\dfrac{\vert z\vert^2}{2}+\widetilde{L_3}^{-2}  \left(\left(1+\vert z\vert (\widetilde{L_3}+\upsilon)\right)^{p-1}-\vert z\vert (\widetilde{L_3}+\upsilon)-1\right) \bigg(\vert z\vert \Big(\frac{\widetilde{L_3}}{2}+\upsilon\Big)+\upsilon\bigg)\Bigg)  \nu(dz)\Bigg)\\
			&\quad + \overline{Q}_{p-2}\left(\left|\widehat{X}_{\underline{s}}^{i,N}\right|\right)+\int_{\bR_{0}^{d}} \Bigg(\sum_{k=2}^{p/2} \sum_{\ell=0}^{k}{p/2 \choose k}|z|^{2k-\ell}Q_{p}\left(p-2,2k-\ell,|\widehat{X}_{\underline{s}}^{i,N}|, 1+\widehat{U}_{\underline{s}}^{i,N}\right)+\widehat{Q}_{p-2}\left(\left|\widehat{X}_{\underline{s}}^{i,N}\right|,|z|\right)\notag\\
			&\quad+  \sum_{j=1}^{p/2}\sum_{k=0}^{j}C_{j,k} \vert z\vert^{j+k}\mathcal{W}_2^{j+k}(\mu_{\underline{s}}^{\widehat{\boldsymbol{X}}^N},\delta_0) \bigg(\left\vert \widehat{X}_{\underline{s}}^{i,N}\right\vert^{j-k-2}+\left|\widehat{X}_{\underline{s}}^{i,N}\right\vert^{p-j-k-2} +\sum_{\ell=2}^{j-k}\left(\left\vert \widehat{X}_{\underline{s}}^{i,N}\right\vert^{j-k-\ell}+\left|\widehat{X}_{\underline{s}}^{i,N}\right\vert^{p-j-k-\ell}\right)\notag\\
			&\quad+\sum_{\ell=2}^{p-2j}\left|\widehat{X}_{\underline{s}}^{i,N}\right\vert^{p-j-k-\ell}\bigg)+C_p\vert z\vert\bigg(\sum\limits_{j=0}^{p-2} |\widehat{X}_{\underline{s}}^{i,N}|^j + \mathcal{W}_2(\mu_{\underline{s}}^{\widehat{\boldsymbol{X}}^N},\delta_0)\sum\limits_{j=0}^{p-3} |\widehat{X}_{\underline{s}}^{i,N}|^j\bigg)\Bigg) \nu(dz).
			\end{align*}
In the following, we choose $\lambda = \widetilde{\gamma}_1+\frac{\widetilde{\gamma}_2}{N}$. 
It follows from  \eqref{sigmacdelta} of \textbf{T5} and \textcolor{black}{the equality $\mathcal{W}_2^2(\mu_{\underline{s}}^{\widehat{\boldsymbol{X}}^{N}},\delta_0)= \frac{1}{N}\sum_{m=1}^{N}\vert \widehat{X}_{\underline{s}}^{m,N}\vert^2$} that 
	\begin{align*}
	&-(\widetilde{\gamma}_1+\frac{\widetilde{\gamma}_2}{N})  |\widehat{X}_{\underline{s}}^{i,N}|^2 + \langle \widehat{X}_{\underline{s}}^{i,N}, b(\widehat{X}_{\underline{s}}^{i,N},\mu_{\underline{s}}^{\widehat{\boldsymbol{X}}^N}) \rangle + \dfrac{p-1}{2} |\sigma_\Delta(\widehat{X}_{\underline{s}}^{i,N},\mu_{\underline{s}}^{\widehat{\boldsymbol{X}}^N})|^2  \notag   \\
			&\quad  +     \left|c_{\Delta}\left(\widehat{X}_{\underline{s}}^{i,N},\mu_{\underline{s}}^{\widehat{\boldsymbol{X}}^N}\right)\right|^{2} \int_{\bR_{0}^{d}} \Bigg(\dfrac{\vert z\vert^2}{2}+\widetilde{L_3}^{-2}  \left(\left(1+\vert z\vert (\widetilde{L_3}+\upsilon)\right)^{p-1}-\vert z\vert (\widetilde{L_3}+\upsilon)-1\right) \bigg(\vert z\vert \Big(\frac{\widetilde{L_3}}{2}+\upsilon\Big)+\upsilon\bigg)\Bigg)  \nu(dz) \Bigg) \notag\\
	\leq & -(\widetilde{\gamma}_1+\frac{\widetilde{\gamma}_2}{N})  |\widehat{X}_{\underline{s}}^{i,N}|^2 +\widetilde{\gamma}_1 |\widehat{X}_{\underline{s}}^{i,N}|^2 +\widetilde{\gamma}_2\mathcal{W}_2^2(\mu_{\underline{s}}^{\widehat{\boldsymbol{X}}^N},\delta_0)+\widetilde{\eta} =  \frac{\widetilde{\gamma}_2}{N}\sum_{m=1, m\neq i}^{N}\vert \widehat{X}_{\underline{s}}^{m,N}\vert^2+\widetilde{\eta}.
	\end{align*} 
Therefore, 	using the estimate \textcolor{black}{$	\mathcal{W}_2(\mu_{\underline{s}}^{\widehat{\boldsymbol{X}}^N},\delta_0) \leq \frac{1}{\sqrt{N}} \sum_{m=1}^N |\Hat{X}^{m, N}_{\underline{s}}|$}, we obtain that 

\begin{align}
			\bE [\mathcal{R}_s]  
			&\leq p \mathbb{E} \Bigg[ \frac{\widetilde{\gamma}_2}{N}|\widehat{X}_{\underline{s}}^{i,N}|^{p-2} \sum_{m=1, m\neq i}^{N}\vert \widehat{X}_{\underline{s}}^{m,N}\vert^2 +\int_{\bR_{0}^{d}} \sum_{k=2}^{p/2} \sum_{\ell=0}^{k}{p/2 \choose k}|z|^{2k-\ell}Q_{p}\left(p-2,2k-\ell,|\widehat{X}_{\underline{s}}^{i,N}|, 1+\widehat{U}_{\underline{s}}^{i,N}\right) \nu(dz)\Bigg] \notag\\
			&\quad+ C\sum_{j=0}^{p-2}\bE \left[ |\widehat{X}_{\underline{s}}^{i,N}|^j\right], \label{tag42}
			\end{align}			
			for some positive constant $C$.
		
	Thanks to Lemma \ref{menh de 3}, $\mathbb{E}[\mathcal{M}_t] =0$. Now, we are going to take the expectation for \eqref{e^pX^p} with $\lambda=\widetilde{\gamma}_1+\frac{\widetilde{\gamma}_2}{N}$, plug \eqref{tag42} into \eqref{e^pX^p}, and use the inductive assumption,
	Condition \textbf{A5} and the following fact thanks to the independence between $|\widehat{X}_{\underline{s}}^{i,N}|^{p-2}$ and $\sum_{m=1, m\neq i}^{N}\vert \widehat{X}_{\underline{s}}^{m,N}\vert^2$
	\begin{equation}\label{indepen}\begin{split}
	&\mathbb{E}\left[ |\widehat{X}_{\underline{s}}^{i,N}|^{p-2} \sum_{m=1, m\neq i}^{N}\vert \widehat{X}_{\underline{s}}^{m,N}\vert^2\right]=\sum_{m=1, m\neq i}^{N}\mathbb{E}\left[ |\widehat{X}_{\underline{s}}^{i,N}|^{p-2}\right]\mathbb{E}\left[\vert \widehat{X}_{\underline{s}}^{m,N}\vert^2\right],\\
	&\mathbb{E}\left[ |\widehat{X}_{\underline{s}}^{i,N}|^{p-\ell} \left(1+\widehat{U}_{\underline{s}}^{i,N}\right)^{\ell}\right]=\mathbb{E}\left[ |\widehat{X}_{\underline{s}}^{i,N}|^{p-\ell} \right]\mathbb{E}\left[  \left(1+\widehat{U}_{\underline{s}}^{i,N}\right)^{\ell}\right],
	\end{split}
	\end{equation}	
	for any  $\ell \in\{2,\ldots,p-2\}$, where recall that $\widehat{U}_{\underline{s}}^{i,N}=\frac{1}{\sqrt{N}}\sum_{m=1; m\neq i}^{N}\vert \widehat{X}_{\underline{s}}^{m,N}\vert$. As a consequence, we get that
	\begin{align*}
	\bE \left[e^{-p\big(\widetilde{\gamma}_1+\frac{\widetilde{\gamma}_2}{N}\big) t}\left|\widehat{X}_t^{i,N}\right|^p \right]&\leq |x_{0}|^{p}+C \int_{0}^{t} e^{ -p\big(\widetilde{\gamma}_1+\frac{\widetilde{\gamma}_2}{N}\big) s}\left(\sum_{\ell=2}^{p-2}\mathbb{E}\left[ |\widehat{X}_{\underline{s}}^{i,N}|^{p-\ell} \right]\mathbb{E}\left[  \left(1+\widehat{U}_{\underline{s}}^{i,N}\right)^{\ell}\right]+\sum_{j=0}^{p-2}\bE \left[ |\widehat{X}_{\underline{s}}^{i,N}|^j\right]\right)ds.
	\end{align*}
	Here, recall that $\widetilde{\gamma} =\widetilde{\gamma}_1 +\widetilde{\gamma}_2$.
	
	\noindent\underline{Case $\widetilde{\gamma}> 0$:}  
	\begin{align*}
	\bE \left[e^{-p\big(\widetilde{\gamma}_1+\frac{\widetilde{\gamma}_2}{N}\big) t}\left|\widehat{X}_t^{i,N}\right|^p \right] &\leq |x_{0}|^{p}+C \int_{0}^{t} e^{ -p\big(\widetilde{\gamma}_1+\frac{\widetilde{\gamma}_2}{N}\big) s}\left(\sum_{\ell=2}^{p-2}e^{ (p-\ell)\big(\widetilde{\gamma}_1+\widetilde{\gamma}_2\big)  s}e^{ \ell\big(\widetilde{\gamma}_1+\widetilde{\gamma}_2\big)  s}+\sum_{j=0}^{p-2}e^{ j\big(\widetilde{\gamma}_1+\widetilde{\gamma}_2\big)  s}\right)ds\\
	&\leq  |x_{0}|^{p}+C \int_{0}^{t} e^{ -p\big(\widetilde{\gamma}_1+\frac{\widetilde{\gamma}_2}{N}\big) s}e^{ p\big(\widetilde{\gamma}_1+\widetilde{\gamma}_2\big)  s}ds\\
	&=|x_{0}|^{p}+C\int_{0}^{t} e^{ -p\widetilde{\gamma}_2\big(\frac{1}{N}-1\big) s}ds\\
	&=|x_{0}|^{p}+\dfrac{C}{-p\widetilde{\gamma}_2\big(\frac{1}{N}-1\big)}\left(e^{ -p\widetilde{\gamma}_2\big(\frac{1}{N}-1\big) t}-1\right),
	\end{align*}
	which implies that
	\begin{align*}
	\bE \left[\left|\widehat{X}_t^{i,N}\right|^p \right]&\leq |x_{0}|^{p}e^{p\big(\widetilde{\gamma}_1+\frac{\widetilde{\gamma}_2}{N}\big) t}+\dfrac{C}{p\widetilde{\gamma}_2\big(1-\frac{1}{N}\big)}\left(e^{p\big(\widetilde{\gamma}_1+\widetilde{\gamma}_2\big) t}-e^{p\big(\widetilde{\gamma}_1+\frac{\widetilde{\gamma}_2}{N}\big) t}\right)\\
	&\leq Ce^{p\big(\widetilde{\gamma}_1+\widetilde{\gamma}_2\big) t}\\
	&= Ce^{p\widetilde{\gamma} t}.
	\end{align*}

	\noindent\underline{Case $\widetilde{\gamma}=0$:}	In this case, we have $\widetilde{\gamma}_1+\frac{\widetilde{\gamma}_2}{N}<\widetilde{\gamma}=0$. Using the integration by parts formula repeatedly, it can be checked that for any $\beta<0$ and $q\in \mathbb{N}^{\ast}$,
	\begin{align*}
	\int_0^te^{-\beta s} (1+s)^qds &\leq C(\beta,q)\left(e^{-\beta t} (1+t)^q+\int_0^te^{-\beta s} ds\right) \\
	&\leq C(\beta,q)\left(e^{-\beta t} (1+t)^q+e^{-\beta t} \right)\\
	&\leq C(\beta,q)e^{-\beta t} (1+t)^q,
	\end{align*}
	for some constant $C(\beta,q)>0$. This deduces that
	\begin{align*}
	\bE \left[e^{-p\big(\widetilde{\gamma}_1+\frac{\widetilde{\gamma}_2}{N}\big) t}\left|\widehat{X}_t^{i,N}\right|^p \right] &\leq |x_{0}|^{p}+C \int_{0}^{t} e^{ -p\big(\widetilde{\gamma}_1+\frac{\widetilde{\gamma}_2}{N}\big) s}\left(\sum_{\ell=2}^{p-2}\left(1+s\right)^{(p-\ell)/2}\left(1+s\right)^{\ell/2}+\sum_{j=0}^{p-2}\left(1+s\right)^{j/2}\right)ds\\
	&\leq |x_{0}|^{p}+C \int_{0}^{t} e^{ -p\big(\widetilde{\gamma}_1+\frac{\widetilde{\gamma}_2}{N}\big) s}\left(1+s\right)^{p/2}ds\\
	&\leq |x_{0}|^{p}+C e^{ -p\big(\widetilde{\gamma}_1+\frac{\widetilde{\gamma}_2}{N}\big) t}\left(1+t\right)^{p/2}.
	\end{align*}
	Therefore, we obtain that
	\begin{align*}
	\mathbb{E}\left[\left|\widehat{X}_t^{i,N}\right|^p \right]\leq C\left(1+t\right)^{p/2}.
	\end{align*}
	\noindent\underline{Case $\widetilde{\gamma}<0$:} In this case, we have $\widetilde{\gamma}_1+\frac{\widetilde{\gamma}_2}{N}<\widetilde{\gamma}<0$. Thus, we get
	\begin{align*}
	\bE \left[e^{-p\big(\widetilde{\gamma}_1+\frac{\widetilde{\gamma}_2}{N}\big) t}\left|\widehat{X}_t^{i,N}\right|^p \right]&\leq |x_{0}|^{p}+C \int_{0}^{t} e^{ -p\big(\widetilde{\gamma}_1+\frac{\widetilde{\gamma}_2}{N}\big) s}ds,
	\end{align*}
	which implies that
	\begin{align*}
	\bE \left[\left|\widehat{X}_t^{i,N}\right|^p \right]\leq C.
	\end{align*}
	Consequently, combining with \eqref{qh}, we conclude that \eqref{EXtmu} holds for $k=k_0+1$. Thus, the result follows.
\end{proof} 

\begin{Rem}\label{hq1} 
	When condition \textbf{A6} and  all conditions of Proposition \ref{dinh ly 2} are satisfied, we get the following estimate on the expectation of the number of timesteps $N_T$  
	\begin{equation}\label{ENT}
	\mathbb{E}\left[N_{T}-1\right] \leq \dfrac{C}{\Delta},
	\end{equation}
	for any $T>0$, where the positive constant $C$ does not depend on $\Delta$.
	
	Indeed, the same argument in the proof of  Lemma 2 in \cite{FG} yields that
	$$N_T = \sum_{k=1}^{N_T} 1 \leq \int_0^T \frac{1}{\Delta\textbf{h}(\widehat{\boldsymbol{X}}_{\underline{t}}^N,\mu_{\underline{t}}^{\widehat{\boldsymbol{X}}^N})}dt + 1.$$ 
	Then, using \eqref{chooseh}, Assumption \textbf{A6}, and Remark  \ref{sigmac} (i) and (iii), we get that for any $i\in\{1,\ldots,N\}$ and $p_0\geq 2(\ell+1)$,
$$	\frac{1}{h(\widehat{X}^{i,N}_{\underline{t}},\mu_{\underline{t}}^{\widehat{\boldsymbol{X}}^N})} \leq C\left(1 + |\widehat{X}^{i,N}_{\underline{t}}|^{p_0} + \mathcal{W}_2^{p_0}(\mu_{\underline{t}}^{\widehat{\boldsymbol{X}}^N}, \delta_0)\right)  \leq C\left(1 + |\widehat{X}^{i,N}_{\underline{t}}|^{p_0} +  \dfrac{1}{N}\sum_{m=1}^{N}\left\vert \widehat{X}_{\underline{t}}^{m,N}\right\vert^{p_0}  \right).$$ 
This, combined with
$\textbf{h}(\widehat{\boldsymbol{X}}^N_{\underline{t}},\mu_{\underline{t}}^{\widehat{\boldsymbol{X}}^N})=\min\{h(\widehat{X}^{1,N}_{\underline{t}},\mu_{\underline{t}}^{\widehat{\boldsymbol{X}}^N}),\ldots,h(\widehat{X}^{N,N}_{\underline{t}},\mu_{\underline{t}}^{\widehat{\boldsymbol{X}}^N})\}$, Lemma \ref{menh de 3} and Proposition \ref{dinh ly 2}, shows the estimate \eqref{ENT}.
\end{Rem}

Next,  the difference between $\widehat{X}_t^{i,N}$ and $\widehat{X}_{\underline{t}}^{i,N}$ has the following uniform bound in time.
\begin{Lem}\label{hq2}
	Let all conditions of Proposition \ref{dinh ly 2} be satisfied. Then  for any $p \in [2,p_0]$,  there exists a positive constant $C_p=C(p,L)$ such that
	\begin{equation*}
	\max_{i\in\{1,\ldots,N\}}\bE \left[  \left|\widehat{X}_t^{i,N} - \widehat{X}_{\underline{t}}^{i,N}\right|^p \big| \mathcal{F}_{\underline{t}}  \right] \le C_p\Delta,
	\end{equation*}
	for any $t\geq 0$ and
	\begin{equation*}
	\max_{i\in\{1,\ldots,N\}}\sup_{t \ge 0}\bE \left[  \left|\widehat{X}_t^{i,N} - \widehat{X}_{\underline{t}}^{i,N}\right|^p \right] \le C_p\Delta.
	\end{equation*}
\end{Lem}
\begin{proof} 
	From \eqref{EM3}, we have that for any $p \ge 2$,
	\begin{align}
	|\widehat{X}_t^{i,N} - \widehat{X}_{\underline{t}}^{i,N}|^p 
	&= \left| b(\widehat{X}_{\underline{t}}^{i,N},\mu_{\underline{t}}^{\widehat{\boldsymbol{X}}^N})(t-\underline{t})+\sm_{\Delta}(\widehat{X}_{\underline{t}}^{i,N},\mu_{\underline{t}}^{\widehat{\boldsymbol{X}}^N})(W_t^i-W_{\underline{t}}^i) +c_{\Delta}(\widehat{X}_{\underline{t}}^{i,N},\mu_{\underline{t}}^{\widehat{\boldsymbol{X}}^N})(Z_t^i-Z_{\underline{t}}^i) \right|^p \notag \\
	&\le 3^{p-1} \left[  \left| b(\widehat{X}_{\underline{t}}^{i,N},\mu_{\underline{t}}^{\widehat{\boldsymbol{X}}^N})\right|^p (t-\underline{t})^p + \left|\sm_{\Delta}(\widehat{X}_{\underline{t}}^{i,N},\mu_{\underline{t}}^{\widehat{\boldsymbol{X}}^N})\right|^p \left|W_t^i-W_{\underline{t}}^i \right|^p +\left|c_{\Delta}(\widehat{X}_{\underline{t}}^{i,N},\mu_{\underline{t}}^{\widehat{\boldsymbol{X}}^N})\right|^p \left|Z_t^i-Z_{\underline{t}}^i \right|^p \right]. \notag 
	\end{align} 
	This, combined with \eqref{con-moment}, concludes the desired result.
\end{proof} 

\subsection{Convergence} 

We are going to state the main result of the paper. To this aim, the following additional conditions will be needed.
\begin{enumerate}
	\item[\bf T6.] There exists a positive constant $L_4$ such that
	\begin{align*}
	&|\sm(x,\mu)-\sm_{\Delta}(x,\mu)| \le L_4 \Delta^{1/2}|\sm(x,\mu)|^2(1+\vert x\vert),\\
	&|c(x,\mu)-c_{\Delta}(x,\mu)| \le L_4 \Delta^{1/2} |c(x,\mu)|^2\left(1+\vert x\vert+ \vert b(x,\mu)\vert\right),
	\end{align*}
	for all $x \in \bR^d$, and $\mu\in \mathcal{P}_2(\mathbb{R}^d)$.	
\end{enumerate}

\begin{Rem}
	If we choose 
	\begin{align} \label{smDtcDt}
	&\sm_{\Delta}(x,\mu)= \dfrac{\sm(x,\mu)}{1+\Delta^{1/2}|\sm(x,\mu)|(1+|x|)}, \quad c_{\Delta}(x,\mu)= \dfrac{c(x,\mu)}{1+\Delta^{1/2}|c(x,\mu)|(1+|x|+\vert b(x,\mu)\vert)},
	\end{align}
	then Conditions \textbf{T3, T4} and \textbf{T6} are satisfied.
\end{Rem}

\begin{Thm}\label{dinh ly 6} 
	Assume  that Conditions \textbf{A1}, \textbf{A3--A5},  \textbf{T2--T6} hold, and $p_0\geq 4\ell+6$,   $N \geq \left( \frac{\max\{3\widetilde{L_3},1\}}{2\upsilon}\right)^2$. 
	Assume further that there exists a constant $\varepsilon>0$ such that \textbf{A2} holds for $\kappa_1=\kappa_2=1+\varepsilon$, $L_1\in\mathbb{R}$, $L_2\geq 0$. Then for any $T>0$, there exist positive constants  $C_T=C(x_0,L,L_1,L_2,L_4,\widetilde{\gamma}_1,\widetilde{\gamma}_2,\widetilde{\eta},\widetilde{L_3}, \varepsilon, T)$ and $C'_T=C'(x_0,L,L_1,L_2,L_4,\widetilde{\gamma}_1,\widetilde{\gamma}_2,\widetilde{\eta},\widetilde{L_3}, \varepsilon,T)$ such that
	\begin{equation} \label{EXmu-X1}
	\max_{i\in\{1,\ldots,N\}}\underset{t\in[0,T]}{\sup} \bE \left[  \left\vert X_t^{i,N}-\widehat{X}_t^{i,N}\right\vert^2 \right] \le  C_T \Delta,
	\end{equation}
and for any $p\in (0,2)$, 
\begin{equation} \label{EXmu-X0}
	\max_{i\in\{1,\ldots,N\}}\bE \left[ \underset{t\in[0,T]}{\sup} \left\vert  X_t^{i,N}-\widehat{X}_t^{i,N}\right\vert^p \right] \le  \left(\dfrac{4-p}{2-p}\right) (C'_T\Delta)^{p/2}.
\end{equation}	
	Moreover, if  $\widetilde{\gamma} = \widetilde{\gamma}_1 + \widetilde{\gamma}_2 < 0$, and $L_1+L_2<0$, then,  there exists a positive constant  $C''=C''(x_0,L,L_1,L_2,L_4,\widetilde{\gamma}_1,\widetilde{\gamma}_2,\widetilde{\eta},\widetilde{L_3}, \varepsilon)$, which does not depend on $T$,  such that
	\begin{equation} \label{EXmu-X2}
	\max_{i\in\{1,\ldots,N\}}\underset{t \ge 0}{\sup}\ \bE \left[  \left\vert X_t^{i,N}-\widehat{X}_t^{i,N}\right\vert^2 \right] \le  C'' \Delta.
	\end{equation} 
\end{Thm}
\begin{proof}	
Thanks to \eqref{equa2} and \eqref{EM2}, we have that for any $i\in\{1,\ldots,N\}$,
\begin{align*}
X_t^{i,N}-\widehat{X}_t^{i,N}&=\int_0^t \left(b(X_s^{i,N},\mu_{s}^{\boldsymbol{X}^N})-b\left(\widehat{X}_{\underline{s}}^{i,N},\mu_{\underline{s}}^{\widehat{\boldsymbol{X}}^{N}}\right)\right)ds+\int_0^t \left(\sigma(X_s^{i,N},\mu_{s}^{\boldsymbol{X}^N})-\sigma_{\Delta}\left(\widehat{X}_{\underline{s}}^{i,N},\mu_{\underline{s}}^{\widehat{\boldsymbol{X}}^{N}}\right)\right)dW_s^i \\
&\quad+ \int_0^t\int_{\mathbb{R}_0^d} \left(c(X_{s-}^{i,N},\mu_{s-}^{\boldsymbol{X}^N})-c_{\Delta}\left(\widehat{X}_{\underline{s}}^{i,N},\mu_{\underline{s}}^{\widehat{\boldsymbol{X}}^{N}}\right)\right)z\widetilde{N}^i(d s, d z).
\end{align*}
Thanks to It\^o's formula, we obtain that for any  $\lambda \in \br$,\footnote{From here to the end of the proof, for the reader's convenience, we will use 3 different colors to track the changes of the quantities of interested.}
\begin{align}
&e^{-\lambda t}\vert X_t^{i,N}-\widehat{X}_t^{i,N}\vert^2=\int_0^te^{-\lambda s}\bigg(-\lambda\vert X_s^{i,N}-\widehat{X}_s^{i,N}\vert^2+\textcolor{blue}{2\left\langle X_s^{i,N}-\widehat{X}_s^{i,N},b(X_s^{i,N},\mu_{s}^{\boldsymbol{X}^N})-b\left(\widehat{X}_{\underline{s}}^{i,N},\mu_{\underline{s}}^{\widehat{\boldsymbol{X}}^{N}}\right)\right\rangle} \notag\\
&+ \textcolor{magenta}{\left\vert \sigma(X_s^{i,N},\mu_{s}^{\boldsymbol{X}^N})-\sigma_{\Delta}\left(\widehat{X}_{\underline{s}}^{i,N},\mu_{\underline{s}}^{\widehat{\boldsymbol{X}}^{N}}\right)\right\vert^2 }\bigg)ds+2\int_0^te^{-\lambda s}\left\langle X_s^{i,N}-\widehat{X}_s^{i,N},\big(\sigma(X_s^{i,N},\mu_{s}^{\boldsymbol{X}^N})-\sigma_{\Delta}(\widehat{X}_{\underline{s}}^{i,N},\mu_{\underline{s}}^{\widehat{\boldsymbol{X}}^{N}})\big)dW_s^i\right\rangle \notag\\
&+\int_0^t\int_{\mathbb{R}_0^d}e^{-\lambda s}\left(\vert X_{s-}^{i,N}-\widehat{X}_{s-}^{i,N}+\big(c(X_{s-}^{i,N},\mu_{s-}^{\boldsymbol{X}^N})-c_{\Delta}(\widehat{X}_{\underline{s}-}^{i,N},\mu_{\underline{s}-}^{\widehat{\boldsymbol{X}}^{N}})\big)z\vert^2-\vert X_{s-}^{i,N}-\widehat{X}_{s-}^{i,N}\vert^2\right)\widetilde{N}^i(d s, d z) \notag\\
&+ \textcolor{red}{\int_0^t\int_{\mathbb{R}_0^d}e^{-\lambda s}\big\vert \big(c(X_{s}^{i,N},\mu_{s}^{\boldsymbol{X}^N})-c_{\Delta}(\widehat{X}_{\underline{s}}^{i,N},\mu_{\underline{s}}^{\widehat{\boldsymbol{X}}^{N}})\big)z\big\vert^2\nu(d z) ds.} \label{express1}
\end{align}
In the following, we will give upper bounds for each term on the right-hand side of \eqref{express1}. 
\textcolor{blue}{First, we decompose 
\begin{align}
2\left\langle X_s^{i,N}-\widehat{X}_s^{i,N},b(X_s^{i,N},\mu_{s}^{\boldsymbol{X}^N})-b\left(\widehat{X}_{\underline{s}}^{i,N},\mu_{\underline{s}}^{\widehat{\boldsymbol{X}}^{N}}\right)\right\rangle =2\left\langle X_s^{i,N}-\widehat{X}_s^{i,N},b(X_s^{i,N},\mu_{s}^{\boldsymbol{X}^N})-b\left(\widehat{X}_{s}^{i,N},\mu_{s}^{\widehat{\boldsymbol{X}}^{N}}\right)\right\rangle + S, \label{ct000}
\end{align} 
where 
$S= 2\left\langle X_s^{i,N}-\widehat{X}_s^{i,N},b\left(\widehat{X}_{s}^{i,N},\mu_{s}^{\widehat{\boldsymbol{X}}^{N}}\right)-b\left(\widehat{X}_{\underline{s}}^{i,N},\mu_{\underline{s}}^{\widehat{\boldsymbol{X}}^{N}}\right)\right\rangle$. 
By using Cauchy's inequality and Condition \textbf{A4}, we obtain that for any $\varepsilon_1>0$,
\begin{align}
S& \leq \varepsilon_1 \left\vert X_s^{i,N}-\widehat{X}_s^{i,N}\right\vert^2+\dfrac{1}{\varepsilon_1}\left\vert b\left(\widehat{X}_{s}^{i,N},\mu_{s}^{\widehat{\boldsymbol{X}}^{N}}\right)-b\left(\widehat{X}_{\underline{s}}^{i,N},\mu_{\underline{s}}^{\widehat{\boldsymbol{X}}^{N}}\right)\right\vert^2 \notag\\
&\leq 
\varepsilon_1 \left\vert X_s^{i,N}-\widehat{X}_s^{i,N}\right\vert^2 +\dfrac{6}{\varepsilon_1}L^2\left(1+\left\vert \widehat{X}_{s}^{i,N}\right\vert^{2\ell}+\left\vert \widehat{X}_{\underline{s}}^{i,N}\right\vert^{2\ell}\right) \left(\left\vert \widehat{X}_{s}^{i,N}-\widehat{X}_{\underline{s}}^{i,N}\right\vert^2+\mathcal{W}_2^2\left(\mu_{s}^{\widehat{\boldsymbol{X}}^{N}},\mu_{\underline{s}}^{\widehat{\boldsymbol{X}}^{N}}\right)\right) \notag\\
&\leq 
\varepsilon_1 \left\vert X_s^{i,N}-\widehat{X}_s^{i,N}\right\vert^2\notag\\
& \quad   +\dfrac{6}{\varepsilon_1}L^2\left(1+2^{2\ell-1}\left(\left\vert \widehat{X}_{s}^{i,N}-\widehat{X}_{\underline{s}}^{i,N}\right\vert^{2\ell}+\left\vert \widehat{X}_{\underline{s}}^{i,N}\right\vert^{2\ell}\right)+\left\vert \widehat{X}_{\underline{s}}^{i,N}\right\vert^{2\ell}\right) \left(\left\vert \widehat{X}_{s}^{i,N}-\widehat{X}_{\underline{s}}^{i,N}\right\vert^2+\mathcal{W}_2^2\left(\mu_{s}^{\widehat{\boldsymbol{X}}^{N}},\mu_{\underline{s}}^{\widehat{\boldsymbol{X}}^{N}}\right)\right) \notag\\
&=
\varepsilon_1 \left\vert X_s^{i,N}-\widehat{X}_s^{i,N}\right\vert^2 +\dfrac{6}{\varepsilon_1}L^2\Bigg(\left\vert \widehat{X}_{s}^{i,N}-\widehat{X}_{\underline{s}}^{i,N}\right\vert^2+\mathcal{W}_2^2\left(\mu_{s}^{\widehat{\boldsymbol{X}}^{N}},\mu_{\underline{s}}^{\widehat{\boldsymbol{X}}^{N}}\right)+2^{2\ell-1}\left\vert \widehat{X}_{s}^{i,N}-\widehat{X}_{\underline{s}}^{i,N}\right\vert^{2\ell+2} \Bigg) \notag \\
&\quad +\dfrac{6}{\varepsilon_1}L^2\Bigg( 2^{2\ell}\left\vert \widehat{X}_{s}^{i,N}-\widehat{X}_{\underline{s}}^{i,N}\right\vert^{2\ell}  \mathcal{W}_2^2\left(\mu_{s}^{\widehat{\boldsymbol{X}}^{N}},\mu_{\underline{s}}^{\widehat{\boldsymbol{X}}^{N}}\right)+(2^{2\ell-1}+1)\left\vert \widehat{X}_{\underline{s}}^{i,N}\right\vert^{2\ell}\left\vert \widehat{X}_{s}^{i,N}-\widehat{X}_{\underline{s}}^{i,N}\right\vert^2 \Bigg) \notag \\
& \quad +\dfrac{6}{\varepsilon_1}L^2 (2^{2\ell}+1)\left\vert \widehat{X}_{\underline{s}}^{i,N}\right\vert^{2\ell}\mathcal{W}_2^2\left(\mu_{s}^{\widehat{\boldsymbol{X}}^{N}},\mu_{\underline{s}}^{\widehat{\boldsymbol{X}}^{N}}\right).\label{express2}
\end{align}
}
 \textcolor{magenta}{
Second, by using Cauchy's inequality
we have that for any $\varepsilon_2>0$,
\begin{align}
&\left\vert \sigma(X_s^{i,N},\mu_{s}^{\boldsymbol{X}^N})-\sigma_{\Delta}\left(\widehat{X}_{\underline{s}}^{i,N},\mu_{\underline{s}}^{\widehat{\boldsymbol{X}}^{N}}\right)\right\vert^2 \notag\\
&\leq (1+\varepsilon_2)\left\vert \sigma(X_s^{i,N},\mu_{s}^{\boldsymbol{X}^N})-\sigma\left(\widehat{X}_{s}^{i,N},\mu_{s}^{\widehat{\boldsymbol{X}}^{N}}\right)\right\vert^2  \notag\\
&\quad+2\left(1+\frac{1}{\varepsilon_2}\right)\left(\left\vert \sigma\left(\widehat{X}_{s}^{i,N},\mu_{s}^{\widehat{\boldsymbol{X}}^{N}}\right)-\sigma\left(\widehat{X}_{\underline{s}}^{i,N},\mu_{\underline{s}}^{\widehat{\boldsymbol{X}}^{N}}\right)\right\vert^2+\left\vert \sigma\left(\widehat{X}_{\underline{s}}^{i,N},\mu_{\underline{s}}^{\widehat{\boldsymbol{X}}^{N}}\right)-\sigma_{\Delta}\left(\widehat{X}_{\underline{s}}^{i,N},\mu_{\underline{s}}^{\widehat{\boldsymbol{X}}^{N}}\right)\right\vert^2\right). \label{ct001}
\end{align} 
It follows from Remark \ref{sigmac} that 
\begin{align}
&\left\vert \sigma\left(\widehat{X}_{s}^{i,N},\mu_{s}^{\widehat{\boldsymbol{X}}^{N}}\right)-\sigma\left(\widehat{X}_{\underline{s}}^{i,N},\mu_{\underline{s}}^{\widehat{\boldsymbol{X}}^{N}}\right)\right\vert^2  \notag \\
\leq & \dfrac{L\widetilde{L}}{1+\varepsilon}\left(1+\left\vert \widehat{X}_{s}^{i,N}\right\vert^{\ell}+\left\vert \widehat{X}_{\underline{s}}^{i,N}\right\vert^{\ell}\right)   \left(\left\vert \widehat{X}_{s}^{i,N}-\widehat{X}_{\underline{s}}^{i,N}\right\vert^2+\mathcal{W}_2^2\left(\mu_{s}^{\widehat{\boldsymbol{X}}^{N}},\mu_{\underline{s}}^{\widehat{\boldsymbol{X}}^{N}}\right)\right) \notag \\
\leq & \dfrac{L\widetilde{L}}{1+\varepsilon}\left(1+2^{\ell}\Big(\left\vert \widehat{X}_{s}^{i,N}-\widehat{X}_{\underline{s}}^{i,N}\right\vert^{\ell}+\left\vert \widehat{X}_{\underline{s}}^{i,N}\right\vert^{\ell}\Big)+\left\vert \widehat{X}_{\underline{s}}^{i,N}\right\vert^{\ell}\right) \left(\left\vert \widehat{X}_{s}^{i,N}-\widehat{X}_{\underline{s}}^{i,N}\right\vert^2+\mathcal{W}_2^2\left(\mu_{s}^{\widehat{\boldsymbol{X}}^{N}},\mu_{\underline{s}}^{\widehat{\boldsymbol{X}}^{N}}\right)\right)  \notag\\
= & \dfrac{L\widetilde{L}}{1+\varepsilon}\bigg(\left\vert \widehat{X}_{s}^{i,N}-\widehat{X}_{\underline{s}}^{i,N}\right\vert^2+\mathcal{W}_2^2\left(\mu_{s}^{\widehat{\boldsymbol{X}}^{N}},\mu_{\underline{s}}^{\widehat{\boldsymbol{X}}^{N}}\right) +2^{\ell}\left\vert \widehat{X}_{s}^{i,N}-\widehat{X}_{\underline{s}}^{i,N}\right\vert^{\ell+2}\notag \\
&+2^{\ell}\left\vert \widehat{X}_{s}^{i,N}-\widehat{X}_{\underline{s}}^{i,N}\right\vert^{\ell}  \mathcal{W}_2^2\left(\mu_{s}^{\widehat{\boldsymbol{X}}^{N}},\mu_{\underline{s}}^{\widehat{\boldsymbol{X}}^{N}}\right)+(2^{\ell}+1)\left\vert \widehat{X}_{\underline{s}}^{i,N}\right\vert^{\ell}\left\vert \widehat{X}_{s}^{i,N}-\widehat{X}_{\underline{s}}^{i,N}\right\vert^2  \notag \\
&+(2^{\ell-1}+1)\big\vert \widehat{X}_{\underline{s}}^{i,N}\big\vert^{\ell}\mathcal{W}_2^2\left(\mu_{s}^{\widehat{\boldsymbol{X}}^{N}},\mu_{\underline{s}}^{\widehat{\boldsymbol{X}}^{N}}\right)\bigg).\label{ct002} 
\end{align} 
Then, it follows from \textbf{T6} and Remark \ref{sigmac}(iii) that 
\begin{align}
&\left\vert \sigma\left(\widehat{X}_{\underline{s}}^{i,N},\mu_{\underline{s}}^{\widehat{\boldsymbol{X}}^{N}}\right)-\sigma_{\Delta}\left(\widehat{X}_{\underline{s}}^{i,N},\mu_{\underline{s}}^{\widehat{\boldsymbol{X}}^{N}}\right)\right\vert^2 \leq  L_4^2 \Delta\left\vert\sm\left(\widehat{X}_{\underline{s}}^{i,N},\mu_{\underline{s}}^{\widehat{\boldsymbol{X}}^{N}}\right)\right\vert^4\left(1+\left\vert \widehat{X}_{\underline{s}}^{i,N}\right\vert\right)^2  \notag \\
\leq &4L_4^2 \Delta\bigg(16\Big(\dfrac{L\widetilde{L}}{1+\varepsilon}\Big)^2\Big(1+\vert \widehat{X}_{\underline{s}}^{i,N}\vert^{2\ell}\Big)\Big(\vert \widehat{X}_{\underline{s}}^{i,N}\vert^4+\mathcal{W}_2^4\big(\mu_{\underline{s}}^{\widehat{\boldsymbol{X}}^{N}},\delta_0\big)\Big)+4\left\vert \sigma(0,\delta_0)\right\vert^4\bigg)\Big(1+\vert \widehat{X}_{\underline{s}}^{i,N}\vert^2\Big). \label{ct003} 
\end{align}
}
\textcolor{red}{ 
Third, using Cauchy's inequality,   we obtain that for any $\varepsilon_3>0$,
\begin{align}
&\left\vert c(X_s^{i,N},\mu_{s}^{\boldsymbol{X}^N})-c_{\Delta}\left(\widehat{X}_{\underline{s}}^{i,N},\mu_{\underline{s}}^{\widehat{\boldsymbol{X}}^{N}}\right)\right\vert^2
\leq (1+\varepsilon_3)\left\vert c(X_s^{i,N},\mu_{s}^{\boldsymbol{X}^N})-c\left(\widehat{X}_{s}^{i,N},\mu_{s}^{\widehat{\boldsymbol{X}}^{N}}\right)\right\vert^2 
\notag\\
&\quad+2\left(1+\frac{1}{\varepsilon_3}\right)\bigg(\left\vert c\left(\widehat{X}_{s}^{i,N},\mu_{s}^{\widehat{\boldsymbol{X}}^{N}}\right)-c\left(\widehat{X}_{\underline{s}}^{i,N},\mu_{\underline{s}}^{\widehat{\boldsymbol{X}}^{N}}\right)\right\vert^2+\left\vert c\left(\widehat{X}_{\underline{s}}^{i,N},\mu_{\underline{s}}^{\widehat{\boldsymbol{X}}^{N}}\right)-c_{\Delta}\left(\widehat{X}_{\underline{s}}^{i,N},\mu_{\underline{s}}^{\widehat{\boldsymbol{X}}^{N}}\right)\right\vert^2\bigg).
\label{ct004} 
\end{align}
Thanks to Remark \ref{sigmac}(ii), we have 
\begin{align}
&\left\vert c\left(\widehat{X}_{s}^{i,N},\mu_{s}^{\widehat{\boldsymbol{X}}^{N}}\right)-c\left(\widehat{X}_{\underline{s}}^{i,N},\mu_{\underline{s}}^{\widehat{\boldsymbol{X}}^{N}}\right)\right\vert^2 \int_{\bR_{0}^d}|z|^{2} \nu(d z)   \notag\\
\leq &  \dfrac{L\widetilde{L}}{1+\varepsilon}\left(1+\left\vert \widehat{X}_{s}^{i,N}\right\vert^{\ell}+\left\vert \widehat{X}_{\underline{s}}^{i,N}\right\vert^{\ell}\right) \left(\left\vert \widehat{X}_{s}^{i,N}-\widehat{X}_{\underline{s}}^{i,N}\right\vert^2+\mathcal{W}_2^2\left(\mu_{s}^{\widehat{\boldsymbol{X}}^{N}},\mu_{\underline{s}}^{\widehat{\boldsymbol{X}}^{N}}\right)\right)  \notag\\
\leq & \dfrac{L\widetilde{L}}{1+\varepsilon}\left(1+2^{\ell}\Big(\left\vert \widehat{X}_{s}^{i,N}-\widehat{X}_{\underline{s}}^{i,N}\right\vert^{\ell}+\left\vert \widehat{X}_{\underline{s}}^{i,N}\right\vert^{\ell}\Big)+\left\vert \widehat{X}_{\underline{s}}^{i,N}\right\vert^{\ell}\right) \left(\left\vert \widehat{X}_{s}^{i,N}-\widehat{X}_{\underline{s}}^{i,N}\right\vert^2+\mathcal{W}_2^2\left(\mu_{s}^{\widehat{\boldsymbol{X}}^{N}},\mu_{\underline{s}}^{\widehat{\boldsymbol{X}}^{N}}\right)\right) \notag\\
= &\dfrac{L\widetilde{L}}{1+\varepsilon}\bigg(\left\vert \widehat{X}_{s}^{i,N}-\widehat{X}_{\underline{s}}^{i,N}\right\vert^2+\mathcal{W}_2^2\left(\mu_{s}^{\widehat{\boldsymbol{X}}^{N}},\mu_{\underline{s}}^{\widehat{\boldsymbol{X}}^{N}}\right) +2^{\ell}\left\vert \widehat{X}_{s}^{i,N}-\widehat{X}_{\underline{s}}^{i,N}\right\vert^{\ell+2}  \notag \\
&\qquad \qquad + 2^{\ell}\left\vert \widehat{X}_{s}^{i,N}-\widehat{X}_{\underline{s}}^{i,N}\right\vert^{\ell} 
\mathcal{W}_2^2\left(\mu_{s}^{\widehat{\boldsymbol{X}}^{N}},\mu_{\underline{s}}^{\widehat{\boldsymbol{X}}^{N}}\right)+(2^{\ell}+1)\left\vert \widehat{X}_{\underline{s}}^{i,N}\right\vert^{\ell}\left\vert \widehat{X}_{s}^{i,N}-\widehat{X}_{\underline{s}}^{i,N}\right\vert^2 \notag \\
&\qquad \qquad +(2^{\ell}+1)\big\vert \widehat{X}_{\underline{s}}^{i,N}\big\vert^{\ell}\mathcal{W}_2^2\left(\mu_{s}^{\widehat{\boldsymbol{X}}^{N}},\mu_{\underline{s}}^{\widehat{\boldsymbol{X}}^{N}}\right)\bigg).\label{ct005}
\end{align} 
Thanks to Condition \textbf{T6} and Remark \ref{sigmac}(iii) and (i), we have 
\begin{align}
&\left\vert c\left(\widehat{X}_{\underline{s}}^{i,N},\mu_{\underline{s}}^{\widehat{\boldsymbol{X}}^{N}}\right)-c_{\Delta}\left(\widehat{X}_{\underline{s}}^{i,N},\mu_{\underline{s}}^{\widehat{\boldsymbol{X}}^{N}}\right)\right\vert^2 \notag\\
\leq & L_4^2 \Delta\left\vert c\left(\widehat{X}_{\underline{s}}^{i,N},\mu_{\underline{s}}^{\widehat{\boldsymbol{X}}^{N}}\right)\right\vert^4\left(1+\left\vert \widehat{X}_{\underline{s}}^{i,N}\right\vert+\left\vert b\left(\widehat{X}_{\underline{s}}^{i,N},\mu_{\underline{s}}^{\widehat{\boldsymbol{X}}^{N}}\right)\right\vert\right)^2 
\notag\\
\leq & \dfrac{6L_4^2 \Delta}{ \Big(\int_{\bR_{0}^d}|z|^{2} \nu(d z) \Big)^2}\bigg(16\Big(\dfrac{L\widetilde{L}}{1+\varepsilon}\Big)^2\Big(1+\vert \widehat{X}_{\underline{s}}^{i,N}\vert^{2\ell}\Big)\Big(\vert \widehat{X}_{\underline{s}}^{i,N}\vert^4+\mathcal{W}_2^4\big(\mu_{\underline{s}}^{\widehat{\boldsymbol{X}}^{N}},\delta_0\big)\Big)+4\left\vert c(0,\delta_0)\right\vert^4\Big(\int_{\bR_{0}^d}|z|^{2} \nu(d z)\Big)^2\bigg) \notag \\ 
&\quad\times\left(1+\vert \widehat{X}_{\underline{s}}^{i,N}\vert^2+2\left(\vert b(0,\delta_0)\vert^2+4L^2\big(1+\vert \widehat{X}_{\underline{s}}^{i,N}\vert^{2\ell}\big)\left(\vert \widehat{X}_{\underline{s}}^{i,N}\vert^{2}+\mathcal{W}_2^2\big(\mu_{\underline{s}}^{\widehat{\boldsymbol{X}}^{N}},\delta_0\big)\right)\right)\right). \label{express4}
\end{align}
}
Therefore, inserting all estimations   \eqref{ct000} -- \eqref{express4} into \eqref{express1}, and choosing $\varepsilon_2=\varepsilon_3=\varepsilon$, we obtain that for any $\lambda \in \br$ and $\varepsilon_1>0$,
\begin{align}
&e^{-\lambda t}\vert X_t^{i,N}-\widehat{X}_t^{i,N}\vert^2 \notag\\
&\leq \int_0^te^{-\lambda s}\Bigg(-\lambda\vert X_s^{i,N}-\widehat{X}_s^{i,N}\vert^2 
+\textcolor{blue}{\varepsilon_1 \left\vert X_s^{i,N}-\widehat{X}_s^{i,N}\right\vert^2
+2\left\langle X_s^{i,N}-\widehat{X}_s^{i,N},b(X_s^{i,N},\mu_{s}^{\boldsymbol{X}^N})-b\left(\widehat{X}_{s}^{i,N},\mu_{s}^{\widehat{\boldsymbol{X}}^{N}}\right)\right\rangle} \notag\\
&\quad+ \textcolor{magenta}{(1+\varepsilon)\left\vert \sigma(X_s^{i,N},\mu_{s}^{\boldsymbol{X}^N})-\sigma\left(\widehat{X}_{s}^{i,N},\mu_{s}^{\widehat{\boldsymbol{X}}^{N}}\right)\right\vert^2}
+\textcolor{red}{(1+\varepsilon)\left\vert c(X_s^{i,N},\mu_{s}^{\boldsymbol{X}^N})-c\left(\widehat{X}_{s}^{i,N},\mu_{s}^{\widehat{\boldsymbol{X}}^{N}}\right)\right\vert^2\int_{\bR_{0}^d}|z|^{2} \nu(d z) } \notag\\
&\quad + \textcolor{blue} {\dfrac{6}{\varepsilon_1}L^2\Bigg[ \left\vert \widehat{X}_{s}^{i,N}-\widehat{X}_{\underline{s}}^{i,N}\right\vert^2 +\mathcal{W}_2^2\left(\mu_{s}^{\widehat{\boldsymbol{X}}^{N}},\mu_{\underline{s}}^{\widehat{\boldsymbol{X}}^{N}}\right)+2^{2\ell}\left\vert \widehat{X}_{s}^{i,N}-\widehat{X}_{\underline{s}}^{i,N}\right\vert^{2\ell+2}  +2^{2\ell}\left\vert \widehat{X}_{s}^{i,N}-\widehat{X}_{\underline{s}}^{i,N}\right\vert^{2\ell} }\notag\\
&\quad \textcolor{blue} { \times \mathcal{W}_2^2\left(\mu_{s}^{\widehat{\boldsymbol{X}}^{N}},\mu_{\underline{s}}^{\widehat{\boldsymbol{X}}^{N}}\right)+(2^{2\ell}+1)\left\vert \widehat{X}_{\underline{s}}^{i,N}\right\vert^{2\ell}\left\vert \widehat{X}_{s}^{i,N}-\widehat{X}_{\underline{s}}^{i,N}\right\vert^2  +(2^{2\ell}+1)\left\vert \widehat{X}_{\underline{s}}^{i,N}\right\vert^{2\ell}\mathcal{W}_2^2\left(\mu_{s}^{\widehat{\boldsymbol{X}}^{N}},\mu_{\underline{s}}^{\widehat{\boldsymbol{X}}^{N}}\right)\Bigg]}  \notag\\
&\quad+ \textcolor{magenta}{ 2\big(1+\frac{1}{\varepsilon}\big)\Bigg[ \dfrac{L\widetilde{L}}{1+\varepsilon}\bigg(\left\vert \widehat{X}_{s}^{i,N}-\widehat{X}_{\underline{s}}^{i,N}\right\vert^2+\mathcal{W}_2^2\left(\mu_{s}^{\widehat{\boldsymbol{X}}^{N}},\mu_{\underline{s}}^{\widehat{\boldsymbol{X}}^{N}}\right)  +2^{\ell}\left\vert \widehat{X}_{s}^{i,N}-\widehat{X}_{\underline{s}}^{i,N}\right\vert^{\ell+2}+2^{\ell}\left\vert \widehat{X}_{s}^{i,N}-\widehat{X}_{\underline{s}}^{i,N}\right\vert^{\ell}}\notag\\
&\quad \textcolor{magenta}{  \times\mathcal{W}_2^2\left(\mu_{s}^{\widehat{\boldsymbol{X}}^{N}},\mu_{\underline{s}}^{\widehat{\boldsymbol{X}}^{N}}\right)+(2^{\ell}+1)\left\vert \widehat{X}_{\underline{s}}^{i,N}\right\vert^{\ell}\left\vert \widehat{X}_{s}^{i,N}-\widehat{X}_{\underline{s}}^{i,N}\right\vert^2  +(2^{\ell}+1)\big\vert \widehat{X}_{\underline{s}}^{i,N}\big\vert^{\ell}\mathcal{W}_2^2\left(\mu_{s}^{\widehat{\boldsymbol{X}}^{N}},\mu_{\underline{s}}^{\widehat{\boldsymbol{X}}^{N}}\right)\bigg) }\notag\\
&\quad +\textcolor{magenta}{  4L_4^2 \Delta\bigg( 16\Big(\dfrac{L\widetilde{L}}{1+\varepsilon}\Big)^2\Big(1+\vert \widehat{X}_{\underline{s}}^{i,N}\vert^{2\ell}\Big)\Big(\vert \widehat{X}_{\underline{s}}^{i,N}\vert^4+\mathcal{W}_2^4\big(\mu_{\underline{s}}^{\widehat{\boldsymbol{X}}^{N}},\delta_0\big)\Big)  +4\left\vert \sigma(0,\delta_0)\right\vert^4\bigg) \Big(1+\vert \widehat{X}_{\underline{s}}^{i,N}\vert^2\Big)\Bigg] }  \notag\\
&\quad + \textcolor{red}{ 2\left(1+\frac{1}{\varepsilon}\right)\Bigg[ \dfrac{L\widetilde{L}}{1+\varepsilon}\bigg(\left\vert \widehat{X}_{s}^{i,N}-\widehat{X}_{\underline{s}}^{i,N}\right\vert^2+\mathcal{W}_2^2\left(\mu_{s}^{\widehat{\boldsymbol{X}}^{N}},\mu_{\underline{s}}^{\widehat{\boldsymbol{X}}^{N}}\right) +2^{\ell}\left\vert \widehat{X}_{s}^{i,N}-\widehat{X}_{\underline{s}}^{i,N}\right\vert^{\ell+2}+2^{\ell}\left\vert \widehat{X}_{s}^{i,N}-\widehat{X}_{\underline{s}}^{i,N}\right\vert^{\ell} }\notag\\
&\quad \textcolor{red}{ \times\mathcal{W}_2^2\left(\mu_{s}^{\widehat{\boldsymbol{X}}^{N}},\mu_{\underline{s}}^{\widehat{\boldsymbol{X}}^{N}}\right)+(2^{\ell}+1)\left\vert \widehat{X}_{\underline{s}}^{i,N}\right\vert^{\ell}\left\vert \widehat{X}_{s}^{i,N}-\widehat{X}_{\underline{s}}^{i,N}\right\vert^2  +(2^{\ell}+1)\big\vert \widehat{X}_{\underline{s}}^{i,N}\big\vert^{\ell}\mathcal{W}_2^2\left(\mu_{s}^{\widehat{\boldsymbol{X}}^{N}},\mu_{\underline{s}}^{\widehat{\boldsymbol{X}}^{N}}\right)\bigg)}\notag\\
&\quad \textcolor{red}{+\dfrac{6L_4^2 \Delta}{\int_{\bR_{0}^d}|z|^{2} \nu(d z)}\bigg( 16\Big(\dfrac{L\widetilde{L}}{1+\varepsilon}\Big)^2\Big(1+\vert \widehat{X}_{\underline{s}}^{i,N}\vert^{2\ell}\Big)\Big(\vert \widehat{X}_{\underline{s}}^{i,N}\vert^4+\mathcal{W}_2^4\big(\mu_{\underline{s}}^{\widehat{\boldsymbol{X}}^{N}},\delta_0\big)\Big)  +4\left\vert c(0,\delta_0)\right\vert^4\Big(\int_{\bR_{0}^d}|z|^{2} \nu(d z)\Big)^2\bigg)}  \notag\\
&\quad \textcolor{red}{\times\left(1+\vert \widehat{X}_{\underline{s}}^{i,N}\vert^2+2\left(\vert b(0,\delta_0)\vert^2+4L^2\left(1+\vert \widehat{X}_{\underline{s}}^{i,N}\vert^{2\ell}\right)\left(\vert \widehat{X}_{\underline{s}}^{i,N}\vert^{2}+\mathcal{W}_2^2\big(\mu_{\underline{s}}^{\widehat{\boldsymbol{X}}^{N}},\delta_0\big)\right)\right)\right)\Bigg]}\Bigg) ds   \notag\\
&\quad+2\int_0^te^{-\lambda s}\left\langle X_s^{i,N}-\widehat{X}_s^{i,N},\big(\sigma(X_s^{i,N},\mu_{s}^{\boldsymbol{X}^N})-\sigma_{\Delta}(\widehat{X}_{\underline{s}}^{i,N},\mu_{\underline{s}}^{\widehat{\boldsymbol{X}}^{N}})\big)dW_s^i\right\rangle  \notag\\
&\quad+\int_0^t\int_{\mathbb{R}_0^d}e^{-\lambda s}\left(\vert X_{s-}^{i,N}-\widehat{X}_{s-}^{i,N}+\big(c(X_{s-}^{i,N},\mu_{s-}^{\boldsymbol{X}^N})-c_{\Delta}(\widehat{X}_{\underline{s}-}^{i,N},\mu_{\underline{s}-}^{\widehat{\boldsymbol{X}}^{N}})\big)z\vert^2-\vert X_{s-}^{i,N}-\widehat{X}_{s-}^{i,N}\vert^2\right)\widetilde{N}^i(d s, d z). \notag 
\end{align}

Using Condition \textbf{A2} for $\kappa_1=\kappa_2=1+\varepsilon$, $L_1\in\mathbb{R}$, $L_2\geq 0$, we obtain that for any $\lambda \in \br$ and $\varepsilon_1>0$,
\begin{align}
&e^{-\lambda t}\vert X_t^{i,N}-\widehat{X}_t^{i,N}\vert^2 \notag\\
&\leq \int_0^te^{-\lambda s}\Bigg\{ -\lambda\vert X_s^{i,N}-\widehat{X}_s^{i,N}\vert^2 
+ \textcolor{blue}{ \varepsilon_1 \left\vert X_s^{i,N}-\widehat{X}_s^{i,N}\right\vert^2}
+L_1 \vert X_s^{i,N}-\widehat{X}_s^{i,N}\vert^2+L_2\mathcal{W}_2^2\left(\mu_{s}^{\boldsymbol{X}^{N}},\mu_{s}^{\widehat{\boldsymbol{X}}^{N}}\right)  \notag\\
&\quad +\textcolor{blue}{\dfrac{6}{\varepsilon_1}L^2\Bigg[\left\vert \widehat{X}_{s}^{i,N}-\widehat{X}_{\underline{s}}^{i,N}\right\vert^2 +\mathcal{W}_2^2\left(\mu_{s}^{\widehat{\boldsymbol{X}}^{N}},\mu_{\underline{s}}^{\widehat{\boldsymbol{X}}^{N}}\right)+2^{2\ell}\left\vert \widehat{X}_{s}^{i,N}-\widehat{X}_{\underline{s}}^{i,N}\right\vert^{2\ell+2} +2^{2\ell}\left\vert \widehat{X}_{s}^{i,N}-\widehat{X}_{\underline{s}}^{i,N}\right\vert^{2\ell}} \notag\\
&\quad \textcolor{blue}{\times\mathcal{W}_2^2\left(\mu_{s}^{\widehat{\boldsymbol{X}}^{N}},\mu_{\underline{s}}^{\widehat{\boldsymbol{X}}^{N}}\right)+(2^{2\ell}+1)\left\vert \widehat{X}_{\underline{s}}^{i,N}\right\vert^{2\ell}\left\vert \widehat{X}_{s}^{i,N}-\widehat{X}_{\underline{s}}^{i,N}\right\vert^2  +(2^{2\ell}+1)\left\vert \widehat{X}_{\underline{s}}^{i,N}\right\vert^{2\ell}\mathcal{W}_2^2\left(\mu_{s}^{\widehat{\boldsymbol{X}}^{N}},\mu_{\underline{s}}^{\widehat{\boldsymbol{X}}^{N}}\right)\Bigg]}\notag\\
&\quad \textcolor{magenta}{+2\big(1+\frac{1}{\varepsilon}\big)\Bigg[ \dfrac{L\widetilde{L}}{1+\varepsilon}\bigg(\left\vert \widehat{X}_{s}^{i,N}-\widehat{X}_{\underline{s}}^{i,N}\right\vert^2+\mathcal{W}_2^2\left(\mu_{s}^{\widehat{\boldsymbol{X}}^{N}},\mu_{\underline{s}}^{\widehat{\boldsymbol{X}}^{N}}\right)  +2^{\ell}\left\vert \widehat{X}_{s}^{i,N}-\widehat{X}_{\underline{s}}^{i,N}\right\vert^{\ell+2}+2^{\ell}\left\vert \widehat{X}_{s}^{i,N}-\widehat{X}_{\underline{s}}^{i,N}\right\vert^{\ell}}\notag\\
&\quad \textcolor{magenta}{ \times\mathcal{W}_2^2\left(\mu_{s}^{\widehat{\boldsymbol{X}}^{N}},\mu_{\underline{s}}^{\widehat{\boldsymbol{X}}^{N}}\right)+(2^{\ell}+1)\left\vert \widehat{X}_{\underline{s}}^{i,N}\right\vert^{\ell}\left\vert \widehat{X}_{s}^{i,N}-\widehat{X}_{\underline{s}}^{i,N}\right\vert^2  +(2^{\ell}+1)\big\vert \widehat{X}_{\underline{s}}^{i,N}\big\vert^{\ell}\mathcal{W}_2^2\left(\mu_{s}^{\widehat{\boldsymbol{X}}^{N}},\mu_{\underline{s}}^{\widehat{\boldsymbol{X}}^{N}}\right)\bigg)}\notag\\
&\quad \textcolor{magenta}{+4L_4^2 \Delta\bigg( 16\Big(\dfrac{L\widetilde{L}}{1+\varepsilon}\Big)^2\Big(1+\vert \widehat{X}_{\underline{s}}^{i,N}\vert^{2\ell}\Big)\Big(\vert \widehat{X}_{\underline{s}}^{i,N}\vert^4+\mathcal{W}_2^4\big(\mu_{\underline{s}}^{\widehat{\boldsymbol{X}}^{N}},\delta_0\big)\Big) +4\left\vert \sigma(0,\delta_0)\right\vert^4\bigg) \Big(1+\vert \widehat{X}_{\underline{s}}^{i,N}\vert^2\Big)\Bigg]} \notag\\
&\quad \textcolor{red}{+2\left(1+\frac{1}{\varepsilon}\right)\Bigg[\dfrac{L\widetilde{L}}{1+\varepsilon}\bigg(\left\vert \widehat{X}_{s}^{i,N}-\widehat{X}_{\underline{s}}^{i,N}\right\vert^2+\mathcal{W}_2^2\left(\mu_{s}^{\widehat{\boldsymbol{X}}^{N}},\mu_{\underline{s}}^{\widehat{\boldsymbol{X}}^{N}}\right)  +2^{\ell}\left\vert \widehat{X}_{s}^{i,N}-\widehat{X}_{\underline{s}}^{i,N}\right\vert^{\ell+2}+2^{\ell}\left\vert \widehat{X}_{s}^{i,N}-\widehat{X}_{\underline{s}}^{i,N}\right\vert^{\ell}}\notag\\
&\quad \textcolor{red}{ \times\mathcal{W}_2^2\left(\mu_{s}^{\widehat{\boldsymbol{X}}^{N}},\mu_{\underline{s}}^{\widehat{\boldsymbol{X}}^{N}}\right)+(2^{\ell}+1)\left\vert \widehat{X}_{\underline{s}}^{i,N}\right\vert^{\ell}\left\vert \widehat{X}_{s}^{i,N}-\widehat{X}_{\underline{s}}^{i,N}\right\vert^2  +(2^{\ell}+1)\big\vert \widehat{X}_{\underline{s}}^{i,N}\big\vert^{\ell}\mathcal{W}_2^2\left(\mu_{s}^{\widehat{\boldsymbol{X}}^{N}},\mu_{\underline{s}}^{\widehat{\boldsymbol{X}}^{N}}\right)\bigg)}\notag\\
&\quad \textcolor{red}{ +\dfrac{6L_4^2 \Delta}{\int_{\bR_{0}^d}|z|^{2} \nu(d z)}\bigg( 16\Big(\dfrac{L\widetilde{L}}{1+\varepsilon}\Big)^2\Big(1+\vert \widehat{X}_{\underline{s}}^{i,N}\vert^{2\ell}\Big)\Big(\vert \widehat{X}_{\underline{s}}^{i,N}\vert^4+\mathcal{W}_2^4\big(\mu_{\underline{s}}^{\widehat{\boldsymbol{X}}^{N}},\delta_0\big)\Big)  +4\left\vert c(0,\delta_0)\right\vert^4\Big(\int_{\bR_{0}^d}|z|^{2} \nu(d z)\Big)^2\bigg)} \notag\\
&\quad \textcolor{red}{ \times \left(1+\vert \widehat{X}_{\underline{s}}^{i,N}\vert^2+2\left(\vert b(0,\delta_0)\vert^2+4L^2\left(1+\vert \widehat{X}_{\underline{s}}^{i,N}\vert^{2\ell}\right)\left(\vert \widehat{X}_{\underline{s}}^{i,N}\vert^{2}+\mathcal{W}_2^2\big(\mu_{\underline{s}}^{\widehat{\boldsymbol{X}}^{N}},\delta_0\big)\right)\right)\right)\Bigg]} \Bigg \} ds \notag\\
&\quad+2\int_0^te^{-\lambda s}\left\langle X_s^{i,N}-\widehat{X}_s^{i,N},\big(\sigma(X_s^{i,N},\mu_{s}^{\boldsymbol{X}^N})-\sigma_{\Delta}(\widehat{X}_{\underline{s}}^{i,N},\mu_{\underline{s}}^{\widehat{\boldsymbol{X}}^{N}})\big)dW_s^i\right\rangle \notag\\
&\quad+\int_0^t\int_{\mathbb{R}_0^d}e^{-\lambda s}\left(\vert X_{s-}^{i,N}-\widehat{X}_{s-}^{i,N}+\big(c(X_{s-}^{i,N},\mu_{s-}^{\boldsymbol{X}^N})-c_{\Delta}(\widehat{X}_{\underline{s}-}^{i,N},\mu_{\underline{s}-}^{\widehat{\boldsymbol{X}}^{N}})\big)z\vert^2-\vert X_{s-}^{i,N}-\widehat{X}_{s-}^{i,N}\vert^2\right)\widetilde{N}^i(d s, d z). \label{m1}
\end{align}
Now, using Lemma \ref{hq2} and Proposition \ref{dinh ly 2}, we have the following estimates
\begin{equation}\label{m2}\begin{split}
&\bE\left [\mathcal{W}_2^2\big(\mu_{s}^{\widehat{\boldsymbol{X}}^{N}},\mu_{\underline{s}}^{\widehat{\boldsymbol{X}}^{N}}\big)\right]\leq\dfrac{1}{N}\sum_{j=1}^{N}\bE\left [\left\vert \widehat{X}_{s}^{j,N}-\widehat{X}_{\underline{s}}^{j,N}\right\vert^2\right]=\bE\left [\left\vert \widehat{X}_{s}^{i,N}-\widehat{X}_{\underline{s}}^{i,N}\right\vert^2\right]\leq C \Delta,\\
&\bE \left[\mathcal{W}_2^2\big(\mu_{\underline{s}}^{\widehat{\boldsymbol{X}}^{N}},\delta_0\big)\right] =\bE\left [\dfrac{1}{N}\sum_{j=1}^{N}\left\vert \widehat{X}^{j,N}_{\underline{s}}\right\vert^2\right]= \dfrac{1}{N}\sum_{j=1}^{N}\bE\left [\left\vert \widehat{X}^{j,N}_{\underline{s}}\right\vert^2\right]=\bE\left [\left\vert \widehat{X}^{i,N}_{\underline{s}}\right\vert^2\right] \leq C,\\
&\bE\left [\mathcal{W}_2^4\big(\mu_{\underline{s}}^{\widehat{\boldsymbol{X}}^{N}},\delta_0\big)\right]=\bE\left [\left(\dfrac{1}{N}\sum_{j=1}^{N}\left\vert \widehat{X}^{j,N}_{\underline{s}}\right\vert^2\right)^2\right]\leq \dfrac{1}{N}\sum_{j=1}^{N}\bE\left [\left\vert \widehat{X}^{j,N}_{\underline{s}}\right\vert^4\right]=\bE\left [\left\vert \widehat{X}^{i,N}_{\underline{s}}\right\vert^4\right]\leq C,
\end{split}
\end{equation}
for any $i\in\{1,\ldots,N\}$ and some constant $C>0$.

Using the estimate
\begin{align*}
\mathcal{W}_2^p\left(\mu_{s}^{\widehat{\boldsymbol{X}}^{N}},\mu_{\underline{s}}^{\widehat{\boldsymbol{X}}^{N}}\right)\leq \dfrac{1}{N}\sum_{j=1}^{N}\left\vert \widehat{X}_{s}^{j,N}-\widehat{X}_{\underline{s}}^{j,N}\right\vert^p,
\end{align*}
valid for any $p\geq 2$ and Lemma \ref{hq2}, we have that for $\rho\in\{\ell, 2\ell\}$,
\begin{align}
&\bE\left [\left\vert \widehat{X}_{s}^{i,N}-\widehat{X}_{\underline{s}}^{i,N}\right\vert^{\rho}\mathcal{W}_2^2\left(\mu_{s}^{\widehat{\boldsymbol{X}}^{N}},\mu_{\underline{s}}^{\widehat{\boldsymbol{X}}^{N}}\right)\right]\leq \dfrac{1}{N}\sum_{j=1}^{N}\bE\left [\left\vert \widehat{X}_{s}^{i,N}-\widehat{X}_{\underline{s}}^{i,N}\right\vert^{\rho}\left\vert \widehat{X}_{s}^{j,N}-\widehat{X}_{\underline{s}}^{j,N}\right\vert^2\right]  \notag\\
&=\dfrac{1}{N}\bE\left [\left\vert \widehat{X}_{s}^{i,N}-\widehat{X}_{\underline{s}}^{i,N}\right\vert^{\rho+2}\right]+\dfrac{1}{N}\sum_{j\neq i} \bE\left [\left\vert \widehat{X}_{s}^{i,N}-\widehat{X}_{\underline{s}}^{i,N}\right\vert^{\rho}\right]\bE\left [\left\vert \widehat{X}_{s}^{j,N}-\widehat{X}_{\underline{s}}^{j,N}\right\vert^2\right]  \notag\\
&\leq C \Delta.  \label{m3}
\end{align}
Next, using Lemma \ref{hq2}, Proposition \ref{dinh ly 2} and the fact that $p_0\geq 4\ell+6$, we get that
\begin{equation}  \label{m4}\begin{split}
&\bE \left[\left\vert \widehat{X}_{s}^{i,N}-\widehat{X}_{\underline{s}}^{i,N}\right\vert^q\right] \leq C\Delta,\;\; \bE \left[\left\vert \widehat{X}_{s}^{i,N}-\widehat{X}_{\underline{s}}^{i,N}\right\vert^{\ell}\right] \leq \left(\bE \left[\left\vert \widehat{X}_{s}^{i,N}-\widehat{X}_{\underline{s}}^{i,N}\right\vert^{2\ell}\right] \right)^{1/2}\leq C\sqrt{\Delta} \leq C,\\
&\bE\left[\left\vert \widehat{X}_{\underline{s}}^{i,N}\right\vert^{\iota}\right]\leq C,
\end{split}
\end{equation}
for $q \in\{2; \ell+2; 2\ell; 2\ell+2\}$, $\iota \in\{\ell; 2\ell; 2\ell+4; 2\ell+6; 4\ell+6\}$, and
\begin{align}
\bE \left[\left\vert \widehat{X}_{\underline{s}}^{i,N}\right\vert^{\kappa}\left\vert \widehat{X}_{s}^{i,N}-\widehat{X}_{\underline{s}}^{i,N}\right\vert^2\right]&=\bE \left[\bE \left[\left\vert \widehat{X}_{\underline{s}}^{i,N}\right\vert^{\kappa}\left\vert \widehat{X}_{s}^{i,N}-\widehat{X}_{\underline{s}}^{i,N}\right\vert^2\big| \mathcal{F}_{\underline{s}}\right]\right] \notag\\
&=\bE \left[\left\vert \widehat{X}_{\underline{s}}^{i,N}\right\vert^{\kappa}\bE \left[\left\vert \widehat{X}_{s}^{i,N}-\widehat{X}_{\underline{s}}^{i,N}\right\vert^2\big| \mathcal{F}_{\underline{s}}\right]\right] \notag\\
&\leq C\Delta \bE \left[\left\vert \widehat{X}_{\underline{s}}^{i,N}\right\vert^{\kappa}\right]  \notag\\
&\leq C\Delta, \label{m5}
\end{align}
for $\kappa \in\{\ell; 2\ell\}$ and some constant $C>0$. 

Furthermore, using Lemma \ref{hq2} and Proposition \ref{dinh ly 2}, we obtain that for $\varrho\in\{\ell, 2\ell\}$,
\begin{align}
&\bE\left [\left\vert \widehat{X}_{\underline{s}}^{i,N}\right\vert^{\varrho}\mathcal{W}_2^2\left(\mu_{s}^{\widehat{\boldsymbol{X}}^{N}},\mu_{\underline{s}}^{\widehat{\boldsymbol{X}}^{N}}\right)\right]\leq \dfrac{1}{N}\sum_{j=1}^{N}\bE\left [\left\vert \widehat{X}_{\underline{s}}^{i,N}\right\vert^{\varrho}\left\vert \widehat{X}_{s}^{j,N}-\widehat{X}_{\underline{s}}^{j,N}\right\vert^2\right]  \notag\\
&=\dfrac{1}{N}\bE\left [\left\vert \widehat{X}_{\underline{s}}^{i,N}\right\vert^{\varrho}\left\vert \widehat{X}_{s}^{i,N}-\widehat{X}_{\underline{s}}^{i,N}\right\vert^2\right]+\dfrac{1}{N} \sum_{j\neq i} \bE\left [\left\vert \widehat{X}_{\underline{s}}^{i,N}\right\vert^{\varrho}\right]\bE\left [\left\vert \widehat{X}_{s}^{j,N}-\widehat{X}_{\underline{s}}^{j,N}\right\vert^2\right]  \notag\\
&\leq C\Delta.  \label{m6}
\end{align}
Using Proposition \ref{dinh ly 2}, we obtain that for $\vartheta\in\{2\ell+2; 4\ell+2\}$,
\begin{align}
&\bE\left[\left\vert \widehat{X}_{\underline{s}}^{i,N}\right\vert^{\vartheta}\mathcal{W}_2^4\big(\mu_{\underline{s}}^{\widehat{\boldsymbol{X}}^{N}},\delta_0\big)\right]\leq \dfrac{1}{N}\sum_{j=1}^{N}\bE\left [\left\vert\widehat{X}_{\underline{s}}^{i,N}\right\vert^{\vartheta}\left\vert \widehat{X}_{\underline{s}}^{j,N}\right\vert^4\right]  \notag\\
&=\dfrac{1}{N}\bE\left [\left\vert\widehat{X}_{\underline{s}}^{i,N}\right\vert^{\vartheta+4}\right]+\dfrac{1}{N}\sum_{j\neq i}\bE\left [\left\vert\widehat{X}_{\underline{s}}^{i,N}\right\vert^{\vartheta}\right]\bE\left [\left\vert \widehat{X}_{\underline{s}}^{j,N}\right\vert^4\right]  \notag\\
&\leq C.  \label{m7}
\end{align}
Proceeding similarly to the above, we get that
\begin{align}
&\bE\left[\left\vert \widehat{X}_{\underline{s}}^{i,N}\right\vert^{4\ell}\mathcal{W}_2^2\big(\mu_{\underline{s}}^{\widehat{\boldsymbol{X}}^{N}},\delta_0\big)\mathcal{W}_2^4\big(\mu_{\underline{s}}^{\widehat{\boldsymbol{X}}^{N}},\delta_0\big)\right]=\bE\left[\left\vert \widehat{X}_{\underline{s}}^{i,N}\right\vert^{4\ell}\mathcal{W}_2^6\big(\mu_{\underline{s}}^{\widehat{\boldsymbol{X}}^{N}},\delta_0\big)\right]\leq  C. \label{m8}
\end{align}
Consequently, plugging \eqref{m2}-\eqref{m8} into \eqref{m1}, taking the expectation on both sides and choosing $\lambda=\varepsilon_1+L_1+L_2$,  we get that for any $t \in[0,T]$,
\begin{align*}
\bE\left[e^{- (\varepsilon_1+L_1+L_2) t}\left\vert X_t^{i,N}-\widehat{X}_t^{i,N}\right\vert^2\right]\leq C \Delta \int_{0}^{t} e^{-(\varepsilon_1+L_1+L_2) s}ds. 
\end{align*}
This implies that
\begin{equation}\label{CT}
\max_{i\in\{1,\ldots,N\}}\underset{t\in[0,T]}{\sup} \bE \left[  \vert X_t^{i,N}-\widehat{X}_t^{i,N}\vert^2 \right] \le  C_T \Delta,
\end{equation}
for some positive constant $C_T=C(x_0,L,L_1,L_2,L_4,\widetilde{\gamma}_1,\widetilde{\gamma}_2,\widetilde{\eta},\widetilde{L_3}, T)$,  which shows \eqref{EXmu-X1}.

Next, for any stopping time $\tau \leq T$, using again \eqref{m1} with $t=\tau$, $\lambda = \varepsilon_1 + L_1$, taking the expectation of the above inequality on both sides and using again the estimates \eqref{m2}--\eqref{m8} and \eqref{CT}, we obtain that
	\begin{align*}
	&\bE \left[e^{-( \varepsilon_1 + L_1) \tau}\vert X_{\tau}^{i,N}-\widehat{X}_{\tau}^{i,N}\vert^2 \right]\\
	&\le \int_{0}^{T} e^{-( \varepsilon_1 + L_1) s} \bE\Bigg\{  L_2\mathcal{W}_2^2\left(\mu_{s}^{\boldsymbol{X}^{N}},\mu_{s}^{\widehat{\boldsymbol{X}}^{N}}\right)  \notag\\
	&\quad + \textcolor{blue}{ \dfrac{6}{\varepsilon_1}L^2\Bigg[ \left\vert \widehat{X}_{s}^{i,N}-\widehat{X}_{\underline{s}}^{i,N}\right\vert^2 +\mathcal{W}_2^2\left(\mu_{s}^{\widehat{\boldsymbol{X}}^{N}},\mu_{\underline{s}}^{\widehat{\boldsymbol{X}}^{N}}\right)+2^{2\ell}\left\vert \widehat{X}_{s}^{i,N}-\widehat{X}_{\underline{s}}^{i,N}\right\vert^{2\ell+2} +2^{2\ell}\left\vert \widehat{X}_{s}^{i,N}-\widehat{X}_{\underline{s}}^{i,N}\right\vert^{2\ell} } \notag\\
	&\quad  \textcolor{blue}{  \times\mathcal{W}_2^2\left(\mu_{s}^{\widehat{\boldsymbol{X}}^{N}},\mu_{\underline{s}}^{\widehat{\boldsymbol{X}}^{N}}\right)+(2^{2\ell}+1)\left\vert \widehat{X}_{\underline{s}}^{i,N}\right\vert^{2\ell}\left\vert \widehat{X}_{s}^{i,N}-\widehat{X}_{\underline{s}}^{i,N}\right\vert^2  +(2^{2\ell}+1)\left\vert \widehat{X}_{\underline{s}}^{i,N}\right\vert^{2\ell}\mathcal{W}_2^2\left(\mu_{s}^{\widehat{\boldsymbol{X}}^{N}},\mu_{\underline{s}}^{\widehat{\boldsymbol{X}}^{N}}\right)\Bigg]} \notag\\
	&\quad + \textcolor{magenta}{  2\big(1+\frac{1}{\varepsilon}\big)\Bigg[ \dfrac{L\widetilde{L}}{1+\varepsilon}\bigg(\left\vert \widehat{X}_{s}^{i,N}-\widehat{X}_{\underline{s}}^{i,N}\right\vert^2+\mathcal{W}_2^2\left(\mu_{s}^{\widehat{\boldsymbol{X}}^{N}},\mu_{\underline{s}}^{\widehat{\boldsymbol{X}}^{N}}\right)  +2^{\ell}\left\vert \widehat{X}_{s}^{i,N}-\widehat{X}_{\underline{s}}^{i,N}\right\vert^{\ell+2}+2^{\ell}\left\vert \widehat{X}_{s}^{i,N}-\widehat{X}_{\underline{s}}^{i,N}\right\vert^{\ell}}\notag\\
	&\quad  \textcolor{magenta}{  \times\mathcal{W}_2^2\left(\mu_{s}^{\widehat{\boldsymbol{X}}^{N}},\mu_{\underline{s}}^{\widehat{\boldsymbol{X}}^{N}}\right)+(2^{\ell}+1)\left\vert \widehat{X}_{\underline{s}}^{i,N}\right\vert^{\ell}\left\vert \widehat{X}_{s}^{i,N}-\widehat{X}_{\underline{s}}^{i,N}\right\vert^2  +(2^{\ell}+1)\big\vert \widehat{X}_{\underline{s}}^{i,N}\big\vert^{\ell}\mathcal{W}_2^2\left(\mu_{s}^{\widehat{\boldsymbol{X}}^{N}},\mu_{\underline{s}}^{\widehat{\boldsymbol{X}}^{N}}\right)\bigg) }\notag\\
	&\quad  \textcolor{magenta}{  +4L_4^2 \Delta\bigg( 16\Big(\dfrac{L\widetilde{L}}{1+\varepsilon}\Big)^2\Big(1+\vert \widehat{X}_{\underline{s}}^{i,N}\vert^{2\ell}\Big)\Big(\vert \widehat{X}_{\underline{s}}^{i,N}\vert^4+\mathcal{W}_2^4\big(\mu_{\underline{s}}^{\widehat{\boldsymbol{X}}^{N}},\delta_0\big)\Big) +4\left\vert \sigma(0,\delta_0)\right\vert^4\bigg) \Big(1+\vert \widehat{X}_{\underline{s}}^{i,N}\vert^2\Big)\Bigg] } \notag\\
	&\quad +   \textcolor{red}{  2\left(1+\frac{1}{\varepsilon}\right)\Bigg[ \dfrac{L\widetilde{L}}{1+\varepsilon}\bigg(\left\vert \widehat{X}_{s}^{i,N}-\widehat{X}_{\underline{s}}^{i,N}\right\vert^2+\mathcal{W}_2^2\left(\mu_{s}^{\widehat{\boldsymbol{X}}^{N}},\mu_{\underline{s}}^{\widehat{\boldsymbol{X}}^{N}}\right)  +2^{\ell}\left\vert \widehat{X}_{s}^{i,N}-\widehat{X}_{\underline{s}}^{i,N}\right\vert^{\ell+2}+2^{\ell}\left\vert \widehat{X}_{s}^{i,N}-\widehat{X}_{\underline{s}}^{i,N}\right\vert^{\ell}} \notag\\
	&\quad  \textcolor{red}{  \times\mathcal{W}_2^2\left(\mu_{s}^{\widehat{\boldsymbol{X}}^{N}},\mu_{\underline{s}}^{\widehat{\boldsymbol{X}}^{N}}\right)+(2^{\ell}+1)\left\vert \widehat{X}_{\underline{s}}^{i,N}\right\vert^{\ell}\left\vert \widehat{X}_{s}^{i,N}-\widehat{X}_{\underline{s}}^{i,N}\right\vert^2  +(2^{\ell}+1)\big\vert \widehat{X}_{\underline{s}}^{i,N}\big\vert^{\ell}\mathcal{W}_2^2\left(\mu_{s}^{\widehat{\boldsymbol{X}}^{N}},\mu_{\underline{s}}^{\widehat{\boldsymbol{X}}^{N}}\right)\bigg) } \notag\\
	&\quad  \textcolor{red}{ +\dfrac{6L_4^2 \Delta}{\int_{\bR_{0}^d}|z|^{2} \nu(d z)}\bigg( 16\Big(\dfrac{L\widetilde{L}}{1+\varepsilon}\Big)^2\Big(1+\vert \widehat{X}_{\underline{s}}^{i,N}\vert^{2\ell}\Big)\Big(\vert \widehat{X}_{\underline{s}}^{i,N}\vert^4+\mathcal{W}_2^4\big(\mu_{\underline{s}}^{\widehat{\boldsymbol{X}}^{N}},\delta_0\big)\Big)  +4\left\vert c(0,\delta_0)\right\vert^4\Big(\int_{\bR_{0}^d}|z|^{2} \nu(d z)\Big)^2\bigg)} \notag\\
	&\quad  \textcolor{red}{ \times \left(1+\vert \widehat{X}_{\underline{s}}^{i,N}\vert^2+2\left(\vert b(0,\delta_0)\vert^2+4L^2\left(1+\vert \widehat{X}_{\underline{s}}^{i,N}\vert^{2\ell}\right)\left(\vert \widehat{X}_{\underline{s}}^{i,N}\vert^{2}+\mathcal{W}_2^2\big(\mu_{\underline{s}}^{\widehat{\boldsymbol{X}}^{N}},\delta_0\big)\right)\right)\right)\Bigg] } \Bigg\} ds\\
	&\le  \widetilde{C}_T \Delta,
	\end{align*}
	for some constant $\widetilde{C}_T>0$.

	Therefore, due to Proposition IV.4.7 in \cite{RY}, we get that for any $p\in(0,2)$,
	\begin{align*}
	\max_{i\in\{1,\ldots,N\}}\bE \left[ \underset{t\in[0,T]}{\sup} e^{-\frac{p ( \varepsilon_1 + L_1) t}{2}}\vert X_{t}^{i,N}-\widehat{X}_{t}^{i,N}\vert^p \right] \leq \left(\dfrac{2-p/2}{1-p/2}\right)  (\widetilde{C}_T\Delta)^{p/2},
	\end{align*}
	which,  combined with the fact that $e^{-\frac{p ( \varepsilon_1 + L_1) t}{2}}\geq e^{-\frac{p \vert  \varepsilon_1 + L_1 \vert T}{2}}$,  concludes \eqref{EXmu-X0}.

Moreover, when $L_1+L_2<0$, we can always choose $\varepsilon_1>0$ such that $L_1 + L_2 + \varepsilon_1<0$. Consequently, when $L_1+L_2<0$ and $\widetilde{\gamma}<0$, the constant $C_T$ in \eqref{CT} now does not depend on $T$. Therefore, we have shown \eqref{EXmu-X2}, which finishes the desired proof.
\end{proof} 

We now state our main result on strong convergence in both finite and infinite time intervals of the tamed-adaptive Euler-Maruyama scheme for multidimensional McKean-Vlasov SDEs driven by L\'evy processes. 
\begin{Thm}

Assume  that Conditions \textbf{A1}, \textbf{A3--A7},  \textbf{T2--T6} hold, and $p_0\geq 4\ell+6$,   $N \geq \left( \frac{\max\{3\widetilde{L_3},1\}}{2\upsilon}\right)^2$. 
Assume further that there exists a constant $\varepsilon>0$ such that \textbf{A2} holds for $\kappa_1=\kappa_2=1+\varepsilon$, $L_1\in\mathbb{R}$, $L_2\geq 0$. Then for any $T>0$, there exists a positive constant  $C_T=C(x_0,L,L_1,L_2,L_3,L_4,\widetilde{\gamma}_1,\widetilde{\gamma}_2,\widetilde{\eta},\widetilde{L_3}, T)$ such that
\begin{equation} \label{main1}
\max_{i\in\{1,\ldots,N\}}\underset{t\in[0,T]}{\sup} \bE \left[  \left\vert X_t^{i}-\widehat{X}_t^{i,N}\right\vert^2 \right] \le  C_T \left(\Delta + \varphi(N)\right),
\end{equation}
for any $N\in\mathbb{N}$, where the constant $C_T>0$ does not depend on $N$.

Moreover, assume that $\gamma= \gamma_1 + \gamma_2 <0$, $\widetilde{\gamma}= \widetilde{\gamma}_1+ \widetilde{\gamma}_2<0$ and $L_1+L_2<0$. Then,  there exists a positive constant  \\$C''=C''(x_0,L,L_1,L_2,L_3,L_4,\widetilde{\gamma}_1,\widetilde{\gamma}_2,\widetilde{\eta},\widetilde{L_3})$ which does not depend on $T$  such that
\begin{equation} \label{main2}
\max_{i\in\{1,\ldots,N\}}\underset{t \ge 0}{\sup}\ \bE \left[ \left \vert X_t^{i}-\widehat{X}_t^{i,N}\right\vert^2 \right] \le  C'' \left(\Delta + \varphi(N)\right).
\end{equation} 
\end{Thm}
\begin{proof} 
As a consequence of Proposition \ref{PoC} and Theorem \ref{dinh ly 6}, the proof is straightforward. Thus, we omit it.
\end{proof} 	

\section{Numerical experiment} \label{sec:nume}

In this section, we consider the rate of convergence of the tamed-adaptive Euler-Maruyama scheme \eqref{EM1},\eqref{chooseh}, \eqref{smDtcDt}  in Theorem \ref{dinh ly 6} for fixed large values of  $N$.
We consider the following L\'evy-driven McKean-Vlasov stochastic differential equation
	\begin{equation} \label{ptvd}
	d X_t=\left(-1-3 \left(X_t+\mathbb{E}\left[X_t\right]\right)- X_t |X_t|^{0.3}\right)dt + 0.2\left(1+|X_t|^{1.1}+\mathbb{E}\left[X_t\right]\right) d W_t+0.2\left(X_{t-}+\mathbb{E}\left[X_{t-} \right]\right) d Z_t.
	\end{equation} 
	That is,
	\begin{align*}
		&b(x, \mu)=-1-3\left(x+\int_{\mathbb{R}} z \mu(d z) \right)- x|x|^{0.3},\\
		&\sigma(x, \mu)=0.2\left(1+|x|^{1.1}+\int_{\mathbb{R}} z \mu(d z)\right),\qquad
		c(x, \mu)=0.2\left(x+\int_{\mathbb{R}} z \mu(d z)\right),
		\end{align*}
	for all $x \in \bR$ and $\mu\in \mathcal{P}_2(\mathbb{R})$.	
	Here we suppose that $Z=(Z_t)_{t\geq 0}$ is a bilateral Gamma process whose scale parameter is 5  and shape 	
	parameter is 1. It is straightforward to verify  that these coefficients satisfy Conditions \textbf{A1--A7} and \textbf{T2--T6}. 
	
	In the following, we will implement the tamed-adaptive Euler approximation scheme \eqref{EM1}-\eqref{chooseh} with 
$N = 500, \ x_0 = 1, \ell = 0.3, p_0 = 8$, and  $T = 10.$
Since the exact solution of equation \eqref{ptvd} is unknown, we will derive the rate of convergence of the tamed-adaptive Euler approximation scheme \eqref{EM1}-\eqref{chooseh} in an indirect way as in \cite{KLN22, KLNT22}. 
We consider the mean squared difference of $\widehat{X}$ on two consecutive levels as follows: 
$$\textrm{MSE}(\ellf, T) = \frac{1}{M} \sum_{k=1}^M |\widehat{X}^{(\ellf,k)}_T - \widehat{X}^{(\ellf+1,k)}_T|^2,$$
where for each $\ellf \geq 1$,  $(\widehat{X}^{(\ellf,k)})_{1 \leq k \leq M}$ is a sequence of independent copies of $\widehat{X}^{(\ellf)}$ defined by equations \eqref{EM1}, \eqref{chooseh}, and \eqref{smDtcDt} with $\Delta = 2^{-\ellf}$. Here $\widehat{X}^{(\ellf,k)}_T$ and  $\widehat{X}^{(\ellf+1,k)}_T$ must be simulated to the same Brownian motions and bilateral Gamma processes (See  Algorithm 1 in \cite{FG}). 

 It is clear that  $\widehat{X}^{(\ellf)}$ converges at some rate of order $\beta \in (0,+\infty)$ in $L^2$-norm  iff $2^{\beta \ellf} \|\widehat{X}^{(\ellf+1)}_T - \widehat{X}^{(\ellf)}_T\|_{L^2} = O(1)$,
which implies that  
$\log_2 \textrm{MSE}(\ellf,T) = -2\beta \ellf + C + o(1),  $
for some constant $C \in \mathbb{R}.$ Thus we can use the regression method to estimate the rate $\beta$.  Figure \ref{fig1} shows the values of $\log_2 MSE(\ellf, T)$ plotted against $\ellf \in \{1, 2,..., 6\}.$ We see that $\beta \approx 0.5$.

\begin{figure}
\begin{center}
\includegraphics[height = 3cm]{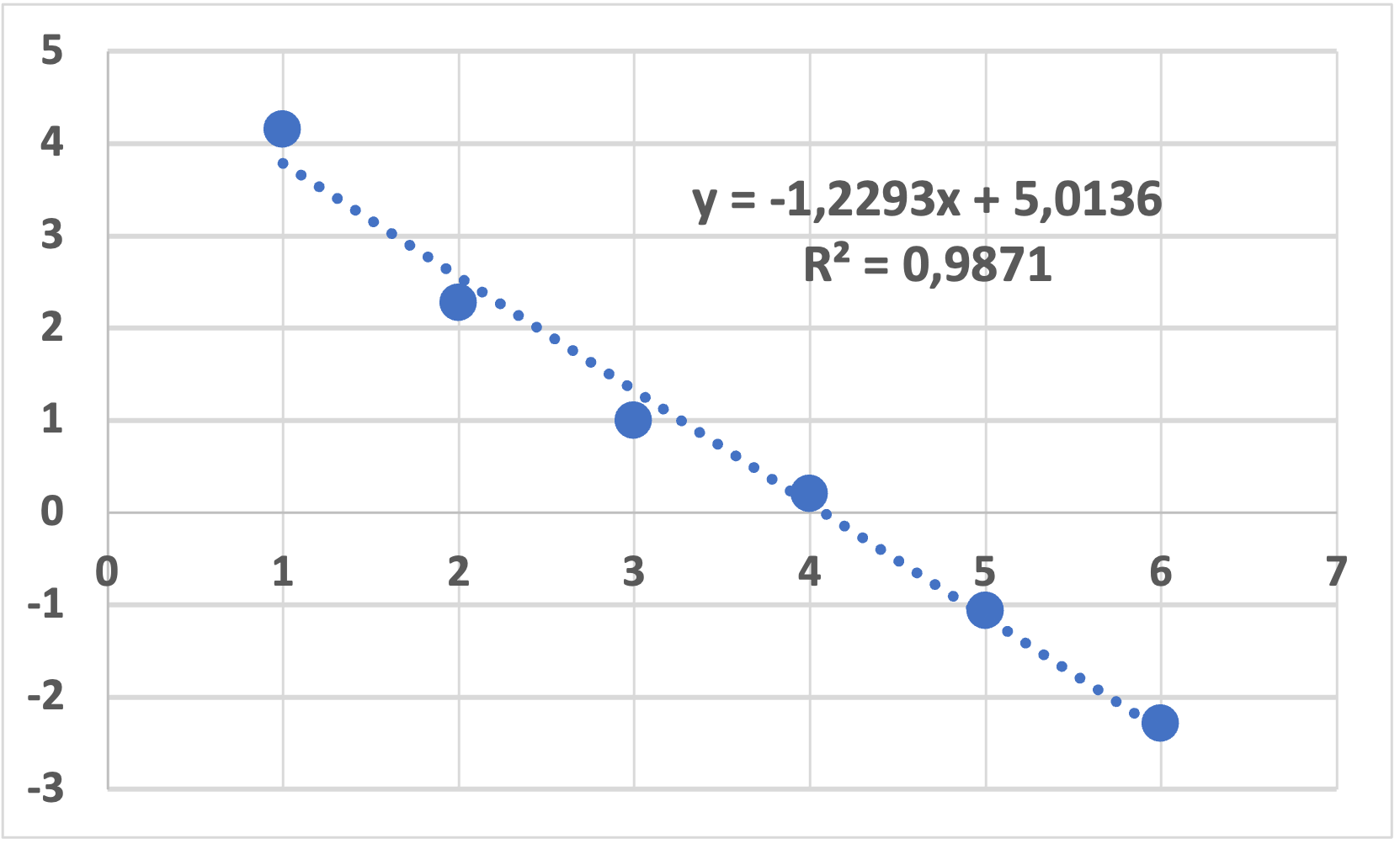}
\caption{Error $\log_2 \textrm{MSE}(\ellf, 10)$ plotted against $\ellf=1,...,6$.}
\label{fig1} 
\end{center} 
\end{figure}

\section*{Acknowledgements}
 This research is funded by the Vietnam National Foundation for Science and Technology
Development (NAFOSTED) under the grant number  101.03-2021.36. 
A part of this work was done when the authors visited Vietnam Institute for Advanced Study in Mathematics (VIASM) in 2023. The authors also would like to thank VIASM for their kind hospitality during their research visit.

\end{document}